\documentclass[11pt]{article}
\usepackage[a4paper,margin=25mm]{geometry}
\usepackage[utf8]{inputenc}
\usepackage[T1]{fontenc}
\usepackage[english]{babel}
\usepackage{csquotes}
\usepackage{amsmath,amssymb,mathtools,amsthm,bbm,amsbsy,stackrel}
\usepackage{enumerate}
\usepackage{hyperref}
\usepackage{color,xcolor}

\newtheorem{theorem}{Theorem}
\newtheorem{proposition}[theorem]{Proposition}
\newtheorem{corollary}[theorem]{Corollary}

\newtheorem{lemma}[theorem]{Lemma}

\numberwithin{theorem}{section}
\numberwithin{equation}{section}

\theoremstyle{definition}

\newtheorem{assumption}{\ensuremath{\boldsymbol{\mathcal{H}}}\!\!}
\newtheorem{remark}[theorem]{Remark}
\newtheorem{example}[theorem]{Example}

\DeclareMathOperator{\Pas}{\mathbb{P}\text{-as}}

\bibliography{main.bib}

\hypersetup{colorlinks=true,urlcolor=black,linkcolor=black,citecolor=black}

\title{Stochastic Gradient Descent Revisited}
\author{Azar~Louzi\thanks{Universit\'e Paris Cit\'e, CNRS. Laboratoire de Probabilit\'es, Statistique et Mod\'elisation. azar.louzi@lpsm.paris}}
\date{\today}

\begin{document}

\maketitle

\begin{abstract}
Stochastic gradient descent (SGD) has been a go-to algorithm for nonconvex stochastic optimization problems arising in machine learning.
Its theory however often requires a strong framework to guarantee convergence properties.
We hereby present a full scope convergence study of biased nonconvex SGD, including weak convergence, function-value convergence and global convergence, and also provide subsequent convergence rates and complexities, all under relatively mild conditions in comparison with literature.
\\

\\

\end{abstract}

\section*{Introduction}

The advent of artificial intelligence (AI) has been rendered possible by the spectacular acceleration of computing chip capacity over the last few decades, and has driven a technological revolution that has not spared any aspect of life, including healthcare, supply chain management, social media, etc.
AI describes a set of machine learning methods that abandon any form of structural representation of data and look instead into uncovering data patterns to produce probabilistic relationships between input and output quantities of interest.

While it has significantly improved people's standards of living, AI has nevertheless engendered many operational risks (e.g.~by producing undesirable or unexpected outcomes) as well as systemic risks (e.g.~the \emph{``Flash Crash''}, whereby a blue-chip company's share price suddenly plummeted and bounced back in the span of minutes~\cite{KL13}).
To better manage, prevent and mitigate such risks, some level of mathematical insight must be brought in to shed light onto the inner workings of AI, in order to allow practitioners and regulators alike to act upon it in order to increase its efficiency and curb its shortcomings.

SGD is the engine of AI, making it a natural stepping stone toward mathematically explaining AI. Indeed, to capture their intricacies, machine learning problems are often modeled using wide and highly parametrized neural networks~\cite{GBC16}, which are then solved using SGD or an adaptive variant thereof, namely Adagrad, Adadelta, RMSProp, Adamax or Adam~\cite{Rud17}.
To approximate a stationary point of a given loss landscape (also referred to as objective or cost function~\cite{LZB22,AL24,AMA05}), SGD recursively spawns a trajectory of iterates by factoring in, at each step, a stochastic gradient modulated by a positive learning rate.
Whereas classical SGD literature provides convergence guarantees and convergence rates within a (strongly) convex framework~\cite{Duf96,BV04,RM51}, machine learning models are often highly nonconvex and require new SGD frameworks to better understand and parametrize them. An adequate parametrization can drastically cut down model training time by eliminating some fine-tuning overhead.
\\

Recent literature attempting to comprehend nonconvex SGD from a theoretical standpoint presents convergence guarantees and convergence rates within various settings.
The frequently studied SGD convergence modes include weak convergence, function-value convergence and global convergence.
Consider a differentiable loss landscape $F$ to minimize and related iterates $(\theta_n)_{n\geq0}$ produced by SGD. By weak convergence, we refer to the almost sure vanishing of $(\nabla F(\theta_n))_{n\geq0}$~\cite{AMA05}, by function-value convergence, the almost sure convergence of $(F(\theta_n))_{n\geq0}$~\cite{Dav+20}, and by global convergence, the almost sure convergence of the iterates themselves to a stationary point~\cite{LP15,Let21}. The speed at which each convergence mode occurs, when quantifiable, helps infer optimal complexities to achieve some prescribed tolerance.

Following the salutary work~\cite{Loj84} on gradient flow convergence for analytical functions, \cite{AMA05} pushed for a new framework for studying nonconvex gradient descent (GD), termed the {\L}ojasiewicz condition.
It has found a large traction in the optimization community as it describes a large class of nonconvex loss landscapes.
Under this condition, for a stationary point $\theta_\star$ of a loss landscape $F$ ($\nabla F(\theta_\star)=0$), one has $\zeta\|\nabla F(\theta)\|\geq|F(\theta)-F(\theta_\star)|^\beta$ on an open neighborhood of $\theta_\star$, for some constant $\zeta>0$ and exponent $\beta\in(0,1)$.
Initially introduced to prove single-point convergence of GD-type algorithms, it has also found success in deriving GD convergence rates~\cite{AB09,FGP15}.
Other variants of the original inequality have been explored as well.
When $\beta=\frac12$, $F$ is said to be locally Polyak-{\L}ojasiewicz (P{\L})~\cite{Pol63}.
When the property is satisfied at every point in space, $F$ is dubbed globally {\L}ojasiewicz~\cite{GG24}.
A variant thereof, known as the Kurdyka-{\L}ojasiewicz (K{\L}) condition, stipulates that $\varphi'(F(\theta)-F(\theta_\star))\|\nabla F(\theta)\|\geq1$ on an open neighborhood of a stationary point $\theta_\star$, where $\varphi$ is a desingularization function that is increasing, concave and differentiable, with $\varphi(0)=0$.
This property is proven to hold for continuously differentiable subanalytic functions~\cite{Kur98}.
Although it is considered to be a slight generalization of the {\L}ojasiewicz condition~\cite{DK24}, reverting to the {\L}ojasiewicz condition, by considering e.g.~$\varphi(x)=x^{1-\beta}$~\cite{FGP15}, is often necessary to derive convergence rates.
\\

\cite{Lei+20,DK24,CFR23,Dav+20} provide weak and function-value convergence guarantees for SGD under various assumptions on the loss landscape.
\cite{Lei+20} proves function-value convergence by lifting the SGD formalism into a reproducing kernel Hilbert space, where the loss landscape is assumed to be H\"older-differentiable.
\cite{DK24} studies biased SGD for locally H\"older-differentiable loss landscapes, presupposing the SGD iterates almost surely bounded.
\cite{Dav+20} focuses on locally Lipschitz functions with possible singularities, but relies on the presumption that SGD iterates are almost surely bounded.

Recent developments on function-value convergence rate build upon varying conditions~\cite{Lei+20,AL24,DK24,Wei+24}.
\cite{Lei+20} provides convergence speeds in expectation under a global P{\L} condition.
\cite{AL24} extends the framework of~\cite{Cha22} to SGD by supposing a local P{\L} condition, then derives an exponential convergence rate on the condition that the iterates remain close to a local minimum.
In a local {\L}ojasiewicz setting, \cite{DK24} elicits a convergence rate in expectation conditionally to suitable controls on the stochastic gradients.
In an identical setting, \cite{Wei+24} provides a convergence rate in high probability, conditionally to the iterates remaining confined around some local minimum.

Global convergence is investigated in~\cite{DK24,CFR23,QMM24} from different perspectives.
\cite{DK24} proves almost sure iterate convergence for locally Lipschitz-differentiable loss landscapes under a local {\L}ojasiewicz condition, by imposing deterministic controls on the behavior of the stochastic gradients.
\cite{CFR23} assumes a local K{\L} condition on a compact set coupled with a strong growth property.
\cite{QMM24} studies momentum \eqref{eq:theta:n} in a local {\L}ojasiewicz setting, by assuming bounded conditional variances of the stochastic gradients.

\cite{Liu+23} derives high probability convergence rates for SGD iterates when initialized close enough to a global minimum, under an aiming and an interpolation conditions.
\cite{GP23} obtains $L^{2p}(\mathbb{P})$-convergence rates for SGD iterates, $p\geq1$, for globally {\L}ojasiewicz loss landscapes with {\L}ojasiewicz exponent $\beta\leq\frac12$, by supposing uniformly bounded conditional polynomial-exponential moments on the stochastic gradient noises.
\\

In this work,
\begin{itemize}
    \item
Thorem~\ref{thm:F:cv:gradF:0} obtains weak and function-value convergences for biased SGD under a relaxation of the smoothness framework and typical learning rate behavior, without requiring iterate boundedness;
    \item
Theorem~\ref{thm:cv} characterizes the topology of the iterates' accumulation points under a plain coercivity argument;
    \item
Theorem~\ref{thm:F-F*:special:cases} derives high probability function-value convergence speeds within both a local and a global {\L}ojasiewicz frameworks, without imposing deterministic controls on the stochastic gradients or confining the SGD trajectories around some local minimum;
    \item
Corollary~\ref{crl:cost} quantifies subsequent complexities and recovers optimal iteration amounts;
    \item
Theorem~\ref{thm:iterates} and Corollaries~\ref{crl:iterates:speed} and~\ref{crl:iter:cost} deduce single-point limit convergence guarantees and rates for biased SGD within the same setting as for function-value convergence.
\end{itemize}

This paper is organized as follows.
Section~\ref{sec:SGD} establishes a baseline SGD framework and enunciates our standing convergence results, followed upon with related comments and discussions.
The technical proofs of our main results are delineated in Sections~\ref{sec:cv}--\ref{sec:global}.

\section{Stochastic Gradient Descent}
\label{sec:SGD}

We consider a probability space $(\Omega,\mathcal{F},\mathbb{P})$ that is rich enough to support all the ensuing random variables.
In a given vector space $\mathbb{R}^p$, $p\geq1$, $\|\,\cdot\,\|$ denotes the $L^2$ vector norm. Given a measurable random variable $Z\colon(\Omega,\mathcal{F})\to(\mathbb{R}^p,\mathcal{B}(\mathbb{R}^p))$, $p\geq1$, we write $Z\in L^q(\mathbb{P})$, $q>0$, to signify that $\mathbb{E}[\|Z\|^q]<\infty$.

We are interested in the stochastic optimization problem
\begin{equation}\label{opt:prob}
\min_{\theta\in\mathbb{R}^m}{\big\{F(\theta)=\mathbb{E}[f(X,\theta)]\big\}},
\tag{SO}
\end{equation}
where $X$ is an $\mathbb{R}^d$-valued random variable, $f\colon\mathbb{R}^d\times\mathbb{R}^m\to\mathbb{R}$ is a function such that $f(X,\theta)\in L^1(\mathbb{P})$, $\theta\in\mathbb{R}^m$, and $d,m\geq1$.
$f$ is often referred to as the loss or utility function.

\begin{assumption}\label{asp:f}
    $F$ is lower bounded, i.e.~$\inf F>-\infty$, differentiable, and its derivative $\nabla F$ is $(L,\alpha)$-H\"older continuous, $L>0$, $\alpha\in(0,1]$, i.e.
\begin{equation*}
\|\nabla F(\theta')-\nabla F(\theta)\|\leq L\|\theta'-\theta\|^\alpha,
\quad\theta,\theta'\in\mathbb{R}^m.
\end{equation*}
\end{assumption}

$\mathcal{H}$\ref{asp:f} relaxes the smoothness condition arising when $\alpha=1$.
In spite of the continuous differentiability of $F$, the function $\theta\in\mathbb{R}^m\mapsto f(x,\theta)$, $x\in\mathbb{R}^d$, does not need to be differentiable everywhere~\cite{Tal01}.

\begin{example}\label{exp:quantile}
Take the $\mu$-quantile regression problem~\cite{RU00}: $\min_\theta\{F(\theta)=\mathbb{E}[f(X,\theta)]\}$, where $f(x,\theta)=\theta+\frac{(x-\theta)^+}{1-\mu}$, $x,\theta\in\mathbb{R}$, with $\mu\in(0,1)$.
Let $x\in\mathbb{R}$. $\theta\in\mathbb{R}\mapsto f(x,\theta)$ is differentiable for $\theta\neq x$ and $\partial_\theta f(x,\theta)=1-\frac{\mathbbm1_{x>\theta}}{1-\mu}$, $\theta\in\mathbb{R}$, which is discontinuous at $x$.
Under suitable conditions, notably the continuity of the cdf $F_X(\theta)\coloneqq\mathbb{P}(X\leq\theta)$, $\theta\in\mathbb{R}$, $F$ is continuously differentiable, convex and coercive, with $F'(\theta)=\frac{F_X(\theta)-\mu}{1-\mu}$, $\theta\in\mathbb{R}$~\cite[Proposition~2.1]{BFP09}. If $X$ admits a bounded pdf $f_X$, then $F$ is twice differentiable, with $\|F''\|_\infty=\frac{\|f_X\|_\infty}{1-\mu}\eqqcolon L<\infty$. Hence $F$ is $L$-Lipschitz-differentiable and fulfills $\mathcal{H}$\ref{asp:f}.
\end{example}

\cite[Assumption~1]{Lei+20} assumes rather that $\theta\in\mathbb{R}^m\mapsto f(x,\theta)$, $x\in\mathbb{R}^d$, is $(L,\alpha)$-H\"older-differentiable.
In practice, $L$ may depend on $x$, say, $L_x$, so that the boundedness of $(L_x)_{x\in\mathbb{R}^m}$ underpins such assumption. Since many machine learning problems suppose $X$ compactly supported, the boundedness of $(L_x)_{x\in\mathcal{X}}$ seems plausible.
Under~\cite[Assumption~1]{Lei+20}, supposing $\nabla F(\theta)=\mathbb{E}[\nabla f(X,\theta)]$, $\theta\in\mathbb{R}^m$, via Jensen's inequality,
\begin{equation*}
\|\nabla F(\theta')-\nabla F(\theta)\|
\leq\mathbb{E}\big[\|\nabla_\theta f(X,\theta')-\nabla_\theta f(X,\theta)\|\big]
\leq L\|\theta'-\theta\|^\alpha,
\quad\theta,\theta'\in\mathbb{R}^m,
\end{equation*}
thus retrieving the last part of $\mathcal{H}$\ref{asp:f}.
\\

Assume we have access to innovations $(X_n)_{n\geq1}\stackrel{\text{\tiny iid}}{\sim}X$ and to an oracle $g:\mathbb{R}^m\times\mathbb{R}^d\to\mathbb{R}$ providing a noisy approximation of $\nabla F$.
Define $b\colon\mathbb{R}^m\to\mathbb{R}^m$ by
\begin{equation}\label{eq:b}
b(\theta)=\mathbb{E}[g(X,\theta)],
\quad\theta\in\mathbb{R}^m.
\end{equation}

\begin{assumption}\label{asp:f:abc}
There exist $C>0$ and $0<\kappa\leq c$ such that
\begin{align}
\mathbb{E}\big[\|g(X,\theta)\|^2\big]&\leq C\big((F(\theta)-\inf F)+1\big),\label{asp:f:abc:i}\tag{i}\\
b(\theta)^\top\nabla F(\theta)&\geq\kappa\|\nabla F(\theta)\|^2,\label{asp:f:abc:ii}\tag{ii}\\
\|b(\theta)\|&\leq c\|\nabla F(\theta)\|,\label{asp:f:abc:iii}\tag{iii}
\quad\theta\in\mathbb{R}^m.
\end{align}
\end{assumption}

$\mathcal{H}$\ref{asp:f:abc}(\ref{asp:f:abc:i}) generalizes the variance transfer formula~\cite[Lemma~4.20]{GG24} and ensures that $b$ is well defined.
$\mathcal{H}$\ref{asp:f:abc}(\ref{asp:f:abc:ii},\ref{asp:f:abc:iii}) are standard in biased \eqref{eq:theta:n} literature and are surveyed for instance in~\cite[Assumption~8]{Dem+23}.
In the case of unbiased \eqref{eq:theta:n}, i.e.~$\mathbb{E}[g(X,\theta)]=b(\theta)=\nabla F(\theta)$, $\theta\in\mathbb{R}^m$, one has $\kappa=c=1$.

\begin{example}
In Example~\ref{exp:quantile}, $|\partial_\theta f(x,\theta)|\leq1\vee\frac\mu{1-\mu}$, $x,\theta\in\mathbb{R}$, hence setting $g(x,\theta)=\partial_\theta f(x,\theta)$, $x,\theta\in\mathbb{R}$, fulfills $\mathcal{H}\ref{asp:f:abc}$(\ref{asp:f:abc:i}).
\end{example}

\cite[Assumption~2.4]{Wei+24} (dubbed the ABC or the expected smoothness assumption, in reference to~\cite[Assumption~2]{KR23}) alternatively postulates that
\begin{equation*}
\mathbb{E}[\|g(X,\theta)\|^2]\leq C\big((F(\theta)-\inf F)+\|\nabla F(\theta)\|^2+1\big),
\quad\theta\in\mathbb{R}^m,
\end{equation*}
for some constant $C>0$.
However, under $\mathcal{H}$\ref{asp:f}, by Lemma~\ref{lmm:Holder}(\ref{lmm:Holder:ii}) and Young's inequality with the adjoint exponents $\big(\frac{1+\alpha}{2\alpha},\frac{1+\alpha}{1-\alpha}\big)$ if $\alpha\neq1$,
\begin{equation}\label{eq:gradF:square}
\|\nabla F(\theta)\|^2
\leq C(F(\theta)-\inf F)^\frac{2\alpha}{1+\alpha}
\leq C\big((F(\theta)-\inf F)+1\big)
<\infty,
\quad\theta\in\mathbb{R}^m,
\end{equation}
for some constant $C>0$ that may change from line to line. This inequality holds also if $\alpha=1$.
We thus retrieve $\mathcal{H}$\ref{asp:f:abc}(\ref{asp:f:abc:i}).

$\mathcal{H}$\ref{asp:f:abc}(\ref{asp:f:abc:i}) is more generic than the interpolation condition $\mathbb{E}[\|g(X,\theta)\|^2]\leq C(F(\theta)-\inf F)$, $\theta\in\mathbb{R}^m$~\cite[\S1.2]{Liu+23}.
We refer to Remark~\ref{rmk:interpolation} for extensive comments on the limitations posed by this condition.

Under~\cite[Assumption~1]{Lei+20}, which stipulates that $\theta\in\mathbb{R}^m\mapsto f(x,\theta)$, $x\in\mathbb{R}^d$, is $(L,\alpha)$-H\"older-differentiable, if $\inf_{x,\theta}f>-\infty$, then, by Lemma~\ref{lmm:Holder}(\ref{lmm:Holder:ii}) and Young's inequality with the adjoint exponents $\big(\frac{1+\alpha}{2\alpha},\frac{1+\alpha}{1-\alpha}\big)$ if $\alpha\neq1$,
\begin{equation*}
\mathbb{E}\big[\|\nabla_\theta f(X,\theta)\|^2\big]
\leq C\mathbb{E}\big[\big(f(X,\theta)-\inf f\big)^\frac{2\alpha}{1+\alpha}\big]
\leq C\big((F(\theta)-\inf F)+1\big)<\infty,
\quad\theta\in\mathbb{R}^m,
\end{equation*}
for some constant $C>0$ that may change from line to line. This remains true if $\alpha=1$.
The above inequality matches $\mathcal{H}$\ref{asp:f:abc}(\ref{asp:f:abc:i}).
\\

Define the stochastic gradient oracles
\begin{equation*}
g_n(\theta)\coloneqq g(X_n,\theta),
\quad\theta\in\mathbb{R}^m,\;n\geq1.
\end{equation*}
SGD leverages the stochastic gradient oracles $(g_n)_{n\geq1}$ to approximate solutions of \eqref{opt:prob} via the dynamics
\begin{equation}
\theta_n=\theta_{n-1}-\gamma_ng_n(\theta_{n-1}),\quad n\geq1,
\tag{SGD}
\label{eq:theta:n}
\end{equation}
driven by the innovations $(X_n)_{n\geq1}\stackrel{\text{\tiny iid}}{\sim}X$, initialized at some $\theta_0\in L^2(\mathbb{P})$ that is independent of $(X_n)_{n\geq1}$, and advancing at the positive learning rate $(\gamma_n)_{n\geq1}$ such that $\sum_{n=1}^\infty\gamma_n=\infty$ and $\lim_{n\to\infty}\gamma_n=0$.

\begin{remark}\label{rmk:grad:f:L2}
Under $\mathcal{H}$\ref{asp:f}, using a first order Taylor-Lagrange expansion,
\begin{equation}\label{eq:F<1+theta^2}
|F(\theta)|
=\Big|F(0)+\theta^\top\int_0^1\nabla F(t\theta)\mathrm{d}t\Big|
\leq|F(0)|+\|\nabla F(0)\|\cdot\|\theta\|+\frac{L\|\theta\|^{1+\alpha}}{1+\alpha}
\leq C(1+\|\theta\|^2),
\quad\theta\in\mathbb{R}^m,
\end{equation}
for some constant $C>0$.
Thus, via the \eqref{eq:theta:n} recursion, $\mathcal{H}$\ref{asp:f:abc} and Jensen's inequality, it ensues iteratively that $\mathbb{E}[|F(\theta_n)|]+\mathbb{E}[\|\theta_n\|^2]+\mathbb{E}[\|g_n(\theta_{n-1})\|^2]+\mathbb{E}[\|\nabla F(\theta_{n-1})\|^2]<\infty$, $n\geq1$.
\end{remark}

\begin{assumption}
$(\gamma_n)_{n\geq1}$ is a positive sequence such that
\begin{equation*}
\sum_{n=1}^\infty\gamma_n=\infty
\quad\text{and}\quad
\sum_{n=1}^\infty\gamma_n^{1+\alpha}<\infty.
\end{equation*}
\label{asp:gamma}
\end{assumption}

Note that the H\"older exponent $\alpha$ from $\mathcal{H}$\ref{asp:f} shows up in $\mathcal{H}$\ref{asp:gamma}. This assumption appears for instance in~\cite[Theorem~4(a)]{Lei+20}.
It translates a desirable trait for \eqref{eq:theta:n} to find a trade-off between exploration of the parameter space and exploitation of the loss landscape geometry.

\begin{example}\label{exp:gamma:sample}
Typical sequences adhering to $\mathcal{H}$\ref{asp:gamma} write as
\begin{equation}
\gamma_n=\frac{\gamma_0}{(n+c)^s},
\quad\gamma_0>0,\;c\geq0,\;s\in\Big(\frac1{1+\alpha},1\Big],\;n\geq1.
\label{eq:gamma:usecase}
\end{equation}
Another possible form is
\begin{equation*}
\gamma_n=\frac{\gamma_0}{(n+c)\ln^{s}{(1+c'+n)}},
\quad\gamma_0>0,\;c,c'\geq0,\;s\in\Big(\frac1{1+\alpha},1\Big],\;n\geq1.
\end{equation*}
\cite[Corollary~3(b)]{Lei+20} also mentions
\begin{equation*}
\gamma_n=\frac{\gamma_0}{((n+c)\ln^s{(1+c'+n)})^{1/(1+\alpha)}},
\quad\gamma_0>0,\;c,c'\geq0,\;s>1,\;n\geq1.
\end{equation*}
\end{example}

\subsection{Convergence Analysis}
\label{ssec:cv}

In this section, we seek convergence guarantees for \eqref{eq:theta:n}.
Convergence of the function values $(F(\theta_n))_{n\geq0}$ is another useful imperative for \eqref{eq:theta:n}, the reason being twofold.
First, convergence of $(\theta_n)_{n\geq0}$ is simply not guaranteed in the nonconvex setting, even when $(F(\theta_n))_{n\geq0}$ does converge~\cite{AMA05}.
Second, when a limit $\theta_\star\in\mathbb{R}^m$ to $(\theta_n)_{n\geq0}$ exists, it is often the case that the behavior of the function value gap $F(\theta_n)-F(\theta_\star)$ is the one reflected in that of the iterate gap $\theta_n-\theta_\star$ rather than the other way around~\cite[\S2.1]{LZ24}.

Throughout, $(\mathcal{F}_n)_{n\geq0}$ designates the filtration recursively spawned by $\theta_0$ and the innovations $(X_n)_{n\geq1}$, and $\mathbb{E}_n$ denotes the conditional expectation $\mathbb{E}[\,\cdot\,|\mathcal{F}_n]$, $n\geq0$.
We hereby provide our standing convergence results and follow up with discussions thereon.
Related proofs are postponed to Section~\ref{sec:cv}.

\begin{theorem}\label{thm:F:cv:gradF:0}
Under $\mathcal{H}$\ref{asp:f}--\ref{asp:gamma}, the function values $(F(\theta_n))_{n\geq0}$ converge $\Pas$ to a real-valued random variable $F_\star\in L^1(\mathbb{P})$ and $\nabla F(\theta_n)\to0$ $\Pas$ as $n\to\infty$.
\end{theorem}

Note that the above result is independent from any coercivity or iterate boundedness assumption.
It encompasses loss landscapes such as $F\colon\theta\in\mathbb{R}\mapsto\ln(1+\mathrm{e}^{-\theta})$, whose gradient only vanishes at infinity, entailing that corresponding \eqref{eq:theta:n} trajectories escape to infinity $\Pas$.
There are in fact two issues with presupposing the $\Pas$-iterate boundedness.
On the one hand, it de facto implies the $\Pas$-boundedness of the gradients $(\nabla F(\theta_n))_{n\geq0}$.
On the other hand, an evaluation of $\mathbb{P}(\sup_{n\geq0}\|\theta_n\|<\infty)$ is necessary to fully assess the related convergence guarantees.
\cite[\S4.1]{Mer+20} signals that presupposing iterate boundedness renders the ensuing convergence results of little practical value because \eqref{eq:theta:n} trajectories can in fact escape to infinity, as already observed for gradient flows~\cite[Theorem~2.2]{AMA05}.
Our convergence guarantees lie within the continuity of recent efforts to eliminate such an assumption~\cite{Mer+20,QMM24}.

\cite[Theorem~4.2]{Dav+20} states the $\Pas$ function-value and weak convergences of unbiased \eqref{eq:theta:n} for locally Lipschitz loss landscapes (i.e.~almost everywhere differentiable, by Rademacher's theorem), under the assumption that the iterates are $\Pas$ bounded.
Our convergence result is free from such an assumption.

\cite[Theorem~1.1]{DK24} reaches similar results for biased \eqref{eq:theta:n} on locally H\"older-differentiable loss landscapes, conditionally to the iterates remaining bounded and the conditional $q$-th moments of the stochastic gradient noises being deterministically quantified~\cite[Assumption~$\mathbb{M}^{\sigma,q}$]{DK24}. In our notation, the latter condition writes
\begin{equation}\label{eq:moment:condition}
\mathbb{E}_{n-1}\big[\| g_n(\theta_{n-1})-b(\theta_{n-1})\|^q\big]\leq v_n,
\quad n\geq1,
\end{equation}
for some $q\geq1$ and a deterministic positive sequence $(v_n)_{n\geq1}$ satisfying some stringent asymptotic constraints, such as $\gamma_nv_n^2\to0$ as $n\to\infty$, as well as additional summability conditions.
At this point already, our assumption $\mathcal{H}$\ref{asp:f:abc}(\ref{asp:f:abc:i}) is significantly milder.
Let us examine further the condition \eqref{eq:moment:condition}.
Consider $f(x,\theta)=\frac12x^2\theta^2$, $F(\theta)=\frac12\mathbb{E}[X^2]\theta^2$, $x,\theta\in\mathbb{R}$, where $X\in L^2(\mathbb{P})$ is a real-valued random variable. Then $g(x,\theta)=\partial_\theta f(x,\theta)=x^2\theta$, $b(\theta)=F'(\theta)=\mathbb{E}[X^2]\theta$, $x,\theta\in\mathbb{R}$, and, by Jensen's inequality,
\begin{equation*}
\mathbb{E}\big[\big|\partial_\theta f(X,\theta)-F'(\theta)\big|^q\big]
\geq\mathbb{E}\big[\big|\partial_\theta f(X,\theta)-F'(\theta)\big|\big]^q
=\mathbb{E}\big[\big|X^2-\mathbb{E}[X^2]\big|\big]^q|\theta|^q,
\quad\theta\in\mathbb{R},\;q\geq1.
\end{equation*}
On this example, the assumption \eqref{eq:moment:condition} entails that $(\theta_n)_{n\geq0}$ are compactly supported. And in view of the \eqref{eq:theta:n} update rule, so are the stochastic gradients $(g_n(\theta_{n-1}))_{n\geq1}$, which excludes for instance stochastic gradients that have a Gaussian distribution.
This limits the scope of applicability of the obtained convergence results.
\cite{DK24} further references \cite[Lemma~D.1]{DK23} as a guarantee for iterate boundedness for coercive Lipschitz-differentiable loss landscapes, which however requires the assumption~\eqref{eq:moment:condition}.
As for weak convergence, the technique used in~\cite{DK24} relies on the uniform continuity of $\nabla F$ on a compact set containing all of the iterates $(\theta_n)_{n\geq0}$.
This argument has roots in~\cite[Proposition~1]{BT00} (c.f. also~\cite[Lemma~1]{Ora20}) and can only be utilized if the iterates are guaranteed to be bounded.

\cite[Theorem~2]{Lei+20} shows function-value convergence of unbiased \eqref{eq:theta:n} via a quasisupermartingale formalism similar to \eqref{eq:quasi:super:martingale}, under an $(L,\alpha)$-H\"older-differentiability assumption on $\theta\in\mathbb{R}^m\mapsto f(x,\theta)$, $x\in\mathbb{R}^m$.
\cite[Theorem~2(c)]{Lei+20} proves that $\mathbb{E}[\|\nabla F(\theta_n)\|]$ vanishes as $n\to\infty$ for $\alpha=1$, which solely ensures that $\nabla F(\theta_n)\to0$ in probability as $n\to\infty$.
\cite[Theorem~4(a)]{Lei+20} eventually proves weak convergence by additionally imposing a global P{\L} condition. Our result does not require such an assumption to reach the $\Pas$-convergence of $(\nabla F(\theta_n))_{n\geq0}$ to $0$.

Function-value convergence is ensured in~\cite[Proposition~5.5]{CFR23} for $\alpha=1$ via a quasisupermartingale formalism close to~\cite{Lei+20}.
Weak convergence is treated in~\cite[Proposition~6.3]{CFR23} under the lens of iterate convergence, by assuming the strong growth condition $\mathbb{E}_{n-1}[\|g_n(\theta_{n-1})\|^2]\leq C\|\nabla F(\theta_{n-1})\|^2$, $n\geq1$, with $C>0$.
This is far from representing a generic learning problem \eqref{opt:prob}.
\\

To be able to characterize the asymptotics of $(\theta_n)_{n\geq0}$, we start by formulating a plain coercivity condition.

\begin{assumption}
\label{asp:f->8}
$F$ is coercive.
That is, $F(\theta)\to\infty$ as $\|\theta\|\to\infty$.
\end{assumption}

\begin{lemma}\label{lmm:aux}
Under $\mathcal{H}$\ref{asp:f}--\ref{asp:f->8}, $(\theta_n)_{n\geq0}$ is $\Pas$ bounded.
\end{lemma}

$\mathcal{H}$\ref{asp:f->8} is devoid of any additional coercive regularity. Machine learning problems are customarily regularized to achieve generalization and to induce convexity.
This practice often involves adding a regularization term to the loss associated with the learning model: $f(x,\theta)=\ell(x,\theta)+\varphi(\theta)\geq\varphi(\theta)$, $x\in\mathbb{R}^d$, $\theta\in\mathbb{R}^m$, where $\ell$ is a nonnegative loss and $\varphi$ is a coercive regularizer.

Lemma~\ref{lmm:aux} retrieves the $\Pas$ boundedness of the iterates $(\theta_n)_{n\geq0}$ as a consequence rather than an assumption, inasmuch as the nominal SGD user can only check the satisfiability of $\mathcal{H}$\ref{asp:f}--\ref{asp:f->8} beforehand.

In~\cite[Assumptions~2 \&~3]{Mer+20}, the sublevels of both $F$ and $\|\nabla F\|$ are assumed to be bounded, which represents a strong coercive regularity on the loss landscape, as the set of stationary points is then necessarily bounded. This is not required by $\mathcal{H}$\ref{asp:f->8}.
\\

Hereafter, we denote $\nabla F^{-1}(0)$ the set of stationary points of $F$.
The main convergence result of our analysis follows.
Note that nder $\mathcal{H}$\ref{asp:f->8}, $\nabla F^{-1}(0)\neq\varnothing$.

\begin{theorem}\label{thm:cv}
Assume $\mathcal{H}$\ref{asp:f}--\ref{asp:f->8} hold. Then,
\begin{enumerate}[\bf(i)]
\item\label{thm:cv:i}
    $\Pas$, the set $\mathcal{A}$ of accumulation points of $(\theta_n)_{n\geq0}$ is a nonempty compact connected set, and every element of $\mathcal{A}$ is a stationary point of $F$;
\item\label{thm:cv:ii}
    if $\nabla F^{-1}(0)$ is additionally at most countable, $(\theta_n)_{n\geq0}$ converges $\Pas$ to a stationary point of $F$.
\end{enumerate}
\end{theorem}

Theorem~\ref{thm:cv} is echoed in~\cite[Proposition~2]{CP17} in a surrogate optimization setting.

\subsection{Convergence Rate Analysis}
\label{ssec:speed}

In this section, we seek to quantify the convergence speed of $(F(\theta_n))_{n\geq0}$ and assess the subsequent \eqref{eq:theta:n} complexity.
Convergence rates of nonconvex approximation schemes are often derived under a {\L}ojasiewicz-type condition~\cite{AB09,FGP15,DK24,Wei+24} as we put forward below.
Related proofs are available in Section~\ref{sec:speed}.

\begin{assumption}
\label{asp:Lojasiewicz}
$F$ is {\L}ojasiewicz, i.e.~it is continuously differentiable and, for all $\theta_\star\in\nabla F^{-1}(0)$, there exist $\mathcal{V}$ an open neighborhood of $\theta_\star$, $\beta\in(0,1)$ and $\zeta>0$ such that
\begin{equation*}
|F(\theta)-F(\theta_\star)|^\beta\leq\zeta\|\nabla F(\theta)\|,
\quad\theta\in\mathcal{V}.
\end{equation*}
\end{assumption}

\begin{remark}\label{rmk:beta>alpha/1+alpha}
Unlike~\cite[Definition~1.2]{DK24}, we do not rule out the possibility that $\beta\in\big(0,\frac12\big)$.
Indeed, via $\mathcal{H}$\ref{asp:f} and Lemma~\ref{lmm:Holder}(\ref{lmm:Holder:ii}), the above property yields $(F(\theta)-\inf F)^\beta\leq\zeta\|\nabla F(\theta)\|\leq C(F(\theta)-\inf F)^\frac\alpha{1+\alpha}$ for $\theta$ close enough to a global optimum, hence $\beta\geq\frac\alpha{1+\alpha}$. This lower bound describes $\big(0,\frac12\big]$ as $\alpha$ varies through $(0,1]$.
\end{remark}

As surveyed in~\cite[Table~1(a)]{GG24} or~\cite[Appendix~A]{KNS16}, the classical assumptions under which nonconvex \eqref{eq:theta:n} convergence rates are derived often require compactness of the iterate space and/or bounded stochastic gradient noises~\cite{LZ24}.
Recently, a slew of literature has departed from such frameworks and settled for a {\L}ojasiewicz-type condition.
Such a condition (also dubbed gradient domination~\cite{Wei+24}) has found success in proving convergence rates and global convergence properties for GD~\cite{AMA05} and proximal methods~\cite{AB09}.
It has been lately adopted in the context of \eqref{eq:theta:n}~\cite{DK24,Wei+24,AL24} for similar purposes.

{\L}ojasiewicz-type conditions help assess the local growth of a function around its stationary points.
One such condition has first been considered in~\cite[Theorem~4]{Pol63} in a global form to obtain an exponential convergence rate for GD.
It has then been revisited in a local form in~\cite{AMA05} to prove iterate convergence of GD, which is not guaranteed otherwise (c.f.~the intuitive counterexample in~\cite[\S3]{Cur44} and the explicit one in~\cite[\S2]{AMA05}).
{\L}ojasiewicz conditions are notably satisfied by analytic functions~\cite{AMA05}.
\begin{proposition}[{\cite[\S18, Proposition~1]{Loj65}}]
Let $\varphi\colon\mathbb{R}^m\to\mathbb{R}$ be a function that is analytic on a neighborhood of some point $\theta_\star\in\mathbb{R}^m$ satisfying $\varphi(\theta_\star)=0$. Then, there exists $\beta\in(0,1)$ such that
\begin{equation*}
\|\nabla\varphi(\theta)\|\geq|\varphi(\theta)|^\beta
\quad\text{on a neighborhood of $\theta_\star$.}
\end{equation*}
\end{proposition}
\noindent When $\nabla\varphi(\theta_\star)\neq0$, the above {\L}ojasiewicz property is straightforward~\cite[Remark~1.3]{DK24}.
The particular case where $\beta=\frac12$ is referred to as the Polyak-{\L}ojasiewicz property~\cite[Remark~2.23]{GG24}.
Bear in mind that the {\L}ojasiewicz condition does not require analyticity per se, as $\mathcal{H}$\ref{asp:Lojasiewicz} encompasses a larger class of functions~\cite{AMA05}.
Indeed, \cite{Kur98} guarantees it for subanalytic functions of class $\mathcal{C}^1$. We refer for instance to~\cite{BM88,BDL07} for developments on subanalytic functions.

\begin{proposition}[{\cite[Theorem~{\L}I]{Kur98}}]
Let $\varphi\colon\mathcal{V}\to\mathbb{R}$ be a subanalytic which is differentiable in $\mathcal{V}\setminus\varphi^{-1}(0)$, where $\mathcal{V}$ is an open bounded subset of $\mathbb{R}^m$.
Then there exist $C>0$, $\varrho>0$ and $\beta\in[0,1)$ such that
\begin{equation*}
\|\nabla\varphi(\theta)\|\geq C|\varphi(\theta)|^\beta,
\end{equation*}
for each $\theta\in\mathcal{V}$ such that $|\varphi(\theta)|\in(0,\varrho)$.
If in addition $\lim_{\theta\to\vartheta}\varphi(\theta)=0$ for some $\vartheta\in\mathcal{V}$ (which holds in the classical case, where $\varphi$ is analytic and $\vartheta\in\mathcal{V}$, $\varphi(\vartheta)=0$), then the above inequality holds for each $\theta\in\mathcal{V}\subset\varphi^{-1}(0)$ close to $\vartheta$.
\end{proposition}
\noindent Of course, if $\nabla\varphi(\vartheta)=0$, then $\beta>0$.
\\

In accordance with the above results, the {\L}ojasiewicz exponent $\beta\in(0,1)$ and the {\L}ojasiewicz constant $\zeta>0$ in $\mathcal{H}$\ref{asp:Lojasiewicz} depend on the stationary point $\theta_\star$.
When $\mathcal{V}=\mathbb{R}^m$, $F$ is termed globally {\L}ojasiewicz~\cite[Remark~2.23]{GG24}, in which case every stationary point is a global minimizer~\cite{KNS16}. The {\L}ojasiewicz parameters $(\beta,\zeta)\in(0,1)\times\mathbb{R}^*_+$ can then be set uniformly across $\nabla F^{-1}(0)$.
Otherwise, in the general case where $\mathcal{V}\subset\mathbb{R}^m$, eliciting a convergence rate directly under $\mathcal{H}$\ref{asp:Lojasiewicz} is rather tedious.
To ease such task, we propose below several alternative frameworks to exploit the local {\L}ojasiewicz property more conveniently.

\begin{proposition}\label{prp:Lojasiewicz}
 Assume $\mathcal{H}$\ref{asp:Lojasiewicz} holds.
 Then,
\begin{enumerate}[\bf(i)]
    \item\label{prp:Lojasiewicz:i}
    if $\mathcal{K}\subset\mathbb{R}^m$ is a compact set such that $\nabla F^{-1}(0)\cap\mathcal{K}\neq\varnothing$,
    \begin{enumerate}[\bf a.]
        \item\label{prp:Lojasiewicz:i:a}
        there exist $\beta\in(0,1)$ and $\zeta>0$ such that, for all $\theta_\star\in\nabla F^{-1}(0)\cap\mathcal{K}$, there exists an open neighborhood $\mathcal{V}$ of $\theta_\star$ such that
        \begin{equation}\label{eq:framework:1}
        |F(\theta)-F(\theta_\star)|^\beta\leq\zeta\|\nabla F(\theta)\|,
        \quad\theta\in\mathcal{V}.
        \end{equation}
        \item\label{prp:Lojasiewicz:i:b}
        there exists $\beta\in(0,1)$ such that, for all $\theta_\star\in\nabla F^{-1}(0)\cap\mathcal{K}$, there exists an open neighborhood $\mathcal{V}$ of $\theta_\star$ such that
        \begin{equation}\label{eq:framework:1:bis}
        |F(\theta)-F(\theta_\star)|^\beta\leq\|\nabla F(\theta)\|,
        \quad\theta\in\mathcal{V}.
        \end{equation}
    \end{enumerate}
    \item\label{prp:Lojasiewicz:ii}
        there exists $\beta\in(0,1]$ such that, for all $\theta_\star\in\nabla F^{-1}(0)$, there exists an open neighborhood $\mathcal{V}$ of $\theta_\star$ such that
        \begin{equation}\label{eq:framework:2}
        |F(\theta)-F(\theta_\star)|^\beta\leq\|\nabla F(\theta)\|,
        \quad\theta\in\mathcal{V}.
        \end{equation}
    \item\label{prp:Lojasiewicz:iii}
    if ${\inf{F}>{-\infty}}$ and $F$ is globally {\L}ojasiewicz, every stationary point of $F$ is a global minimizer, and there exist $\beta\in\big[\frac\alpha{1+\alpha},1\big)$ and $\zeta>0$ such that
    \begin{equation}\label{eq:framework:3}
    (F(\theta)-\inf F)^\beta\leq\zeta\|\nabla F(\theta)\|,
    \quad\theta\in\mathbb{R}^m.
    \end{equation}
\end{enumerate}
\end{proposition}

\begin{remark}
If one assumes the set of stationary points of $F$ bounded, then Proposition~\ref{prp:Lojasiewicz}(\ref{prp:Lojasiewicz:i}) retrieves uniform {\L}ojasiewicz parameters $(\beta,\zeta)\in(0,1)\times\mathbb{R}^*_+$ across all stationary points.
We however opt out of this extra condition to treat the more general case of unbounded set $\nabla F^{-1}(0)$.
\end{remark}

\cite[Lemma~1]{AB09} and~\cite[Lemma~F.1]{Wei+24} fix a compact connected set $\mathcal{K}\subset\nabla F^{-1}(0)$ (of local minima in the case of~\cite{Wei+24}) and prove the existence of uniform {\L}ojasiewicz parameters $(\beta,\zeta)\in(0,1)\times\mathbb{R}^*_+$ on an open neighborhood of $\mathcal{K}$.
The issue with such approach is the additional need to control $\mathbb{P}(\mathcal{A}\subset\mathcal{K})$ (recalling that $\mathcal{A}$ is the random set of accumulation points of $(\theta_n)_{n\geq0}$), which is not a given, unless additional measures are taken to confine the \eqref{eq:theta:n} trajectory to a small neighborhood of $\mathcal{A}\subset\mathcal{K}$.
In Proposition~\ref{prp:Lojasiewicz}(\ref{prp:Lojasiewicz:i})\hyperref[prp:Lojasiewicz:i:a]{a}, we utilize~\cite[Lemma~4.3]{DK24} to obtain uniform {\L}ojasiewicz parameters $(\beta,\zeta)\in(0,1)\times\mathbb{R}^*_+$ on an arbitrary compact set $\mathcal{K}\subset\mathbb{R}^m$.
The coercivity of $F$ via $\mathcal{H}$\ref{asp:f->8} hints at its levelsets as natural candidate compact sets upon which one can uniformize the {\L}ojasiewicz parameters.
Let
\begin{equation*}
\mathcal{L}_\ell\coloneqq\{\theta\in\mathbb{R}^m,F(\theta)\leq\ell\},
\quad\ell>\inf F.
\end{equation*}

\begin{remark}
The continuity and coercivity of $F$, via $\mathcal{H}$\ref{asp:f} and $\mathcal{H}$\ref{asp:f->8}, entail the compactness of $\mathcal{L}_\ell$, $\ell>\inf F$.
\end{remark}

Proposition~\ref{prp:Lojasiewicz}(\ref{prp:Lojasiewicz:i})\hyperref[prp:Lojasiewicz:i:b]{b} provides a similar framework to Proposition~\ref{prp:Lojasiewicz}(\ref{prp:Lojasiewicz:i})\hyperref[prp:Lojasiewicz:i:a]{a}, however with a {\L}ojasiewicz constant $\zeta=1$.
The presented advantage is discussed in more detail in Remark~\ref{rmk:advantage}.

To allow for a global view on the function-value convergence behavior, we develop a novel formulation for the {\L}ojasiewicz condition in Proposition~\ref{prp:Lojasiewicz}(\ref{prp:Lojasiewicz:ii}), justified by the outcome of Lemma~\ref{lmm:beta+:zeta1}.
The latter is based on the observation that $\frac{|F(\theta)-F(\theta_\star)|^{\beta_{\theta_\star}}}{\|\nabla F(\theta)\|}$ is locally bounded around each $\theta_\star\in\nabla F^{-1}(0)$~\cite{BDL07}.
This new formulation ensures a unique global exponent $\beta\in(0,1]$ and a unique global scaling factor $\zeta=1$ throughout all stationary points of $F$.

Eventually, Proposition~\ref{prp:Lojasiewicz}(\ref{prp:Lojasiewicz:iii}) recalls a classical result on globally {\L}ojasiewicz functions.
\\

The following result provides a convergence rate in high probability using $\mathcal{H}$\ref{asp:gamma} as a sole descriptor of the asymptotic behavior of $(\gamma_n)_{n\geq1}$.

\begin{theorem}\label{thm:F-F*}
Assume $\mathcal{H}$\ref{asp:f}--\ref{asp:Lojasiewicz} hold.
Then,
\begin{enumerate}[\bf(i)]
\item\label{thm:F-F*:i}
for all $\delta\in(0,1)$, there exists $\ell_\delta>\inf F$ such that $\mathbb{P}(\mathcal{A}\subset\mathcal{L}_{\ell_\delta})\geq1-\delta$, and
\begin{enumerate}[\bf a.]
    \item\label{thm:F-F*:i:a}
    there exist $\beta_\delta\in(0,1)$ and $\zeta_\delta>0$ such that \eqref{eq:framework:1} holds for $\mathcal{K}=\mathcal{L}_{\ell_\delta}$.
    \item\label{thm:F-F*:i:b}
    there exists $\beta_\delta\in(0,1)$ such that \eqref{eq:framework:1:bis} holds for $\mathcal{K}=\mathcal{L}_{\ell_\delta}$.
\end{enumerate}
Besides, there exists $C_\delta>0$ such that, for all $n\geq1$,
\begin{equation*}
\inf_{0\leq k\leq n}{F(\theta_k)}-F_\star
\leq C_\delta\Big(\sum_{k=1}^n\gamma_k\Big)^{-1\wedge\frac1{2\beta_\delta}}
\quad\text{with probability at least $1-\delta$.}
\end{equation*}
\item\label{thm:F-F*:ii}
there exists $\beta\in(0,1]$ such that \eqref{eq:framework:2} holds. Moreover, for all $\delta\in(0,1)$, there exists $C_\delta>0$ such that, for all $n\geq1$,
\begin{equation*}
\inf_{0\leq k\leq n}{F(\theta_k)}-F_\star
\leq C_\delta\Big(\sum_{k=1}^n\gamma_k\Big)^{-1\wedge\frac1{2\beta}}
\quad\text{with probability at least $1-\delta$.}
\end{equation*}
\item\label{thm:F-F*:iii}
if $F$ is globally {\L}ojasiewicz, there exist $\beta\in\big[\frac\alpha{1+\alpha},1\big)$ and $\zeta>0$ such that \eqref{eq:framework:3} holds. Furthermore, there exists $C>0$ such that, for all $\delta\in(0,1)$ and all $n\geq1$,
\begin{equation*}
\inf_{0\leq k\leq n}{F(\theta_k)}-\inf F
\leq\frac{C}{\delta^\frac1\beta}\Big(\sum_{k=1}^n\gamma_k\Big)^{-\frac1{2\beta}}
\quad\text{with probability at least $1-\delta$.}
\end{equation*}
\end{enumerate}
\end{theorem}

To obtain finer convergence rates than to the ones provided above, we describe the asymptotic behavior of $(\gamma_n)_{n\geq1}$ more precisely.

\begin{assumption}\label{asp:gamma:misc}
\begin{enumerate}[\bf(i)]
    \item\label{asp:gamma:misc:i}
Either of the following holds:
    \begin{enumerate}[\bf a.]
        \item\label{asp:gamma:misc:ia}
    $\ln{\big(\frac{\gamma_{n-1}}{\gamma_n}\big)}=\mathrm{o}(\gamma_n)$;
        \item\label{asp:gamma:misc:ib}
    there exists $\gamma_\star>0$ such that $\ln{\big(\frac{\gamma_{n-1}}{\gamma_n}\big)}\sim\frac{\gamma_n}{\gamma_\star}$.
    \end{enumerate}
    \item\label{asp:gamma:misc:ii}
$\sum_{n=1}^\infty\frac{\gamma_n}{\sum_{k=1}^n\gamma_k}=\infty$.
    \item\label{asp:gamma:misc:iii}
$\gamma_n=\mathrm{O}\big(\big(\sum_{k=1}^n\gamma_k\big)^{-\rho}\big)$, for some $\rho>\frac1\alpha\geq1$.
\end{enumerate}
\end{assumption}

\begin{example}
Suppose that, as in \eqref{eq:gamma:usecase}, $\gamma_n=\Theta(\frac1{n^s})$, $n\geq1$, with $s\in(0,1]$.
Then $s<1$ corresponds to $\mathcal{H}$\ref{asp:gamma:misc}(\ref{asp:gamma:misc:i})\hyperref[asp:gamma:misc:ia]{a} and $s=1$ to $\mathcal{H}$\ref{asp:gamma:misc}(\ref{asp:gamma:misc:i})\hyperref[asp:gamma:misc:ib]{b}.
$\mathcal{H}$\ref{asp:gamma:misc}(\ref{asp:gamma:misc:ii}) checks for all $s\in(0,1]$ and $\mathcal{H}$\ref{asp:gamma:misc}(\ref{asp:gamma:misc:iii}) applies if $s\in[\frac\rho{1+\rho},1]$.
By the increasing monotony of $x\in\mathbb{R}_+\mapsto\frac{x}{1+x}$ and the fact that $\rho>\frac1\alpha$, one has $\frac\rho{1+\rho}>\frac1{1+\alpha}$, so $\mathcal{H}$\ref{asp:gamma} is automatically verified.
\end{example}

\begin{remark}
$\mathcal{H}$\ref{asp:gamma:misc}(\ref{asp:gamma:misc:i}) is standard in the stochastic approximation literature~\cite[Assumption~C4]{For15}. The parameter $\gamma_\star$ in $\mathcal{H}$\ref{asp:gamma:misc}(\ref{asp:gamma:misc:i})\hyperref[asp:gamma:misc:ib]{b} can be interpreted as an initialization parameter, since if $\gamma_n=\frac{\gamma_1}n$, $n\geq1$, one simply has $\gamma_\star=\gamma_1$.
$\mathcal{H}$\ref{asp:gamma:misc}(\ref{asp:gamma:misc:ii}) appears as early as in~\cite[Equation (26)]{RM51} for studying stochastic approximation algorithms.
$\mathcal{H}$\ref{asp:gamma:misc}(\ref{asp:gamma:misc:iii}) generalizes~\cite[Proposition~4.3(c)]{DK24} by allowing for a flexible choice on $\rho$.
Essential properties satisfied by sequences fulfilling $\mathcal{H}$\ref{asp:gamma:misc} are provided in Lemma~\ref{lmm:gamma}.
\end{remark}

\begin{theorem}\label{thm:F-F*:special:cases}
Assume $\mathcal{H}$\ref{asp:f}--\ref{asp:gamma:misc} hold.
Then,
\begin{enumerate}[\bf(i)]
\item\label{thm:F-F*:special:cases:i}
for all $\delta\in(0,1)$, there exists $\ell_\delta>\inf F$ such that $\mathbb{P}(\mathcal{A}\subset\mathcal{L}_{\ell_\delta})\geq1-\delta$, and
\begin{enumerate}[\bf a.]
    \item\label{thm:F-F*:special:cases:i:a}
    there exist $\beta_\delta\in(0,1)$ and $\zeta_\delta>0$ such that \eqref{eq:framework:1} holds for $\mathcal{K}=\mathcal{L}_{\ell_\delta}$.
    Besides, there exists $C_\delta>0$ such that, for all $n\geq1$, with probability at least $1-\delta$,
    \begin{equation*}
    F(\theta_n)-F_\star\leq C_\delta
    \begin{cases}
    \gamma_n^{\alpha\wedge\frac12},
    &\text{if $\beta_\delta\in(0,\frac12]$ and $\kappa\gamma_\star>\zeta_\delta^{1/\beta_\delta}(\alpha\vee\frac12)$ under $\mathcal{H}$\ref{asp:gamma:misc}(\ref{asp:gamma:misc:i})\hyperref[asp:gamma:misc:ib]{b},}\\
    \big(\sum_{k=1}^n\gamma_k\big)^{-r_\delta},
    &\text{if $\beta_\delta\in(\frac12,1)$, with $r_\delta=\frac1{2\beta_\delta-1}\wedge(\alpha\rho-1)\wedge\frac{\rho-1}2>0$.}
    \end{cases}
    \end{equation*}
    \item\label{thm:F-F*:special:cases:i:b}
    there exists $\beta_\delta\in(0,1)$ such that \eqref{eq:framework:1:bis} holds for $\mathcal{K}=\mathcal{L}_{\ell_\delta}$.
    Besides, there exists $C_\delta>0$ such that, for all $n\geq1$, with probability at least $1-\delta$,
    \begin{equation*}
    F(\theta_n)-F_\star\leq C_\delta
    \begin{cases}
    \gamma_n^{\alpha\wedge\frac12},
    &\text{if $\beta_\delta\in(0,\frac12]$ and $\kappa\gamma_\star>\alpha\vee\frac12$ under $\mathcal{H}$\ref{asp:gamma:misc}(\ref{asp:gamma:misc:i})\hyperref[asp:gamma:misc:ib]{b},}\\
    \big(\sum_{k=1}^n\gamma_k\big)^{-r_\delta},
    &\text{if $\beta_\delta\in(\frac12,1)$, with $r_\delta=\frac1{2\beta_\delta-1}\wedge(\alpha\rho-1)\wedge\frac{\rho-1}2>0$.}
    \end{cases}
    \end{equation*}
\end{enumerate}
\item\label{thm:F-F*:special:cases:ii}
there exists $\beta\in(0,1]$ such that \eqref{eq:framework:2} holds. Moreover, for all $\delta\in(0,1)$, there exists $C_\delta>0$ such that, for all $n\geq1$, with probability at least $1-\delta$,
\begin{equation*}
F(\theta_n)-F_\star\leq C_\delta
\begin{cases}
\gamma_n^{\alpha\wedge\frac12},
&\text{if $\beta\in(0,\frac12]$ and $\kappa\gamma_\star>\alpha\vee\frac12$ under $\mathcal{H}$\ref{asp:gamma:misc}(\ref{asp:gamma:misc:i})\hyperref[asp:gamma:misc:ib]{b},}\\
\big(\sum_{k=1}^n\gamma_k\big)^{-r},
&\text{if $\beta\in(\frac12,1]$, with $r=\frac1{2\beta-1}\wedge(\alpha\rho-1)\wedge\frac{\rho-1}2>0$.}
\end{cases}
\end{equation*}
\item\label{thm:F-F*:special:cases:iii}
if $F$ is globally {\L}ojasiewicz, there exist $\beta\in\big[\frac\alpha{1+\alpha},1)$ and $\zeta>0$ such that \eqref{eq:framework:3} holds.
Suppose $\beta\in\big(\frac\alpha{1+\alpha},1\big)\cup\big\{\frac12\big\}$.
Then, there exists $C>0$ such that,
\begin{itemize}
    \item
if $\beta\in\big(\frac\alpha{1+\alpha},\frac12\big]\cup\big\{\frac12\big\}$, assuming
\begin{equation*}
\kappa\gamma_\star>\zeta^{2+\lambda\frac{1+\alpha}\alpha}\Big(\frac{(1+\alpha)L^\frac1\alpha}\alpha\Big)^\lambda\alpha,
\quad\text{where}\quad
\lambda=\begin{cases}
\frac{1-2\beta}{\frac{1+\alpha}\alpha\beta-1},
&\text{if $\beta<\frac12$},\\
0,
&\text{if $\beta=\frac12$},
\end{cases}
\quad
\text{under $\mathcal{H}$\ref{asp:gamma:misc}(\ref{asp:gamma:misc:i})\hyperref[asp:gamma:misc:ib]{b},}
\end{equation*}
for all $\delta\in(0,1)$ and all $n\geq1$,
\begin{equation*}
F(\theta_n)-\inf F\leq\frac{C}\delta\gamma_n^\alpha
\quad\text{with probability at least $1-\delta$.}
\end{equation*}
    \item
if $\beta\in\big(\frac12,1\big)$, denoting $r=\frac1{2\beta-1}\wedge(\alpha\rho-1)>0$, for all $\delta\in(0,1)$ and all $n\geq1$,
\begin{equation*}
F(\theta_n)-\inf F\leq\frac{C}\delta\Big(\sum_{k=1}^n\gamma_k\Big)^{-r}
\quad\text{with probability at least $1-\delta$.}
\end{equation*}
\end{itemize}
\end{enumerate}
\end{theorem}

\begin{remark}
The convergence rate proof for globally {\L}ojasiewicz loss landscapes with $\beta\leq\frac12$ relies on obtaining a positive constant $\mu>0$ such that $\|\nabla F(\theta)\|^2\geq\mu(F(\theta)-\inf F)$, $\theta\in\mathbb{R}^m$. This is not possible if $\beta=\frac\alpha{1+\alpha}$ and $\alpha\in(0,1)$.
Indeed, within the framework of Theorem~\ref{thm:F-F*:special:cases}(\ref{thm:F-F*:special:cases:iii}), if $\beta=\frac\alpha{1+\alpha}$, $\alpha\in(0,1]$, then, by \eqref{eq:framework:3} and Lemma~\ref{lmm:Holder}(\ref{lmm:Holder:ii}),
\begin{equation*}
\beta L^{-\frac1\alpha}\|\nabla F(\theta)\|^\frac1\beta
\leq F(\theta)-\inf F
\leq\zeta^\frac1\beta\|\nabla F(\theta)\|^\frac1\beta,
\quad\theta\in\mathbb{R}^m.
\end{equation*}
Assuming $\beta L^{-\frac1\alpha}\leq\zeta^\frac1\beta$, there exists a function $\phi\colon\mathbb{R}^m\to\big[\beta L^{-\frac1\alpha},\zeta^\frac1\beta\big]$ such that $F(\theta)-\inf F=\phi(\theta)\|\nabla F(\theta)\|^\frac1\beta$, $\theta\in\mathbb{R}^m$.
But then,
\begin{equation*}
\mu\leq\inf_{\theta\in\mathbb{R}^m}\frac{\|\nabla F(\theta)\|^2}{F(\theta)-\inf F}
=
\frac1{\sup_{\theta\in\mathbb{R}^m}\phi(\theta)^\beta(F(\theta)-\inf F)^{1-2\beta}},
\end{equation*}
which, given the coercivity of $F$, is nonnull if and only if $\beta=\frac\alpha{1+\alpha}=\frac12$, i.e.~$\alpha=1$.
\end{remark}

\begin{remark}\label{rmk:optimal:speed}
For simplicity, denote $\beta$ the {\L}ojasiewicz exponent within the framework of either of Theorems~\ref{thm:F-F*:special:cases}(\ref{thm:F-F*:special:cases:i})\hyperref[thm:F-F*:special:cases:i:a]{a}, \hyperref[thm:F-F*:special:cases:i:b]{b}, (\ref{thm:F-F*:special:cases:ii}) or (\ref{thm:F-F*:special:cases:iii}). Suppose $\alpha=1$ and $\gamma_n=\Theta(\frac1{n^s})$, $s\in(0,1]$.
\begin{enumerate}[\bf(i)]
    \item\label{rmk:optimal:speed:i}
Within the frameworks of Theorems~\ref{thm:F-F*:special:cases}(\ref{thm:F-F*:special:cases:i})\hyperref[thm:F-F*:special:cases:i:a]{a}, \hyperref[thm:F-F*:special:cases:i:b]{b} or (\ref{thm:F-F*:special:cases:ii}),
if $\beta=\frac12$, the \eqref{eq:theta:n} convergence is in $\mathrm{O}(n^{-\frac{s}2})$ and attains its best rate in $\mathrm{O}(n^{-\frac12})$ when $s=1$.
If $\beta>\frac12$, the \eqref{eq:theta:n} convergence in $\mathrm{O}\big(n^{-(1-s)(\frac1{2\beta-1}\wedge\frac{\rho-1}2)}\big)$ if $s<1$, and in $\mathrm{O}\big((\ln{n})^{-\frac1{2\beta-1}\wedge\frac{\rho-1}2}\big)$ if $s=1$, achieves its best rate if $\rho=\frac{1+2\beta}{2\beta-1}$ and $s=\frac\rho{1+\rho}=\frac{1+2\beta}{4\beta}$, scoring an $\mathrm{O}(n^{-\frac1{4\beta}})$. Note that this convergence rate coincides with that of $\beta=\frac12$ when $\beta\downarrow\frac12$, and approaches an $\mathrm{O}(n^{-\frac14})$ when $\beta\uparrow1$.
    \item\label{rmk:optimal:speed:ii}
As for the globally {\L}ojasiewicz framework of Thorem~\ref{thm:F-F*:special:cases}(\ref{thm:F-F*:special:cases:iii}), the \eqref{eq:theta:n} convergence in $\mathrm{O}(n^{-s})$ when $\beta=\frac12$ scores its best rate for $s=1$, yielding an $\mathrm{O}(n^{-1})$. When $\beta>\frac12$, it is in $\mathrm{O}\big(n^{-(1-s)(\frac1{2\beta-1}\wedge(\rho-1))}\big)$ if $s<1$ and in $\mathrm{O}\big((\ln{n})^{-\frac1{2\beta-1}\wedge(\rho-1)}\big)$ if $s=1$, which attains its best rate for $\rho=\frac{2\beta}{2\beta-1}$ and $s=\frac{2\beta}{4\beta-1}$, resulting in an $\mathrm{O}(n^{-\frac1{4\beta-1}})$.
This convergence rate is of the same order as the one for $\beta=\frac12$ if $\beta\downarrow\frac12$, and reaches an $\mathrm{O}(n^{-\frac13})$ if $\beta\uparrow1$.
\end{enumerate}
\end{remark}

Let us study the computational complexity of \eqref{eq:theta:n}.
Define the reciprocal $\gamma^{-1}\colon\mathbb{R}^*_+\to\mathbb{N}$ by
\begin{equation*}
\gamma^{-1}(t)=\inf\{n\geq1\colon\gamma_n\leq t\},
\quad t>0.
\end{equation*}
Note that $\gamma_{\gamma^{-1}(t)}\leq t$, $t>0$.
Define also
\begin{equation*}
\Sigma_n\coloneqq\sum_{k=1}^{n+1}\gamma_k,
\quad n\geq1.
\end{equation*}
$(\Sigma_n)_{n\geq1}$ is hence nondecreasing. Let $\Sigma^{-1}\colon\mathbb{R}^*_+\to\mathbb{N}$ be the mapping given by
\begin{equation*}
\Sigma^{-1}(t)=\inf\{n\geq1\colon\Sigma_n\geq t\},
\quad t>0.
\end{equation*}
Observe that $\Sigma_{\Sigma^{-1}(t)}\geq t$, $t>0$.

\begin{example}\label{exp:special:gamma}
If $(\gamma_n)_{n\geq1}$ is given by \eqref{eq:gamma:usecase}, with $\gamma_0=1$ and $c=0$, we have
\begin{equation*}
\gamma^{-1}(t)=\Big\lceil\frac1{t^\frac1s}\Big\rceil
\quad\text{and}\quad
\Sigma^{-1}(t)=
\begin{cases}
\lceil((1-s)t)^\frac1{1-s}\rceil,
&s\in(0,1),\\\
\lceil\mathrm{e}^{t-1}\rceil,
&s=1,
\end{cases}
\quad t>0.
\end{equation*}
\end{example}

\begin{corollary}\label{crl:cost}
Let $\varepsilon>0$ be a prescribed tolerance.
Then,
\begin{enumerate}[\bf(i)]
\item\label{crl:cost:i}
\begin{enumerate}[\bf a.]
    \item\label{crl:cost:ia}
    within the framework of either Theorem~\ref{thm:F-F*}(\ref{thm:F-F*:i})\hyperref[thm:F-F*:i:a]{a} or~\hyperref[thm:F-F*:i:b]{b}, by setting
    \begin{equation*}
    n\propto\Sigma^{-1}\big(\varepsilon^{-1\vee2\beta_\delta}\big),
    \end{equation*}
    one has $\inf_{0\leq k\leq n}F(\theta_k)-F_\star\leq C_\delta\varepsilon$ with probability at least $1-\delta$.
    \item
    within the framework of Theorem~\ref{thm:F-F*}(\ref{thm:F-F*:ii}), by setting
    \begin{equation*}
    n\propto\Sigma^{-1}\big(\varepsilon^{-1\vee2\beta}\big),
    \end{equation*}
    one has $\inf_{0\leq k\leq n}F(\theta_k)-F_\star\leq C_\delta\varepsilon$ with probability at least $1-\delta$.
    \item
    within the framework of Theorem~\ref{thm:F-F*}(\ref{thm:F-F*:iii}), by setting
    \begin{equation*}
    n\propto\Sigma^{-1}(\varepsilon^{-2\beta}),
    \end{equation*}
    one has $\inf_{0\leq k\leq n}F(\theta_k)-\inf F\leq \frac{C}{\delta^{1/\beta}}\varepsilon$ with probability at least $1-\delta$.
\end{enumerate}
\item\label{crl:cost:ii}
\begin{enumerate}[\bf a.]
    \item\label{crl:cost:iia}
    within the framework of either Theorem~\ref{thm:F-F*:special:cases}(\ref{thm:F-F*:special:cases:i})\hyperref[thm:F-F*:special:cases:i:a]{a} or~\hyperref[thm:F-F*:special:cases:i:b]{b}, by setting
    \begin{equation*}
    n\propto
    \begin{cases}
    \gamma^{-1}(\varepsilon^{2\vee\frac1\alpha}),
    &\text{if $\beta_\delta\in(0,\frac12]$,}\\
    \Sigma^{-1}(\varepsilon^{-\frac1{r_\delta}}),
    &\text{if $\beta_\delta\in(\frac12,1)$,}
    \end{cases}
    \end{equation*}
    one has $F(\theta_n)-F_\star\leq C_\delta\varepsilon$ with probability at least $1-\delta$.
    \item
    within the framework of Theorem~\ref{thm:F-F*:special:cases}(\ref{thm:F-F*:special:cases:ii}), by setting
    \begin{equation*}
    n\propto
    \begin{cases}
    \gamma^{-1}(\varepsilon^{2\vee\frac1\alpha}),
    &\text{if $\beta\in(0,\frac12]$,}\\
    \Sigma^{-1}(\varepsilon^{-\frac1r}),
    &\text{if $\beta\in(\frac12,1]$,}
    \end{cases}
    \end{equation*}
    one has $F(\theta_n)-F_\star\leq C_\delta\varepsilon$ with probability at least $1-\delta$.
    \item
    within the framework of Theorem~\ref{thm:F-F*:special:cases}(\ref{thm:F-F*:special:cases:iii}), by setting
    \begin{equation*}
    n\propto
    \begin{cases}
    \gamma^{-1}(\varepsilon^{\frac1\alpha}),
    &\text{if $\beta\in(\frac\alpha{1+\alpha},\frac12]\cup\{\frac12\}$,}\\
    \Sigma^{-1}(\varepsilon^{-\frac1r}),
    &\text{if $\beta\in(\frac12,1)$,}
    \end{cases}
    \end{equation*}
    one has $F(\theta_n)-F_\star\leq C_\delta\varepsilon$ with probability at least $1-\delta$.
\end{enumerate}
\end{enumerate}
All in all, the \eqref{eq:theta:n} computational complexity has bound $\mathrm{Cost}\leq Cn$, for some constant $C>0$.
\end{corollary}

The convergence rates derived in Theorems~\ref{thm:F-F*:special:cases}(\ref{thm:F-F*:special:cases:i},\ref{thm:F-F*:special:cases:ii}) impose no condition on the initialization $\theta_0$ and no concentration on the stochastic gradients $(g_n(\theta_{n-1}))_{n\geq0}$ (we refer to~\cite{AL24,Sa+22,JO19,MDB24} for recent developments on \eqref{eq:theta:n} gradient concentration).
Our obtained results concern the function-value convergence speeds toward the random critical level $F_\star$ as described in Theorem~\ref{thm:F:cv:gradF:0}, and not to a predefined deterministic local minimum as in~\cite[Theorem~3.1]{AL24} or~\cite[Theorem~5.1]{Wei+24}.
This means in particular that the function value gaps $(F(\theta_n)-F_\star)_{n\geq0}$ that we deal with can in fact take negative values.
\\

Except for the condition on $\gamma_\star$ in the case of $\beta_\delta\in\big(0,\frac12\big]$ and $\mathcal{H}$\ref{asp:gamma:misc}(\ref{asp:gamma:misc:i})\hyperref[asp:gamma:misc:ib]{b} in Theorem~\ref{thm:F-F*:special:cases}(\ref{thm:F-F*:special:cases:i})\hyperref[thm:F-F*:special:cases:i:a]{a}, to derive a convergence rate at the confidence level $1-\delta$, $\delta\in(0,1)$, our settings are $\delta$-free.
Note that the extra constraint on $\gamma_\star$ in Theorem~\ref{thm:F-F*:special:cases} when the {\L}ojasiewicz exponent lies in $\big(0,\frac12\big]$ and $\mathcal{H}$\ref{asp:gamma:misc}(\ref{asp:gamma:misc:i})\hyperref[asp:gamma:misc:ib]{b} holds is classical in stochastic approximation literature~\cite{Duf96}.
\begin{remark}\label{rmk:advantage}
The advantage of Theorem~\ref{thm:F-F*:special:cases}(\ref{thm:F-F*:special:cases:i})\hyperref[thm:F-F*:special:cases:i:b]{b} over~(\ref{thm:F-F*:special:cases:i})\hyperref[thm:F-F*:special:cases:i:a]{a} is that the initialization condition $\kappa\gamma_\star>\frac12\vee\alpha$ for the case of $\beta_\delta\in\big(0,\frac12\big]$ and $\mathcal{H}$\ref{asp:gamma:misc}(\ref{asp:gamma:misc:i})\hyperref[asp:gamma:misc:ib]{b} becomes free from the knowledge of $\zeta_\delta^{1/\beta_\delta}$, bearing in mind that the exponent $\beta_\delta$ and the factor $C_\delta$ in the convergence rate bound differ. Such a framework will be useful in Section~\ref{ssec:global} to help extend global convergence to the case $\mathcal{H}$\ref{asp:gamma:misc}(\ref{asp:gamma:misc:i})\hyperref[asp:gamma:misc:ib]{b}.
\end{remark}

To obtain high probability convergence rates at the confidence level $1-\delta$, $\delta\in(0,1)$, \cite[Theorem~3.1]{AL24} sets the initialization $\theta_0$ in a small $\delta$-dependent neighborhood of a local minimum, and \cite[Theorem~5.1]{Wei+24} requires an initialization that is sufficiently close to a local minimum and bounds the learning rates by a small value dependent on $\delta$.
These constraints are designed to prevent \eqref{eq:theta:n} trajectories from escaping a small region of space where a {\L}ojasiewicz property holds.
They may however be restrictive as they require a priori knowledge on $\nabla F^{-1}(0)$, which is unavailable, and their settings are $\delta$-dependent.
In practice, an expensive computational overhead is expected to meet all the required conditions.
\\

For the sake of simplicity, let $\beta$ denote the {\L}ojasiewicz exponent irrespectively of the frameworks of Theorems~\ref{thm:F-F*:special:cases}(\ref{thm:F-F*:special:cases:i})\hyperref[thm:F-F*:special:cases:i:a]{a}, \hyperref[thm:F-F*:special:cases:i:b]{b} and (\ref{thm:F-F*:special:cases:ii}).
The shift in behavior, demarcated at $\beta=\frac12$, is already observed for GD and proximal methods~\cite{FGP15,AB09}. When $\beta=\frac12$, $F$ has a local quadratic growth~\cite[Lemma~2.1]{Liu+23}, hence promoting faster convergence rates. When $\beta>\frac12$, $F$ is locally subquadratic, so \eqref{eq:theta:n} is slower to converge.
The globally {\L}ojasiewicz case in Theorem~\ref{thm:F-F*:special:cases}(\ref{thm:F-F*:special:cases:iii}) displays faster convergence rates than Theorems~\ref{thm:F-F*:special:cases}(\ref{thm:F-F*:special:cases:i},\ref{thm:F-F*:special:cases:ii}), as it possesses a unique critical level -- its global minimum -- that acts as an attractor to \eqref{eq:theta:n} trajectories.
\\

\cite[Theorem~4.6(b)]{Lei+20} proves a convergence rate on the expected function-value gap $\mathbb{E}[F(\theta_n)-\inf F]$ in $\mathrm{O}(n^{-\alpha})$ for learning rates $\gamma_n=\Theta\big(\frac1n\big)$, assuming $(L,\alpha)$-H\"older differentiability of $\theta\in\mathbb{R}^m\mapsto f(x,\theta)$, $x\in\mathbb{R}^d$, and a global P{\L} condition on $F$ ($\beta=\frac12$).
This result is comparable to Theorem~\ref{thm:F-F*:special:cases}(\ref{thm:F-F*:special:cases:iii}) for $\beta=\frac12$ and a learning rate $\gamma_n=\Theta\big(\frac1n\big)$.

\cite[Proposition~4.3]{DK24} deals with the case where the {\L}ojasiewicz exponent $\beta$ is in $\big(\frac12,1\big)$ and presents a convergence rate in expectation that is conditional to the existence of deterministic constraints on the stochastic gradients' conditional moments \eqref{eq:moment:condition}, and the pertaining of the iterates to a predefined compact set.
The proposed convergence rate is in $\mathrm{O}\big(\big(\sum_{k=1}^n\gamma_k\big)^{-\frac1{2\beta-1}}\big)$. If $\gamma_n=\Theta\big(\frac1{n^s}\big)$, $s\in(0,1]$, \cite[Remark~4.1]{DK24} recommends choosing $s$ close enough to $1$, lower bounded by parameters governing \eqref{eq:moment:condition}. Our assumptions in contrast are free from any deterministic control on the stochastic gradients.

Assuming the stochastic gradients to be unbiased and Gaussianly concentrated, \cite[Theorem~3.1]{AL24} derives an exponential convergence rate in high probability for locally Lipschitz-differentiable loss landscapes, conditionally to all the iterates pertaining to a closed ball where a local P{\L} condition ($\beta=\frac12$) holds and to $\theta_0$ being initialized in a small neighborhood of the local minimum. Our setting avoids any constraint initialization and is available for a wider range of {\L}ojasiewicz exponents $\beta$.

\cite[Theorem~5.1]{Wei+24} presents a high probability convergence rate for unbiased \eqref{eq:theta:n} in the case of locally Lipschitz-differentiable loss landscapes, under a local {\L}ojasiewicz condition with $\beta\in\big[\frac12,1\big]$. This result is conditional to an initialization that is sufficiently close to a local minimum and is obtained for learning rates such that $\gamma_n=\Theta\big(\frac1{n^s}\big)$, $s\in(0,1)$, bounded by a small value. The provided rate is unavailable for $s=1$.
Upon closer inspection, their result is likely erroneous: in~\cite[Lemma~F.4]{Wei+24}, the inequality ``$D_{n+1}\mathbbm1_{\Omega_{n+1}}\leq D_{n+1}\mathbbm1_{\Omega_n}$'' is wrong, as it requires $D_{n+1}\mathbbm1_{\Omega_n}\geq0$, which is not guaranteed therein. Using our notation, for a small neighborhood $\mathcal{V}$ around a local minimum $\theta_\star\in\nabla F^{-1}(0)$ such that $F(\theta)\geq F(\theta_\star)$, $\theta\in\mathcal{V}$, the aforementioned inequality reads $(F(\theta_{n+1})-F(\theta_\star))\prod_{k=1}^{n+1}\mathbbm1_{\theta_k\in\mathcal{V}}\leq(F(\theta_{n+1})-F(\theta_\star))\prod_{k=1}^n\mathbbm1_{\theta_k\in\mathcal{V}}$, which is false since the nonnegativity of $(F(\theta_{n+1})-F(\theta_\star))\prod_{k=1}^n\mathbbm1_{\theta_k\in\mathcal{V}}$ is not guaranteed.
\cite[Theorem~4.1]{Wei+24} retrieves a convergence rate for Lipschitz-differentiable globally {\L}ojasiewicz loss landscapes with $\beta\in\big[\frac12,1\big]$ that attains optimally an $\mathrm{o}(n^{-\eta})$, with $\eta<\frac1{4\beta-1}$.
The optimal convergence rate we recover in Remark~\ref{rmk:optimal:speed}(\ref{rmk:optimal:speed:ii}) is sharper and is exactly in $\mathrm{O}(n^{-\frac1{4\beta-1}})$.
Besides, as in the locally {\L}ojasiewicz case, the provided rate in \cite{Wei+24} is unavailable for $s=1$.

\subsection{Global Convergence Analysis}
\label{ssec:global}

The goal of this section is to prove the $\Pas$ convergence of the iterates $(\theta_n)_{n\geq0}$ to a single-point limit.
The proofs of this section are gathered in Section~\ref{sec:global}.

To simplify matters, we assume
\begin{assumption}\label{asp:gamma=1/n^s}
$\gamma_n=\Theta\big(\frac1{n^s}\big)$, $s\in(0,1]$.
\end{assumption}

\begin{theorem}\label{thm:iterates}
Assume $\mathcal{H}$\ref{asp:f}--\ref{asp:gamma=1/n^s} hold and that $\rho>3\vee\frac2\alpha$, $s\in\big[\frac\rho{1+\rho},1\big]$ and $\kappa\gamma_\star>\frac12\vee\alpha$ if $s=1$.
Then $(\theta_n)_{n\geq0}$ converges $\Pas$ to an $\mathbb{R}^m$-valued random variable $\theta_\star$ that is $\Pas$ a stationary point of $F$ and satisfies $F(\theta_\star)=F_\star$ $\Pas$.
\end{theorem}

\begin{remark}\label{rmk:s>}
It ensues from $\rho>3\vee\frac2\alpha$ and $s\in\big[\frac\rho{1+\rho},1\big]$ that $s\in\big(\frac34\vee\frac2{2+\alpha},1\big]$.
\end{remark}

Recalling Lemma~\ref{lmm:F(theta):converges}(\ref{lmm:F(theta):converges:iv}), a quick glance at the decomposition \eqref{eq:theta:n=sum} reveals that the $\Pas$-summability of $(\gamma_n\|b(\theta_{n-1})\|)_{n\geq1}$ translates directly into the global convergence of \eqref{eq:theta:n}.
Proving this summability has thus been the focus of most recent literature in this direction~\cite{DK24,CFR23,QMM24}.
Theorem~\ref{thm:iterates} shows that, by choosing $\rho$ large enough and $\gamma_\star$ as well if $s=1$, then global convergence is ascertained for all $s\in\big[\frac\rho{1+\rho},1\big]$.
\\

\cite[Theorem~1.2]{DK24} proves global convergence for locally Lipschitz-differentiable loss landscapes, with {\L}ojasiewicz exponent $\beta\in\big[\frac12,1\big)$, however under the additional condition of existence of deterministic constraints on stochastic gradients' conditional moments \eqref{eq:moment:condition} for $q\geq2$ and $v_n=\mathrm{o}(\sqrt{n})$.
We do not need any deterministic control on the stochastic gradients for our global convergence result, and simply require $\mathcal{H}$\ref{asp:f:abc}(\ref{asp:f:abc:i}).

\cite[Proposition~6.3]{CFR23} considers a Lipschitz-differentiable objective function verifying a K{\L} property on a compact set, and shows iterate convergence for \eqref{eq:theta:n} trajectories converging in this set.
Nonetheless, this result relies on assumptions that may be difficult to adhere to.
Firstly, it is supposed that $(V(\theta_n))_{n\geq0}$ asymptotically never registers a dropdown below its target $V_\star$, i.e.~$\mathbb{P}(\liminf_{n\to\infty}\{V(\theta_n)>V_\star\})=1$.
Secondly, it is assumed that the information revelation on $F_\star$ is lower-bounded by $F_\star$ itself, i.e.~$\mathbb{E}[F_\star|\mathcal{F}_n]\geq F_\star$ $\Pas$.
Thirdly, it is presumed that $\|\nabla F(\theta_n)\|>0$, $n\geq1$, $\Pas$.
Finally, it is required that a strong growth condition holds, i.e.~$\mathbb{E}_{n-1}[\|g_n(\theta_{n-1})\|^2]\leq C\|\nabla F(\theta_{n-1})\|^2$, $n\geq1$, with $C>0$.

\cite[Theorem~5.1 \& Corollary~5.2]{QMM24} look into an \eqref{eq:theta:n} variant with momentum, a technique known to help with iterate convergence. These results suppose however the uniform boundedness of the conditional stochastic gradient variances~\cite[Assumption~2.1]{QMM24}.

Without supposing any {\L}ojasiewicz-type assumption, \cite[Theorem~2]{PZT22} attempts to prove the global convergence of unbiased SGD for locally H\"older-differentiable loss landscapes, conditionally on the boundedness of the iterates.
Unfortunately, the provided proof is erroneous.
It relies on \cite[Theorem~1]{PZT22}, whose proof involves the following statement (c.f.~the passage from Equation~(120) to Equation~(121) therein):
\begin{align}
&\mathbb{P}\big(\|\theta_{n+1}-\bar\theta\|>R+\eta,\,\|\theta_n-\bar\theta\|<R,\,\mathrm{i.o.}\big)=0,
\quad R,\eta>0,\label{eq:false:0}\\
&\Rightarrow\quad
\mathbb{P}\big(\limsup_{n\to\infty}\|\theta_n-\bar\theta\|>\liminf_{n\to\infty}\|\theta_n-\bar\theta\|\big)=0,\label{eq:false:1}
\end{align}
where $\bar\theta\in\mathbb{R}^m$.
Denote $u_n=\|\theta_n-\bar\theta\|$, $n\geq1$.
The above implication is easily refutable, as there are plenty of sequences for which \eqref{eq:false:0} is true for any $R,\eta>0$ and yet \eqref{eq:false:1} is false, least of which is the deterministic sequence $u_n=\sin(\pi\sqrt{n})$, $n\geq0$, that could occur in a cycle limit case.
Indeed, given that $x\in\mathbb{R}\mapsto\sin(x)$ is $1$-Lipschitz, we have $|u_{n+1}-u_n|=|\sin(\pi\sqrt{n+1})-\sin(\pi\sqrt{n})|\leq\pi(\sqrt{n+1}-\sqrt{n})\sim\frac\pi{2\sqrt{n}}\to0$, so that the statement $(u_{n+1}>R+\eta,\,u_n<R,\,\mathrm{i.o.})$ is false for any $R,\eta>0$.
However, if $n=p^2$ is a perfect square, then $u_{p^2}=\sin(\pi p)=(-1)^p$, so that $(1=\limsup_{n\to\infty}u_n>\liminf_{n\to\infty}u_n=-1)$ is true.

\begin{corollary}\label{crl:iterates:speed}
Within the framework of Theorem~\ref{thm:iterates},
\begin{enumerate}[\bf(i)]
\item\label{crl:iterates:speed:i}
for all $\delta\in(0,1)$, there exists $\ell_\delta>\inf F$ such that $\mathbb{P}(\mathcal{A}\subset\mathcal{L}_{\ell_\delta})\geq1-\delta$, and
\begin{enumerate}[\bf a.]
    \item\label{crl:iterates:speed:i:a}
    there exist $\beta_\delta\in(0,1)$ and $\zeta_\delta>0$ such that \eqref{eq:framework:1} holds for $\mathcal{K}=\mathcal{L}_{\ell_\delta}$.
    \item\label{crl:iterates:speed:i:b}
    there exists $\beta_\delta\in(0,1)$ such that \eqref{eq:framework:1:bis} holds for $\mathcal{K}=\mathcal{L}_{\ell_\delta}$.
\end{enumerate}
Besides, there exists $C_\delta>0$ such that,
\begin{itemize}
    \item
if $\beta_\delta\in\big(0,\frac12\big]$, by setting $\sigma\in\big[\frac1{2(2s-1)}\vee\frac{\alpha s}{2((1+\alpha)s-1)}\vee\frac{1-((1-\alpha)\vee\frac12)s}{2(\alpha\wedge\frac12)s},1\big)$ and $\kappa\gamma_\star>\zeta_\delta^{1/\beta_\delta}(\alpha\vee\frac12)$ if $s=1$ and (\ref{crl:iterates:speed:i})\hyperref[crl:iterates:speed:i:a]{a} holds, then, for all $n\geq1$,
\begin{equation*}
\|\theta_n-\theta_\star\|\leq\frac{C_\delta}{1-\sigma}n^{-((\alpha\wedge\frac12)s\wedge\frac{2s-1}2\wedge((1+\alpha)s-1))(1-\sigma)}
\quad\text{with probability at least $1-\delta$.}
\end{equation*}
    \item
if $\beta_\delta\in\big(\frac12,1\big)$ and $s\in\big[\frac\rho{1+\rho},1\big)$, by setting $\sigma\in\big[\frac1{2(2s-1)}\vee\frac{\alpha s}{2((1+\alpha)s-1)}\vee\frac{1+r_\delta}{2r_\delta},1\big)$, then, for all $n\geq1$,
\begin{equation*}
\|\theta_n-\theta_\star\|\leq\frac{C_\delta}{1-\sigma}n^{-(r_\delta(1-s)\wedge\frac{2s-1}2\wedge((1+\alpha)s-1))(1-\sigma)}
\quad\text{with probability at least $1-\delta$,}
\end{equation*}
recalling that $r_\delta=\frac1{2\beta_\delta-1}\wedge(\alpha\rho-1)\wedge\frac{\rho-1}2>1$.
    \item
if $\beta_\delta\in\big(\frac12,1\big)$ and $s=1$, by setting $\sigma\in\big[\frac{1+r_\delta}{2r_\delta},1\big)$, then, for all $n\geq1$,
\begin{equation*}
\|\theta_n-\theta_\star\|\leq\frac{C_\delta}{1-\sigma}(\ln{n})^{-r_\delta(1-\sigma)}
\quad\text{with probability at least $1-\delta$,}
\end{equation*}
recalling that $r_\delta=\frac1{2\beta_\delta-1}\wedge(\alpha\rho-1)\wedge\frac{\rho-1}2>1$.
\end{itemize}
\item\label{crl:iterates:speed:ii}
if $F$ is globally {\L}ojasiewicz, there exist $\beta\in\big[\frac\alpha{1+\alpha},1)$ and $\zeta>0$ such that \eqref{eq:framework:3} holds.
Suppose $\beta\in\big(\frac\alpha{1+\alpha},1\big)\cup\big\{\frac12\big\}$.
Then, there exists $C>0$ such that,
\begin{itemize}
    \item
if $\beta\in\big(\frac\alpha{1+\alpha},\frac12\big]\cup\big\{\frac12\big\}$, by setting $\sigma\in\big[\frac1{2(2s-1)}\vee\frac{\alpha s}{2((1+\alpha)s-1)},1\big)$ and, if $s=1$,
\begin{equation*}
\kappa\gamma_\star>\zeta^{2+\lambda\frac{1+\alpha}\alpha}\Big(\frac{(1+\alpha)L^\frac1\alpha}\alpha\Big)^\lambda\alpha,
\quad\text{where}\quad
\lambda=\begin{cases}
\frac{1-2\beta}{\frac{1+\alpha}\alpha\beta-1},
&\text{if $\beta<\frac12$},\\
0,
&\text{if $\beta=\frac12$},
\end{cases}
\end{equation*}
for all $n\geq1$,
\begin{equation*}
\|\theta_n-\theta_\star\|\leq\frac{C}{\delta(1-\sigma)}n^{-(\alpha s\wedge\frac{2s-1}2\wedge((1+\alpha)s-1))(1-\sigma)}
\quad\text{with probability at least $1-\delta$.}
\end{equation*}
    \item
if $\beta\in\big(\frac12,1\big)$ and $s\in\big[\frac\rho{1+\rho},1\big)$, by setting $\sigma\in\big[\frac1{2(2s-1)}\vee\frac{\alpha s}{2((1+\alpha)s-1)}\vee\frac{1+r}{2r},1\big)$, for all $n\geq1$,
\begin{equation*}
\|\theta_n-\theta_\star\|\leq\frac{C}{\delta(1-\sigma)}n^{-(r(1-s)\wedge\frac{2s-1}2\wedge((1+\alpha)s-1))(1-\sigma)}
\quad\text{with probability at least $1-\delta$,}
\end{equation*}
recalling that $r=\frac1{2\beta-1}\wedge(\alpha\rho-1)>1$.
    \item
if $\beta\in\big(\frac12,1\big)$ and $s=1$, by setting $\sigma\in\big[\frac{1+r}{2r},1\big)$, for all $n\geq1$,
\begin{equation*}
\|\theta_n-\theta_\star\|\leq\frac{C}{\delta(1-\sigma)}(\ln{n})^{-r(1-\sigma)}
\quad\text{with probability at least $1-\delta$,}
\end{equation*}
recalling that $r=\frac1{2\beta-1}\wedge(\alpha\rho-1)>1$.
\end{itemize}
\end{enumerate}
Overall, the optimal convergence rates are obtained by taking the smallest possible parameter $\sigma$.
\end{corollary}

\begin{remark}
Let $(\beta,\zeta)\in(0,1)\times\mathbb{R}^*_+$ denote the {\L}ojasiewicz parameters regardless of the framework of Corollary~\ref{crl:iterates:speed}(\ref{crl:iterates:speed:i})\hyperref[crl:iterates:speed:i:a]{a}, \hyperref[crl:iterates:speed:i:b]{b} or (\ref{crl:iterates:speed:ii}).
Note that, by $\mathcal{H}\ref{asp:f}$ and Proposition~\ref{prp:Lojasiewicz}, one has asymptotically
\begin{align*}
\|\nabla F(\theta_n)\|
&=\|\nabla F(\theta_n)-\nabla F(\theta_\star)\|
\leq L\|\theta_n-\theta_\star\|^\alpha
\quad\Pas,\\
F(\theta_n)-F_\star
&\leq\zeta^\frac1\beta\|\nabla F(\theta_n)\|^\frac1\beta
\leq L^\frac1\beta\zeta^\frac1\beta\|\theta_n-\theta_\star\|^\frac\alpha\beta\quad\Pas.
\end{align*}
This offers sharper gradient convergence rates than Lemma~\ref{lmm:grad:F:n} and alternative, more pessimistic, function-value convergence rates to Theorem~\ref{thm:F-F*:special:cases}.
\end{remark}

\begin{corollary}\label{crl:iter:cost}
Let $\varepsilon>0$ be a prescribed tolerance.
Then,
\begin{enumerate}[\bf(i)]
\item
    within the framework of either Corollary~\ref{crl:iterates:speed}(\ref{crl:iterates:speed:i})\hyperref[crl:iterates:speed:i:a]{a} or~\hyperref[crl:iterates:speed:i:b]{b}, by setting
    \begin{equation*}
    n\propto
    \begin{cases}
    \lceil\varepsilon^{\frac{-1}{1-\sigma}(\frac{2\vee\alpha^{-1}}s\vee\frac2{2s-1}\vee\frac1{(1+\alpha)s-1})}\rceil,
    &\text{if $\beta_\delta\in(0,\frac12]$,}\\
    \lceil\varepsilon^{\frac{-1}{1-\sigma}(\frac1{r_\delta(1-s)}\vee\frac2{2s-1}\vee\frac1{(1+\alpha)s-1})}\rceil,
    &\text{if $\beta_\delta\in(\frac12,1)$ and $s<1$,}\\
    \lceil\exp(\varepsilon^{-\frac1{r_\delta(1-\sigma)}})\rceil
    &\text{if $\beta_\delta\in(\frac12,1)$ and $s=1$,}
    \end{cases}
    \end{equation*}
    one has $\|\theta_n-\theta_\star\|\leq\frac{C_\delta}{1-\sigma}\varepsilon$ with probability at least $1-\delta$.
\item
    within the framework of Corollary~\ref{crl:iterates:speed}(\ref{crl:iterates:speed:ii}), by setting
    \begin{equation*}
    n\propto
    \begin{cases}
    \lceil\varepsilon^{\frac{-1}{1-\sigma}(\frac1{\alpha s}\vee\frac2{2s-1}\vee\frac1{(1+\alpha)s-1})}\rceil,
    &\text{if $\beta\in(0,\frac12]$,}\\
    \lceil\varepsilon^{\frac{-1}{1-\sigma}(\frac1{r(1-s)}\vee\frac2{2s-1}\vee\frac1{(1+\alpha)s-1})}\rceil,
    &\text{if $\beta\in(\frac12,1)$ and $s<1$,}\\
    \lceil\exp(\varepsilon^{-\frac1{r(1-\sigma)}})\rceil
    &\text{if $\beta\in(\frac12,1)$ and $s=1$,}
    \end{cases}
    \end{equation*}
    one has $\|\theta_n-\theta_\star\|\leq\frac{C}{\delta(1-\sigma)}\varepsilon$ with probability at least $1-\delta$.
\end{enumerate}
Altogether, the subsequent \eqref{eq:theta:n} computational cost has bound: $\mathrm{Cost}\leq Cn$, for some constant $C>0$.
Moreover, the optimal computational cost is recovered by selecting the smallest possible parameter $\sigma$.
\end{corollary}

\cite{Liu+23} takes a closer look at overparametrized neural networks satisfying a local quadratic growth (reminiscent of a P{\L} condition, i.e.~$\beta=\frac12$), as well as an interpolation and an aiming properties around their global minima.
The aiming property predicates that the loss landscape's gradients locally follow the direction of the global optimum.
The interpolation assumption stipulates the existence of a deterministic $\theta_\star\in\mathbb{R}^m$ minimizing simultaneously $F$ and $\theta\in\mathbb{R}^m\mapsto f(x,\theta)$ for any $x\in\mathbb{R}^d$.
\cite{Liu+23} then derives exponential convergence rates in expectation and in high probability for \eqref{eq:theta:n} iterates conditionally to an initialization that is sufficiently close to the global minimum.
The proposed result quantifies rather the distance of the iterates to a subset of the global minimizers. It imposes an initialization in a $\delta$-dependent radius around the global optimum to achieve a convergence rate with confidence $1-\delta$. The smaller the $\delta$, the closer the initialization must be to the global optimum.

\begin{remark}\label{rmk:interpolation}
Note that the interpolation condition, claimed to hold for overparametrized neural networks, fails notwithstanding to encompass some elementary \eqref{eq:theta:n} problems.
Take the classical problem of approximating $\mathbb{E}[Y|X]$ by a neural network $\mathcal{N}_\theta(X)$ parametrized by $\theta$. The optimal $\theta_\star$ can be retrieved as solution to $\min_\theta\mathbb{E}[(Y-\mathcal{N}_\theta(X))^2]$. But if $Y$ is not $X$-measurable, no matter how overparametrized $\mathcal{N}_\theta$ could get, canceling $Y-\mathcal{N}_\theta(X)$ almost surely with some deterministic $\theta_\star$, hence perfectly interpolating $Y$ with $\mathcal{N}_{\theta_\star}(X)$, implies that $Y$ is $X$-measurable, which is absurd.
\end{remark}

\cite[Theorem~5(iii)]{GP23} exhibits $L^{2p}(\mathbb{P})$-convergence rates for \eqref{eq:theta:n} iterates, $p\geq1$, for loss landscapes that are of class $C^2$, Lipschitz-differentiable, globally {\L}ojasiewicz with {\L}ojasiewicz exponent $\beta\in\big[0,\frac12\big]$ and possess a unique minimizer (by Proposition~\ref{prp:Lojasiewicz}(\ref{prp:Lojasiewicz:iii}), necessarily, $\beta=\frac12$). The proposed result requires however that the stochastic gradient noises have uniformly bounded conditional polynomial-exponential moments~\cite[Assumption~$\mathrm{H}^\phi_{\overline\Sigma_p}$]{GP23}.

\section{Convergence Analysis Proofs}
\label{sec:cv}

Henceforth, $C>0$ designates a constant that may change from line to line, but is independent of all variables of the problem at hand. Given a variable $\delta\in\mathbb{R}$, $C_\delta>0$ denotes one such constant that may only depend on $\delta$.
\\

The $\Pas$ convergence of $(F(\theta_n))_{n\geq0}$ is obtained by applying the Siegmund-Robbins lemma~\cite{RS71} to the quasisupermartingale property \eqref{eq:quasi:super:martingale}.
The weak convergence of \eqref{eq:theta:n} expands on some arguments used in~\cite[Theorem~2(c)]{Lei+20} for $L^1(\mathbb{P})$ vanishing of $(\nabla F(\theta_n))_{n\geq0}$ when $\alpha=1$.
It builds upon the premise that it is impossible to simultaneously fulfill $\|\nabla F(\theta_n)-\nabla F(\theta_{n-1})\|\to0$ as $n\to\infty$ and $\sum_{n=1}^\infty\gamma_n\|\nabla F(\theta_{n-1})\|^2<\infty$, unless $\limsup_{n\to\infty}\|\nabla F(\theta_n)\|=0$ $\Pas$.

\begin{lemma}\label{lmm:F(theta):converges}
Under $\mathcal{H}$\ref{asp:f}--\ref{asp:gamma},
\begin{enumerate}[\bf(i)]
  \item\label{lmm:F(theta):converges:i}
$(F(\theta_n))_{n\geq0}$ converges $\Pas$ to a real-valued random variable $F_\star$, ${\sum_{n=1}^\infty\gamma_n\|\nabla F(\theta_{n-1})\|^2<\infty}$ $\Pas$ and $\liminf_{n\to\infty}\|\nabla F(\theta_n)\|=0$ $\Pas$;
  \item\label{lmm:F(theta):converges:ii}
$(\mathbb{E}[F(\theta_n)])_{n\geq0}$ converges, $\sum_{n=1}^\infty\gamma_n\mathbb{E}[\|\nabla F(\theta_{n-1})\|^2]<\infty$ and ${\liminf_{n\to\infty}\mathbb{E}[\|\nabla F(\theta_n)\|^2]=0}$;
  \item\label{lmm:F(theta):converges:iii}
$\sum_{n=1}^\infty\|\theta_n-\theta_{n-1}\|^{1+\alpha}$ is in $L^1(\mathbb{P})$ and is $\Pas$-finite;
  \item\label{lmm:F(theta):converges:iv}
the martingale $\big(\sum_{k=1}^n\gamma_k(g_k(\theta_{k-1})-b(\theta_{k-1}))\big)_{n\geq1}$ converges in $L^2(\mathbb{P})$ and $\Pas$;
  \item\label{lmm:F(theta):converges:v}
$F_\star\in L^1(\mathbb{P})$.
\end{enumerate}
\end{lemma}

\begin{remark}
\begin{enumerate}[\bf(i)]
    \item
By telescopic summation, the triangle inequality and H\"older's inequality with the adjoint exponents $\big(1+\alpha,\frac{1+\alpha}\alpha\big)$,
\begin{equation*}
\|\theta_n\|
\leq\|\theta_0\|+\Big(\sum_{k=1}^n\|\theta_k-\theta_{k-1}\|^{1+\alpha}\Big)^\frac1{1+\alpha}n^\frac\alpha{1+\alpha}
\leq\|\theta_0\|+Yn^\frac\alpha{1+\alpha},
\quad n\geq1,
\end{equation*}
where $Y\coloneqq\big(\sum_{k=1}^\infty\|\theta_k-\theta_{k-1}\|^{1+\alpha}\big)^\frac1{1+\alpha}\in L^1(\mathbb{P})$ by Jensen's inequality and Lemma~\ref{lmm:F(theta):converges}(\ref{lmm:F(theta):converges:iii}).
This upper bound does not exclude that $(\theta_n)_{n\geq0}$ may escape to $\infty$, at most at an $\mathrm{O}_{L^1(\mathbb{P})}(n^\frac\alpha{1+\alpha})$ speed.
This can occur for loss landscapes where the minimum is only attained at infinity.
Consequently, recalling that $\theta_0\in L^2(\mathbb{P})$,
\begin{equation*}
\mathbb{E}[\|\theta_n\|]\leq\mathbb{E}[\|\theta_0\|]+\mathbb{E}[Y]n^\frac\alpha{1+\alpha},
\quad n\geq1.
\end{equation*}
    \item
By the \eqref{eq:theta:n} update rule and \eqref{eq:b},
\begin{equation*}
\mathbb{E}[\|\theta_n-\theta_{n-1}\|^2]
=\gamma_n^2\mathbb{E}\big[\|g_n(\theta_{n-1})-b(\theta_{n-1})\|^2\big]
+\gamma_n^2\mathbb{E}\big[\|b(\theta_{n-1})\|^2\big],
\quad n\geq1.
\end{equation*}
Using that, for some $C>0$,
\begin{equation*}
\sup_{n\geq1}\mathbb{E}[\|g_n(\theta_{n-1})-b(\theta_{n-1})\|^2]\leq C(1+\sup_{n\geq1}\mathbb{E}[F(\theta_{n-1})-\inf F])<\infty
\end{equation*}
(by $\mathcal{H}$\ref{asp:f}, $\mathcal{H}$\ref{asp:f:abc}(\ref{asp:f:abc:i},\ref{asp:f:abc:iii}), \eqref{eq:gradF:square} and Lemma~\ref{lmm:F(theta):converges}(\ref{lmm:F(theta):converges:ii})),
and that $\sum_{n=1}^\infty\gamma_n^2\mathbb{E}[\|\nabla F(\theta_{n-1})\|^2]<\infty$ (by Lemma~\ref{lmm:F(theta):converges}(\ref{lmm:F(theta):converges:ii}) and $\mathcal{H}$\ref{asp:gamma}),
we obtain a weaker alternative to Lemma~\ref{lmm:F(theta):converges}(\ref{lmm:F(theta):converges:iii}) predicating that $\sum_{n=1}^\infty\|\theta_n-\theta_{n-1}\|^2$ is in $L^1(\mathbb{P})$ and $\Pas$-finite (by $\mathcal{H}$\ref{asp:gamma} and Fubini's theorem).
\end{enumerate}
\end{remark}

\begin{proof}[Proof of Lemma~\ref{lmm:F(theta):converges}]
Let $n\geq1$.
\\

\noindent{\bf(\ref{lmm:F(theta):converges:i})}\
By $\mathcal{H}$\ref{asp:f}, Lemma~\ref{lmm:Holder}(\ref{lmm:Holder:i}) and the \eqref{eq:theta:n} recursion,
\begin{equation}\label{eq:F(theta:n)<F(theta:n-1)}
F(\theta_n)
=F\big(\theta_{n-1}-\gamma_n g_n(\theta_{n-1})\big)
\leq F(\theta_{n-1})-\gamma_n\nabla F(\theta_{n-1})^\top g_n(\theta_{n-1})+\frac{L}{1+\alpha}\gamma_n^{1+\alpha}\| g_n(\theta_{n-1})\|^{1+\alpha}.
\end{equation}
Since $\inf F>-\infty$, by $\mathcal{H}$\ref{asp:f:abc}(\ref{asp:f:abc:ii}),
\begin{equation}\label{eq:quasi:super:martingale}
\mathbb{E}_{n-1}[F(\theta_n)-\inf F]
\leq F(\theta_{n-1})-\inf F-\kappa\gamma_n\|\nabla F(\theta_{n-1})\|^2+\frac{L}{1+\alpha}\gamma_n^{1+\alpha}\mathbb{E}_{n-1}\big[\| g_n(\theta_{n-1})\|^{1+\alpha}\big].
\end{equation}
Using that $x^{1+\alpha}\leq C(1+x^2)$, $x\geq0$, and $\mathcal{H}$\ref{asp:f:abc},
\begin{equation}
\label{eq:E:n-1:fn}
\mathbb{E}_{n-1}\big[\| g_n(\theta_{n-1})\|^{1+\alpha}\big]
\leq C\big(1+\mathbb{E}_{n-1}\big[\|g_n(\theta_{n-1})\|^2\big]\big)
\leq C\big((F(\theta_{n-1})-\inf F)+1\big).
\end{equation}
Thus
\begin{equation}
\label{eq:SR:n-1}
\mathbb{E}_{n-1}[F(\theta_n)-\inf F]
\leq\big(1+C\gamma^{1+\alpha}_n\big)\big(F(\theta_{n-1})-\inf F\big)-\kappa\gamma_n\|\nabla F(\theta_{n-1})\|^2+C\gamma_n^{1+\alpha}.
\end{equation}
Recalling $\mathcal{H}$\ref{asp:gamma}, the Siegmund-Robbins lemma~\cite{RS71} guarantees that $\sum_{n=1}^\infty\gamma_n\|\nabla F(\theta_{n-1})\|^2<\infty$ $\Pas$ and that $(F(\theta_n))_{n\geq0}$ converges $\Pas$ to a finite real-valued random variable $F_\star$.

Furthermore, one has $\big(\sum_{n=1}^\infty\gamma_n\|\nabla F(\theta_{n-1})\|^2\big)\mathbbm1_{\liminf_{n\to\infty}\|\nabla F(\theta_n)\|>0}<\infty$ $\Pas$. Given that $\sum_{n=1}^\infty\gamma_n=\infty$ under $\mathcal{H}$\ref{asp:gamma}, it ensues that $\mathbb{P}(\liminf_{n\to\infty}\|\nabla F(\theta_n)\|>0)=0$.
\\

\noindent{\bf(\ref{lmm:F(theta):converges:ii})}\
Taking the expectation in \eqref{eq:SR:n-1},
\begin{equation}
\label{eq:SR:0}
\mathbb{E}[F(\theta_n)-\inf F]
\leq\big(1+C\gamma^{1+\alpha}_n\big)\mathbb{E}[F(\theta_{n-1})-\inf F]-\kappa\gamma_n\mathbb{E}\big[\|\nabla F(\theta_{n-1})\|^2\big]+C\gamma_n^{1+\alpha}.
\end{equation}
By~\cite[Lemma~5.31]{BC17}, $(\mathbb{E}[F(\theta_n)])_{n\geq0}$ is convergent and ${\sum_{n=1}^\infty\gamma_n\mathbb{E}[\|\nabla F(\theta_{n-1})\|^2]<\infty}$.
In view of the fact that $\sum_{n=1}^\infty\gamma_n=\infty$, necessarily ${\liminf_{n\to\infty}\mathbb{E}\big[\|\nabla F(\theta_n)\|^2\big]=0}$.
\\

\noindent{\bf(\ref{lmm:F(theta):converges:iii})}\
Using the \eqref{eq:theta:n} update rule and taking the expectation in the inequality \eqref{eq:E:n-1:fn} yield
\begin{equation*}
\mathbb{E}[\|\theta_n-\theta_{n-1}\|^{1+\alpha}]
=\gamma_n^{1+\alpha}\mathbb{E}\big[\| g_n(\theta_{n-1})\|^{1+\alpha}\big]
\leq C\big(\mathbb{E}[F(\theta_{n-1})-\inf F]+1\big)\gamma_n^{1+\alpha}.
\end{equation*}
Given that $(\mathbb{E}[F(\theta_n)])_{n\geq1}$ is convergent, it is bounded and
\begin{equation*}
\mathbb{E}[\|\theta_n-\theta_{n-1}\|^{1+\alpha}]\leq C\gamma_n^{1+\alpha}.
\end{equation*}
Thus, by Fubini's theorem, $\mathbb{E}\big[\sum_{n=1}^\infty\|\theta_n-\theta_{n-1}\|^{1+\alpha}\big]<\infty$, and $\mathbb{P}\big(\sum_{n=1}^\infty\|\theta_n-\theta_{n-1}\|^{1+\alpha}=\infty\big)=0$, yielding the sought result.
\\

\noindent{\bf(\ref{lmm:F(theta):converges:iv})}\
By $\mathcal{H}$\ref{asp:f:abc}(\ref{asp:f:abc:i}), \eqref{eq:gradF:square} and Lemma~\ref{lmm:F(theta):converges}(\ref{lmm:F(theta):converges:ii}),
\begin{equation}\label{eq:E:sup:grad:f:2}
\mathbb{E}\big[\|\nabla F(\theta_{k-1})\|^2\big]\vee\mathbb{E}\big[\|g_k(\theta_{k-1})\|^2\big]
\leq C\big(\sup_{n\geq0}\mathbb{E}[F(\theta_n)-\inf F]+1\big)
\eqqcolon K^2<\infty,
\quad k\geq1.
\end{equation}
Define the $L^2(\mathbb{P})$ $(\mathcal{F}_n)_{n\geq0}$-martingale
\begin{equation}
\label{eq:M}
M_n\coloneqq\sum_{k=1}^n\gamma_k\big(g_k(\theta_{k-1})-b(\theta_{k-1})\big),
\quad n\geq1.
\end{equation}
Using that $\big(g_k(\theta_{k-1})-b(\theta_{k-1})\big)_{k\geq1}$ are martingale increments, $\mathcal{H}$\ref{asp:f:abc}(\ref{asp:f:abc:iii}), \eqref{eq:E:sup:grad:f:2} and that $\sum_{k=1}^\infty\gamma_k^2<\infty$ under $\mathcal{H}$\ref{asp:gamma}, we get
\begin{equation*}
\sup_{n\geq1}\mathbb{E}[\|M_n\|^2]
=\sum_{k=1}^\infty\gamma_k^2\mathbb{E}\big[\| g_k(\theta_{k-1})-b(\theta_{k-1})\|^2\big]
\leq(1+c)K^2\sum_{k=1}^\infty\gamma_k^2
<\infty.
\end{equation*}
$(M_n)_{n\geq1}$ is thus bounded in $L^2(\mathbb{P})$, and converges in $L^2(\mathbb{P})$ and $\Pas$.
\\

\noindent{\bf(\ref{lmm:F(theta):converges:v})}\
Using Lemma~\ref{lmm:F(theta):converges}(\ref{lmm:F(theta):converges:i}), Fatou's lemma, the triangle inequality and Lemma~\ref{lmm:F(theta):converges}(\ref{lmm:F(theta):converges:ii}),
\begin{equation*}
\mathbb{E}[|F_\star|]
=\mathbb{E}\big[\lim_{n\to\infty}|F(\theta_n)|\big]
\leq\liminf_{n\to\infty}\mathbb{E}[F(\theta_n)-\inf F]+|\inf F|
<\infty.
\end{equation*}
\end{proof}

\begin{lemma}\label{lmm:gradF->0}
Suppose $\mathcal{H}$\ref{asp:f}--\ref{asp:gamma} hold. Then $\nabla F(\theta_n)\to0$ both in $L^2(\mathbb{P})$ and $\Pas$, as $n\to\infty$.
\end{lemma}

\begin{proof}
We will show separately the $L^2(\mathbb{P})$ and $\Pas$-convergences.
\\

\noindent{\emph{$\stackrel{\blacktriangleright}{}$ Claim~1. $\nabla F(\theta_n)\to0$ in $L^2(\mathbb{P})$ as $n\to\infty$.}}\newline
\noindent{\emph{$\stackrel{\blacktriangleright}{}$ Step~1.1. Reductio ad absurdum.}}\newline
In view of $\mathcal{H}$\ref{asp:f}, Jensen's inequality and Lemma~\ref{lmm:F(theta):converges}(\ref{lmm:F(theta):converges:iii}),
\begin{equation}\label{eq:E:diff:grad:F:->0}
\lim_{n\to\infty}\mathbb{E}\big[\|\nabla F(\theta_n)-\nabla F(\theta_{n-1})\|^2\big]^\frac12
\leq L\lim_{n\to\infty}\mathbb{E}[\|\theta_n-\theta_{n-1}\|^{1+\alpha}]^\frac\alpha{1+\alpha}
=0.
\end{equation}

Recall that $\liminf_{n\to\infty}\mathbb{E}\big[\|\nabla F(\theta_n)\|^2\big]=0$ according to Lemma~\ref{lmm:F(theta):converges}(\ref{lmm:F(theta):converges:ii}).
Suppose, by contradiction, that $\limsup_{n\to\infty}\mathbb{E}\big[\|\nabla F(\theta_n)\|^2\big]^\frac12>\eta$ for some $\eta>0$.
Then, via Lemma~\ref{lmm:seq:jump}, there exist extractions $\chi,\psi\colon\mathbb{N}\to\mathbb{N}$ such that, for all $n\geq0$, $\chi(n)<\psi(n)$,
\begin{gather}
\mathbb{E}\big[\|\nabla F(\theta_{\chi(n)})\|^2\big]^\frac12<\frac\eta3,
\quad
\mathbb{E}\big[\|\nabla F(\theta_{\psi(n)})\|^2\big]^\frac12>\frac{2\eta}3,\label{eq:E:conditions:a}\\
\frac\eta3\leq\mathbb{E}\big[\|\nabla F(\theta_k)\|^2\big]^\frac12\leq\frac{2\eta}3,
\quad\chi(n)<k<\psi(n),\label{eq:E:conditions:b}
\end{gather}
and, taking into consideration \eqref{eq:E:diff:grad:F:->0},
\begin{equation}\label{eq:E:condition:gamma}
\mathbb{E}\big[\|\nabla F(\theta_k)-\nabla F(\theta_{k-1})\|^2\big]^\frac12\leq\frac\eta6,
\quad k\geq\chi(0).
\end{equation}
\newline
\noindent{\emph{$\stackrel{\blacktriangleright}{}$ Step~1.2. Subsequent properties.}}\newline
Let $n\geq0$.
Via \eqref{eq:E:conditions:a}, the triangle inequality, $\mathcal{H}$\ref{asp:f}, Jensen's inequality and \eqref{eq:E:sup:grad:f:2},
\begin{equation*}
\begin{aligned}
\Big(\frac\eta3\Big)^\frac1\alpha
\leq\big|\mathbb{E}\big[\|\nabla F(\theta_{\psi(n)})\|^2\big]^\frac12-\mathbb{E}\big[\|\nabla F(\theta_{\chi(n)})\|^2\big]^\frac12\big|^\frac1\alpha
&\leq L^\frac1\alpha\mathbb{E}[\|\theta_{\psi(n)}-\theta_{\chi(n)}\|^2]^\frac12\\
\leq L^\frac1\alpha\sum_{k=\chi(n)+1}^{\psi(n)}\gamma_k\mathbb{E}\big[\|g_k(\theta_{k-1})\|^2\big]^\frac12
&\leq(K\vee1)L^\frac1\alpha\sum_{k=\chi(n)+1}^{\psi(n)}\gamma_k,
\end{aligned}
\end{equation*}
hence
\begin{equation}\label{eq:E:lower:bound:1}
\sum_{k=\chi(n)+1}^{\psi(n)}\gamma_k
\geq\frac1{K\vee1}\Big(\frac\eta{3L}\Big)^\frac1\alpha.
\end{equation}
Moreover, using \eqref{eq:E:condition:gamma},
$|\mathbb{E}[\|\nabla F(\theta_{\chi(n)+1})\|^2]^\frac12-\mathbb{E}[\|\nabla F(\theta_{\chi(n)})\|^2]^\frac12|\leq\frac\eta6$, so that, by \eqref{eq:E:conditions:b}, $\mathbb{E}[\|\nabla F(\theta_{\chi(n)})\|^2]^\frac12\geq\frac\eta6$.
Thus
\begin{equation}\label{eq:E:lower:bound:2}
\mathbb{E}\big[\|\nabla F(\theta_k)\|^2\big]
\geq\Big(\frac\eta6\Big)^2,
\quad\chi(n)\leq k<\psi(n).
\end{equation}
\newline
\noindent{\emph{$\stackrel{\blacktriangleright}{}$ Step~1.3. Contradiction.}}\newline
Owing to \eqref{eq:E:lower:bound:1} and \eqref{eq:E:lower:bound:2},
\begin{equation*}
0
<\frac1{K\vee1}\Big(\frac\eta{3L}\Big)^\frac1\alpha\Big(\frac\eta6\Big)^2
\leq\sum_{k=\chi(n)+1}^{\psi(n)}\gamma_k\mathbb{E}\big[\|\nabla F(\theta_{k-1})\|^2\big]
\leq\sum_{k=\chi(n)+1}^\infty\gamma_k\mathbb{E}\big[\|\nabla F(\theta_{k-1})\|^2\big].
\end{equation*}
Considering Lemma~\ref{lmm:F(theta):converges}(\ref{lmm:F(theta):converges:ii}), the right hand side of the above inequality tends to $0$ as $n$ goes to $\infty$, which is contradictory.
\\

\noindent{\emph{$\stackrel{\blacktriangleright}{}$ Claim~2. $\nabla F(\theta_n)\to0$ $\Pas$ as $n\to\infty$.}}\newline
\noindent{\emph{$\stackrel{\blacktriangleright}{}$ Step~2.1. Reductio ad absurdum.}}\newline
By Lemma~\ref{lmm:F(theta):converges}(\ref{lmm:F(theta):converges:i}), $(F(\theta_n))_{n\geq1}$ is $\Pas$ bounded. Hence, by Lemma~\ref{lmm:Holder}(\ref{lmm:Holder:ii}), for all $k\geq1$,
\begin{equation}\label{eq:grad:F:<}
\|\nabla F(\theta_{k-1})\|
\leq\Big(\frac{1+\alpha}\alpha\Big)^\frac\alpha{1+\alpha}L^\frac1{1+\alpha}\sup_{n\geq0}\big(F(\theta_{n-1})-\inf F\big)^\frac\alpha{1+\alpha}
\eqqcolon\widetilde{K}
<\infty
\quad\Pas.
\end{equation}
Moreover, via $\mathcal{H}$\ref{asp:f} and Lemma~\ref{lmm:F(theta):converges}(\ref{lmm:F(theta):converges:iii}),
\begin{equation}
\label{eq:diff:grad:F:->0}
\lim_{n\to\infty}\|\nabla F(\theta_n)-\nabla F(\theta_{n-1})\|
\leq L\lim_{n\to\infty}\|\theta_n-\theta_{n-1}\|^\alpha
=0
\quad\Pas.
\end{equation}

Recall that, by Lemma~\ref{lmm:F(theta):converges}(\ref{lmm:F(theta):converges:i}), $\liminf_{n\to\infty}\|\nabla F(\theta_n)\|=0$.
Suppose, by contradiction, that $\mathbb{P}(\limsup_{n\to\infty}\|\nabla F(\theta_n)\|>\eta)>0$ for some $\eta>0$.
Denote $\Lambda=\{\limsup_{n\to\infty}\|\nabla F(\theta_n)\|>\eta\}$.
Then, via Lemma~\ref{lmm:seq:jump}, $\Pas$ on $\Lambda$, there exist extractions $\chi,\psi\colon\mathbb{N}\to\mathbb{N}$ such that, for all $n\geq0$, $\chi(n)<\psi(n)$,
\begin{gather}
\|\nabla F(\theta_{\chi(n)})\|<\frac\eta3,
\quad
\|\nabla F(\theta_{\psi(n)})\|>\frac{2\eta}3\label{eq:conditions:a}\\
\frac\eta3\leq\|\nabla F(\theta_k)\|\leq\frac{2\eta}3,
\quad\chi(n)<k<\psi(n),\label{eq:conditions:b}
\end{gather}
and, in view of \eqref{eq:diff:grad:F:->0} and Lemma~\ref{lmm:F(theta):converges}(\ref{lmm:F(theta):converges:iv}),
\begin{equation}\label{eq:condition}
\|\nabla F(\theta_k)-\nabla F(\theta_{k-1})\|\leq\frac\eta6,
\quad
\Big\|\sum_{j=k}^\infty\gamma_j\big(g_j(\theta_{j-1})-b(\theta_{j-1})\big)\Big\|
\leq\frac1{4L^\frac1\alpha}\Big(\frac\eta3\Big)^\frac1\alpha,
\quad k\geq\chi(0).
\end{equation}
\newline
\noindent{\emph{$\stackrel{\blacktriangleright}{}$ Step~2.2. Subsequent properties.}}\newline
We put ourselves in $\Lambda$ and consider $n\geq0$.
Via \eqref{eq:conditions:a}, the triangle inequality, $\mathcal{H}$\ref{asp:f}, \eqref{eq:condition} and \eqref{eq:grad:F:<},
\begin{equation*}
\begin{aligned}
\Big(\frac\eta3\Big)^\frac1\alpha
\leq\big|\|\nabla F(\theta_{\psi(n)})\|-\|\nabla F(\theta_{\chi(n)})\|\big|^\frac1\alpha
&\leq L^\frac1\alpha\|\theta_{\psi(n)}-\theta_{\chi(n)}\|\\
\leq L^\frac1\alpha\bigg\|\sum_{k=\chi(n)+1}^{\psi(n)}\gamma_k\big(g_k(\theta_{k-1})&-b(\theta_{k-1})\big)\bigg\|
+L^\frac1\alpha\sum_{k=\chi(n)+1}^{\psi(n)}\gamma_k\|b(\theta_{k-1})\|\\
&\leq\frac12\Big(\frac\eta3\Big)^\frac1\alpha+c(\widetilde{K}\vee1)L^\frac1\alpha\sum_{k=\chi(n)+1}^{\psi(n)}\gamma_k,
\end{aligned}
\end{equation*}
hence
\begin{equation}\label{eq:lower:bound:1}
\sum_{k=\chi(n)+1}^{\psi(n)}\gamma_k
\geq\frac1{2c(\widetilde{K}\vee1)}\Big(\frac\eta{3L}\Big)^\frac1\alpha.
\end{equation}
Besides, via \eqref{eq:condition},
$\big|\|\nabla F(\theta_{\chi(n)+1})\|-\|\nabla F(\theta_{\chi(n)})\|\big|\leq\frac\eta6$,
hence, by \eqref{eq:conditions:b},
$\|\nabla F(\theta_{\chi(n)})\|\geq\frac\eta6$.
Therefore,
\begin{equation}\label{eq:lower:bound:2}
\|\nabla F(\theta_k)\|^2\geq\Big(\frac\eta6\Big)^2,
\quad\chi(n)\leq k<\psi(n).
\end{equation}
\newline
\noindent{\emph{$\stackrel{\blacktriangleright}{}$ Step~2.3. Contradiction.}}\newline
Owing to \eqref{eq:lower:bound:1} and \eqref{eq:lower:bound:2},
\begin{equation*}
0
<\frac1{2c(\widetilde{K}\vee1)}\Big(\frac\eta{3L}\Big)^\frac1\alpha\Big(\frac\eta6\Big)^2
\leq\sum_{k=\chi(n)+1}^{\psi(n)}\gamma_k\|\nabla F(\theta_{k-1})\|^2
\leq\sum_{k=\chi(n)+1}^\infty\gamma_k\|\nabla F(\theta_{k-1})\|^2.
\end{equation*}
By Lemma~\ref{lmm:F(theta):converges}(\ref{lmm:F(theta):converges:i}), the right hand side above converges $\Pas$ on $\Lambda$ to $0$ as $n\to\infty$, which is absurd.
\end{proof}

\begin{proof}[Proof of Theorem~\ref{thm:F:cv:gradF:0}]
See Lemmas~\ref{lmm:F(theta):converges}(\ref{lmm:F(theta):converges:i},\ref{lmm:F(theta):converges:v}) and~\ref{lmm:gradF->0}.
\end{proof}

\begin{proof}[Proof of Lemma~\ref{lmm:aux}]
\noindent
Since, by Lemma~\ref{lmm:F(theta):converges}(\ref{lmm:F(theta):converges:i}), $(F(\theta_n))_{n\geq0}$ is $\Pas$ convergent, it is $\Pas$ bounded.
There then exists $\Pas$ some $\eta>0$ such that, for all $n\geq0$, $F(\theta_n)\leq\eta$ $\Pas$.
By $\mathcal{H}$\ref{asp:f->8}, there exists $\Pas$ $\rho>0$ such that, for all $\theta\in\mathbb{R}^m$ such that $\|\theta\|>\rho$, it holds $F(\theta)>\eta$ $\Pas$.
Thus, for all $n\geq0$, $\|\theta_n\|\leq\rho$ $\Pas$.
\end{proof}

The compactness of the accumulation points' set is a direct consequence of the iterate boundedness.
Its conntectedness however stems from Ostrowski's lemma~\cite[Theorem~26.1]{Ost74}, a classical result of topology based on the premise that it is impossible to simultaneously ensure that $\|\theta_n-\theta_{n-1}\|\to0$ as $n\to\infty$ and that the set of accumulation points $\cap_{p\geq0}\overline{\{\theta_n,n\geq p\}}$ of $(\theta_n)_{n\geq0}$ write as a union of two disjoint compact sets~\cite[\S4, Exercise 6]{Gou20}, thus mirroring the main argument for weak convergence.

\begin{proof}[Proof of Theorem~\ref{thm:cv}]
\noindent{\bf(\ref{thm:cv:i})}\
Via Lemmas~\ref{lmm:F(theta):converges}(\ref{lmm:F(theta):converges:iii}) and~\ref{lmm:aux}, $\lim_{n\to\infty}\theta_n-\theta_{n-1}=0$ $\Pas$ and $(\theta_n)_{n\geq0}$ is $\Pas$ bounded.
Thus, by Ostrowski's lemma~\cite[Theorem~26.1]{Ost74}, the set of accumulation points of $(\theta_n)_{n\geq1}$ is $\Pas$ a nonempty compact connected set $\mathcal{A}$.

According to Lemma~\ref{lmm:gradF->0}, there exists $\Lambda\in\mathcal{F}$, $\mathbb{P}(\Lambda)=1$, such that for all $\omega\in\Lambda$, one has $\lim_{n\to\infty}\nabla F(\theta_n(\omega))=0$.
Let $\omega\in\Lambda$, $\theta_\star\in\mathcal{A}(\omega)$ and an extraction $\psi\colon\mathbb{N}\to\mathbb{N}$ such that $\lim_{n\to\infty}\theta_{\psi(n)}(\omega)=\theta_\star$.
Then, by the continuity of $\nabla F$ under $\mathcal{H}$\ref{asp:f}, one has $\nabla F(\theta_\star)=\lim_{n\to\infty}\nabla F(\theta_{\psi(n)}(\omega))=0$.
\\

\noindent{\bf(\ref{thm:cv:ii})}\
Since $\mathcal{A}\subset\nabla F^{-1}(0)$, $\mathcal{A}$ is $\Pas$ at most countable.
Being $\Pas$-compact, connected and at most countable, $\mathcal{A}$ necessarily reduces $\Pas$ to a singleton $\{\theta_\star\}$.
Finally, being bounded and possessing a unique accumulation point, $(\theta_n)_{n\geq0}$ converges $\Pas$ to $\theta_\star$.
\end{proof}

\section{Convergence Rate Proofs}
\label{sec:speed}

\begin{proof}[Proof of Proposition~\ref{prp:Lojasiewicz}]
\noindent{\bf(\ref{prp:Lojasiewicz:i})\hyperref[prp:Lojasiewicz:i:a]{a}.}\
By~\cite[Lemma~4.3(i)]{DK24}, the set of critical levels $\{F(\theta),\theta\in\nabla F^{-1}(0)\cap\mathcal{K}\}$ is finite, say, of the form $(\ell_i)_{1\leq i\leq I}$, $I\geq1$. According to~\cite[Lemma~4.3(ii)]{DK24}, for all $1\leq i\leq I$, there exist $\mathcal{V}_i\supset F^{-1}(\ell_i)\cap\nabla F^{-1}(0)\cap\mathcal{K}$ an open neighborhood, $\beta_i\in(0,1)$ and $\zeta_i>0$ such that $|F(\theta)-\ell_i|^{\beta_i}\leq\zeta_i\|\nabla F(\theta)\|$, $\theta\in\mathcal{V}_i$.
By the continuity of $\nabla F$, we can assume up to a shrinking of the $(\mathcal{V}_i)_{1\leq i\leq I}$ that $\|\nabla F(\theta)\|\leq1$, $\theta\in\mathcal{V}_i$, $1\leq i\leq I$.
Set $\beta=\max_i\beta_i\in(0,1)$ and $\zeta=(1\vee\max_i\zeta_i^{1/\beta_i})^\beta>0$. Then $|F(\theta)-\ell_i|^\beta\leq\zeta\|\nabla F(\theta)\|$, $\theta\in\mathcal{V}_i$, $1\leq i\leq I$.
Finally, letting $\theta_\star\in\nabla F^{-1}(0)\cap\mathcal{K}$, one has $F(\theta_\star)=\ell_i$ for some $1\leq i\leq I$, so that $\theta_\star\in\mathcal{V}_i\cap\nabla F^{-1}(0)$ and $|F(\theta)-F(\theta_\star)|^\beta=|F(\theta)-\ell_i|^\beta\leq\zeta\|\nabla F(\theta)\|$, $\theta\in\mathcal{V}_i$, thus the result.
\\

\noindent{\bf(\ref{prp:Lojasiewicz:i})\hyperref[prp:Lojasiewicz:i:b]{b}.}\
The sought property follows by the previous result and Lemma~\ref{lmm:beta+:zeta1}.
\\

\noindent{\bf(\ref{prp:Lojasiewicz:ii})}\
Since $F$ is {\L}ojasiewicz, to any $\theta_\star\in\nabla F^{-1}(0)$, we can associate a triplet $(\mathcal{V}_{\theta_\star},\beta_{\theta_\star},\zeta_{\theta_\star})\in\mathcal{B}(\mathbb{R}^m)\times(0,1)\times\mathbb{R}^+$ such that $\mathcal{V}_{\theta_\star}$ is an open neighborhood of $\theta_\star$ and
\begin{equation*}
|F(\theta)-F(\theta_\star)|^{\beta_{\theta_\star}}\leq\zeta_{\theta_\star}\|\nabla F(\theta)\|,\quad\theta\in\mathcal{V}_{\theta_\star}.
\end{equation*}
Take
\begin{equation*}
\beta=\frac12\big(\sup_{\theta_\star\in\nabla F^{-1}(0)}\beta_{\theta_\star}+1\big)\in(0,1].
\end{equation*}
Then $\beta>\beta_{\theta_\star}$, $\theta_\star\in\nabla F^{-1}(0)$, hence the sought property by invoking Lemma~\ref{lmm:beta+:zeta1}.
\\

\noindent{\bf(\ref{prp:Lojasiewicz:iii})}\
Let $\theta_\star\in\nabla F^{-1}(0)$. Then $(F(\theta_\star)-\inf F)^\beta\leq\zeta\|\nabla F(\theta_\star)\|=0$ for some $(\beta,\zeta)\in(0,1)\times\mathbb{R}^*_+$.
A similar observation to Remark~\ref{rmk:beta>alpha/1+alpha} ensures that $\beta\geq\frac\alpha{1+\alpha}$.
\end{proof}

\begin{lemma}\label{lmm:levels}
Assume $\mathcal{H}$\ref{asp:f}--\ref{asp:f->8} hold.
Then, for all $\delta\in(0,1)$, there exists $\ell_\delta>\inf F$ such that $\mathbb{P}(\mathcal{A}\subset\mathcal{L}_{\ell_\delta})\geq1-\delta$.
\end{lemma}

\begin{proof}
\noindent{\emph{$\stackrel{\blacktriangleright}{}$ Step~1. Final hitting time.}}\newline
Consider the random time
\begin{equation}\label{eq:stop:0}
\tau_0\coloneqq\inf{\{n\geq0\colon\forall k\geq n,\|\nabla F(\theta_k)\|\leq1\}},
\end{equation}
with the convention $\inf{\varnothing}=\infty$.
According to Lemma~\ref{lmm:gradF->0}, $\tau_0<\infty$ $\Pas$.

Let $\delta\in(0,1)$. The reverse Fatou lemma ensures that
\begin{equation*}
\limsup_{n\to\infty}\mathbb{P}(\tau_0\geq n)
\leq\mathbb{P}\big(\limsup_{n\to\infty}{\{\tau_0\geq n\}}\big)
=0.
\end{equation*}
Thus, there exists $n_\delta\geq0$ such that
\begin{equation}
\label{eq:jigsaw:1}
\mathbb{P}(\tau_0<n_\delta)\geq1-\frac\delta2.
\end{equation}

According to Lemmas~\ref{lmm:F(theta):converges}(\ref{lmm:F(theta):converges:i}) and~\ref{lmm:aux} and Theorem~\ref{thm:cv}(\ref{thm:cv:i}), there exists $\Lambda\in\mathcal{F}$ such that $\mathbb{P}(\Lambda)=1$ and, for all $\omega\in\Lambda$, $F(\theta_n(\omega))\to F_\star(\omega)$, $(\theta_n(\omega))_{n\geq0}$ is bounded and its set of accumulation points $\mathcal{A}(\omega)$ is a compact connected set.
Let $\omega\in\Lambda$.
By the continuity of $F$ via $\mathcal{H}$\ref{asp:f}, for all $\theta_\star\in\mathcal{A}(\omega)$, $F(\theta_\star)=\lim_{n\to\infty} F(\theta_{\psi(n)}(\omega))=F_\star(\omega)$, where $\psi\colon\mathbb{N}\to\mathbb{N}$ is an extraction such that $\theta_{\psi(n)}(\omega)\to\theta_\star$ as $n\to\infty$.
$F$ is therefore constant on $\mathcal{A}(\omega)$ equal to $F_\star(\omega)$.

Let $\ell>\inf F$.
The previous discussion entails
\begin{equation}\label{eq:<->}
\{\mathcal{A}\subset\mathcal{L}_\ell\}
=\{F_\star\leq\ell\}
=\big\{\forall\varepsilon\in\mathbb{Q}^*_+,\exists p\geq n_\delta,\forall n\geq p,F(\theta_n)\leq\ell+\varepsilon\big\}.
\end{equation}
\newline
\noindent{\emph{$\stackrel{\blacktriangleright}{}$ Step~2. High probability bound.}}\newline
Let $n>n_\delta$.
By \eqref{eq:<->} and Markov's inequality,
\begin{equation}\label{eq:master:ineq}
\begin{aligned}
\mathbb{P}&(\mathcal{A}\not\subset\mathcal{L}_\ell,\tau_0<n_\delta)
=\mathbb{P}\big(\big\{\exists\varepsilon\in\mathbb{Q}^*_+,\forall p\geq n_\delta,\exists n\geq p,F(\theta_n)>\ell+\varepsilon\big\}\cap\{\tau_0<n_\delta\}\big)\\
&\leq\inf_{p\geq n_\delta}\frac{\mathbb{E}[\sup_{n\geq p}\sup_{\varepsilon\in\mathbb{Q}^*_+}(F(\theta_n)-\inf F-\varepsilon)^+\mathbbm1_{\tau_0<n_\delta}]}{\ell-\inf F}
\leq\frac{\mathbb{E}[\sup_{n\geq n_\delta}(F(\theta_n)-\inf F)\mathbbm1_{\tau_0<n_\delta}]}{\ell-\inf F}.
\end{aligned}
\end{equation}
Let $(\Delta_n)_{n\geq1}$ be the $(\mathcal{F}_n)_{n\geq0}$-martingale increment sequence given by
\begin{equation}\label{eq:Delta}
\Delta_n=\nabla F(\theta_{n-1})^\top\big(b(\theta_{n-1})- g_n(\theta_{n-1})\big)\mathbbm1_{\|\nabla F(\theta_{n-1})\|\leq 1},
\quad n\geq1.
\end{equation}
Using \eqref{eq:F(theta:n)<F(theta:n-1)} and $\mathcal{H}$\ref{asp:f:abc}(\ref{asp:f:abc:ii}),
\begin{equation*}
\begin{aligned}
(F(\theta_n)&-\inf F)\mathbbm1_{\tau_0<n_\delta}\\
&\leq\Big(F(\theta_{n-1})-\inf F-\kappa\gamma_n\|\nabla F(\theta_{n-1})\|^2
+\gamma_n\Delta_n
+\frac{L}{1+\alpha}\gamma_n^{1+\alpha}\| g_n(\theta_{n-1})\|^{1+\alpha}\Big)\mathbbm1_{\tau_0<n_\delta},
\end{aligned}
\end{equation*}
hence, by telescoping,
\begin{equation*}
\begin{aligned}
(F(\theta_n)&-\inf F)\mathbbm1_{\tau_0<n_\delta}\\
&\leq\Big(F(\theta_{n_\delta-1})-\inf F
+\sum_{k=n_\delta}^n\gamma_k\Delta_k
+\frac{L}{1+\alpha}\sum_{k=n_\delta}^n\gamma_k^{1+\alpha}\| g_k(\theta_{k-1})\|^{1+\alpha}\Big)\mathbbm1_{\tau_0<n_\delta}.
\end{aligned}
\end{equation*}
Therefore, by Fubini's theorem,
\begin{equation}\label{eq:ineq:A+B}
\begin{aligned}
\mathbb{E}\big[\sup_{n\geq n_\delta}&(F(\theta_n)-\inf F)\mathbbm1_{\tau<n_\delta}\big]\\
&\leq\sup_{n\geq0}\mathbb{E}[F(\theta_n)-\inf F]
+\mathbb{E}\Big[\sup_{n>n_\delta}\Big|\sum_{k=n_\delta}^n\gamma_k\Delta_k\Big|\Big]
+\frac{L\sup_{n\geq1}\mathbb{E}[\|g_n(\theta_{n-1})\|^{1+\alpha}]}{1+\alpha}\sum_{k=1}^\infty\gamma_k^{1+\alpha}.
\end{aligned}
\end{equation}
By Lemma~\ref{lmm:F(theta):converges}(\ref{lmm:F(theta):converges:ii}), $\sup_{n\geq0}\mathbb{E}[F(\theta_n)-\inf F]<\infty$.
Moreover, using that $x^{1+\alpha}\leq C(1+x^2)$, $x\in\mathbb{R}$, and $\mathcal{H}$\ref{asp:f:abc},
\begin{equation}\label{eq:sup:grad:f:1+alpha}
\sup_{n\geq1}\mathbb{E}\big[\| g_n(\theta_{n-1})\|^{1+\alpha}\big]
\leq C\big(\sup_{n\geq1}\mathbb{E}[F(\theta_{n-1})-\inf F]+1\big)<\infty.
\end{equation}
Besides, by the definition \eqref{eq:Delta}, using the Cauchy-Schwarz inequality, Jensen's inequality and $\mathcal{H}$\ref{asp:f:abc}(\ref{asp:f:abc:i}),
\begin{equation}\label{eq:E[Delta^2]<}
\sup_{n\geq1}\mathbb{E}\big[\Delta_n^2\big]
\leq2\big(\sup_{n\geq1}\mathbb{E}\big[\| g_n(\theta_{n-1})\|^2\big]+c^2\big)
\leq C\big(\sup_{n\geq1}\mathbb{E}[F(\theta_{n-1})-\inf F]+1\big)
<\infty,
\quad n\geq1.
\end{equation}
Hence, by monotone convergence, Jensen's inquality, Doob's maximal inequality (or the Burkholder-Davis-Gundy inequality), \eqref{eq:E[Delta^2]<} and $\mathcal{H}$\ref{asp:gamma},
\begin{equation}\label{eq:ineq:B}
\mathbb{E}\Big[\sup_{n>n_\delta}\Big|\sum_{k=n_\delta}^n\gamma_k\Delta_k\Big|\Big]
=\lim_{N\to\infty}\mathbb{E}\Big[\sup_{n_\delta<n\leq N}\Big|\sum_{k=n_\delta}^n\gamma_k\Delta_k\Big|^2\Big]^\frac12
\leq\Big(\sum_{k=1}^\infty\gamma_k^2\mathbb{E}[\Delta_k^2]\Big)^\frac12
<\infty.
\end{equation}
Coming back to \eqref{eq:master:ineq}, by combining \eqref{eq:ineq:A+B}, \eqref{eq:sup:grad:f:1+alpha} and \eqref{eq:ineq:B},
\begin{equation}\label{eq:P(A:not:in:L)}
\mathbb{P}(\mathcal{A}\not\subset\mathcal{L}_\ell,\tau_0<n_\delta)
\leq\frac{C}{\ell-\inf F}
\to0
\quad\text{as $\ell\to\infty$.}
\end{equation}
\newline
\noindent{\emph{$\stackrel{\blacktriangleright}{}$ Step~3. Conclusion.}}\newline
\eqref{eq:P(A:not:in:L)} shows that there exists $\ell_\delta>\inf F$ such that
\begin{equation}\label{eq:jigsaw:2}
\mathbb{P}(\mathcal{A}\not\subset\mathcal{L}_{\ell_\delta},\tau_0<n_\delta)\leq\frac\delta2.
\end{equation}
Consequently, from \eqref{eq:jigsaw:1} and \eqref{eq:jigsaw:2},
\begin{equation*}
\mathbb{P}(\mathcal{A}\subset\mathcal{L}_{\ell_\delta})
\geq\mathbb{P}(\mathcal{A}\subset\mathcal{L}_{\ell_\delta},\tau_0<n_\delta)
\geq1-\delta.
\end{equation*}
\end{proof}

The next result deals with the convergence rate of the gradients $(\|\nabla F(\theta_n)\|)_{n\geq0}$ and requires only $\mathcal{H}$\ref{asp:f}--\ref{asp:gamma}.
It is provided here help prove our first \eqref{eq:theta:n} convergence rate in Theorem~\ref{thm:F-F*}.

\begin{lemma}\label{lmm:grad:F:n}
Suppose $\mathcal{H}$\ref{asp:f}--\ref{asp:gamma} hold. Then,
\begin{enumerate}[\bf(i)]
    \item\label{lmm:grad:F:n:i}
there exists a constant $C>0$ such that, for all $0\leq p\leq n$,
\begin{equation*}
\inf_{p\leq k\leq n}{\mathbb{E}\big[\|\nabla F(\theta_k)\|^2\big]}
\leq C\Big(\sum_{k=p+1}^{n+1}\gamma_k\Big)^{-1}.
\end{equation*}
    \item\label{lmm:grad:F:n:ii}
there exists a constant $C>0$ such that, for all $\delta\in(0,1)$, for all $0\leq p\leq n$,
\begin{equation*}
\inf_{p\leq k\leq n}\|\nabla F(\theta_k)\|\leq\frac{C}\delta\Big(\sum_{k=p+1}^{n+1}\gamma_k\Big)^{-\frac12}
\quad\text{with probability at least $1-\delta$.}
\end{equation*}
\end{enumerate}
\end{lemma}

The first part of Lemma~\ref{lmm:grad:F:n} is mirrored for instance in~\cite[Theorem~5.12]{GG24}.
It is somewhat a more generic result than~\cite[Theorem~2(a)]{Lei+20}, as the left end indexation on $k$ starts at an arbitrary integer $p$, not necessarily null.

\begin{proof}[Proof of Lemma~\ref{lmm:grad:F:n}]
\noindent{\bf(\ref{lmm:grad:F:n:i})}\
From Lemma~\ref{lmm:F(theta):converges}(\ref{lmm:F(theta):converges:ii}),
\begin{equation*}
\inf_{p\leq k\leq n}\mathbb{E}\big[\|\nabla F(\theta_k)\|^2\big]
\sum_{k=p+1}^{n+1}\gamma_k
\leq\sum_{k=1}^\infty\gamma_k\mathbb{E}\big[\|\nabla F(\theta_{k-1})\|^2\big]
\eqqcolon\Xi<\infty.
\quad0\leq p\leq n.
\end{equation*}
\newline
\noindent{\bf(\ref{lmm:grad:F:n:ii})}\
Let $\delta\in(0,1)$. By the Markov and Jensen inequalities,
\begin{equation*}
\mathbb{P}\Big(\inf_{p\leq k\leq n}\|\nabla F(\theta_k)\|\geq\frac{\sqrt\Xi}\delta\Big(\sum_{k=p+1}^{n+1}\Big)^{-\frac12}\Big)\leq\delta. 
\end{equation*}
\end{proof}

\begin{proof}[Proof of Theorem~\ref{thm:F-F*}]
\noindent{\bf(\ref{thm:F-F*:i})\hyperref[thm:F-F*:i:a]{a}.}\
{\emph{$\stackrel{\blacktriangleright}{}$ Step~1. Final hitting time.}}\newline
Let $\delta\in(0,1)$. According to Lemma~\ref{lmm:levels}, there exists $\ell_\delta>\inf F$ such that
\begin{equation}\label{eq:levelset}
\mathbb{P}(\mathcal{A}\subset\mathcal{L}_{\ell_\delta})\geq1-\frac\delta3.
\end{equation}
Besides, by Proposition~\ref{prp:Lojasiewicz}(\ref{prp:Lojasiewicz:i}), there exist $(\beta_\delta,\zeta_\delta)\in(0,1)\times\mathbb{R}^*_+$ such that, for all $\theta_\star\in\nabla F^{-1}(0)\cap\mathcal{L}_{\ell_\delta}$, there exists an open neighborhood $\mathcal{V}_\delta(\theta_\star)$ of $\theta_\star$ such that $|F(\theta)-F(\theta_\star)|^{\beta_\delta}\leq\zeta_\delta\|\nabla F(\theta)\|$,
$\theta\in\mathcal{V}_\delta(\theta_\star)$.
Furthermore, by Theorem~\ref{thm:cv}(\ref{thm:cv:i}), there exists $\Lambda\in\mathcal{F}$ such that $\mathbb{P}(\Lambda)=1$ and, for all $\omega\in\Lambda$, the set of accumulation $\mathcal{A}(\omega)$ is a compact connected set.
Recalling that $F(\theta_\star)=F_\star(\omega)$, $\theta_\star\in\mathcal{A}(\omega)$, note that
\begin{equation}\label{eq:Lojasiewicz:delta}
|F(\theta)-F_\star(\omega)|^{\beta_\delta}
\leq\zeta_\delta\|\nabla F(\theta)\|,
\quad\theta\in\mathcal{V}_\delta(\theta_\star),
\;\theta_\star\in\mathcal{A}(\omega),
\;\omega\in\Lambda\cap\{\mathcal{A}\subset\mathcal{L}_{\ell_\delta}\}.
\end{equation}

Define $\mathcal{V}_{\delta,\star}(\omega)\coloneqq\cup_{\theta_\star\in\mathcal{A}(\omega)}\mathcal{V}_\delta(\theta_\star)$, $\omega\in\Lambda\cap\{\mathcal{A}\subset\mathcal{L}_{\ell_\delta}\}$.
Note that, for all $\omega\in\Lambda\cap\{\mathcal{A}\subset\mathcal{L}_{\ell_\delta}\}$, $\mathcal{A}(\omega)\subsetneq\mathcal{V}_{\delta,\star}(\omega)$, so that $\mathcal{V}_{\delta,\star}(\omega)\neq\varnothing$. 
Let
\begin{equation}\label{eq:stop:1}
\tau_1\coloneqq 
\begin{cases}
\tau_0\vee\inf{\{n\geq0\colon\forall k\geq n, \theta_k\in\mathcal{V}_{\delta,\star}\}}
&\text{on $\{\mathcal{A}\subset\mathcal{L}_{\ell_\delta}\}$,}\\
\infty
&\text{on $\{\mathcal{A}\not\subset\mathcal{L}_{\ell_\delta}\}$,}
\end{cases}
\end{equation}
with the convention $\inf{\varnothing}=\infty$.
Via the $\Pas$-boundedness of $(\theta_n)_{n\geq0}$ by Lemma~\ref{lmm:aux}, one has $\tau_1<\infty$ $\Pas$ on $\{\mathcal{A}\subset\mathcal{L}_{\ell_\delta}\}$.
Hence, by the reverse Fatou lemma,
\begin{equation*}
\limsup_{n\to\infty}\mathbb{P}(\tau_1\geq n,\mathcal{A}\subset\mathcal{L}_{\ell_\delta})
\leq\mathbb{P}\big(\limsup_{n\to\infty}{\{\tau_1\geq n\}}\cap\{\mathcal{A}\subset\mathcal{L}_{\ell_\delta}\}\big)
=0.
\end{equation*}
Via \eqref{eq:levelset}, there then exists $n_\delta\geq0$ such that
\begin{equation}
\label{eq:n:delta}
\mathbb{P}(\tau_1<n_\delta,\mathcal{A}\subset\mathcal{L}_{\ell_\delta})
\geq\mathbb{P}(\mathcal{A}\subset\mathcal{L}_{\ell_\delta})-\frac\delta3
\geq1-\frac23\delta.
\end{equation}
\newline
\noindent{\emph{$\stackrel{\blacktriangleright}{}$ Step~2. $L^1(\mathbb{P})$ bound on $(F(\theta_n))_{n\geq0}$.}}\newline
Introduce
\begin{equation}\label{eq:mu:delta}
\mu_\delta\coloneqq\zeta_\delta^{-\frac1{\beta_\delta}}>0
\end{equation}
and
\begin{equation}\label{eq:A:delta}
\mathbb{A}_\delta\coloneqq\{\tau_1<n_\delta\}\cap\{\mathcal{A}\subset\mathcal{L}_{\ell_\delta}\}.
\end{equation}
Let $n_\delta\leq k\leq n$. Observe that, by \eqref{eq:stop:0} and \eqref{eq:stop:1}, $\|\nabla F(\theta_k)\|\leq1$ on $\mathbb{A}_\delta$.
Hence, using \eqref{eq:Lojasiewicz:delta},
\begin{equation*}
\begin{aligned}
\mu_\delta\mathbb{E}[|F(\theta_{k})&-F_\star|\mathbbm1_{\mathbb{A}_\delta}]\\
&\leq\mathbb{E}\big[\|\nabla F(\theta_k)\|^\frac1{\beta_\delta}\mathbbm1_{\mathbb{A}_\delta}\big]\\
&\leq\begin{cases}
\mathbb{E}\big[\|\nabla F(\theta_k)\|^2\big]
&\text{if $\beta_\delta\in(0,\frac12]$, since $\|\nabla F(\theta_k)\|\leq1$,}\\
\mathbb{E}\big[\|\nabla F(\theta_k)\|^2\big]^\frac1{2\beta_\delta}
&\text{if $\beta_\delta\in(\frac12,1)$, via Jensen's inequality,}
\end{cases}\\
&=\mathbb{E}\big[\|\nabla F(\theta_k)\|^2\big]^{1\wedge\frac1{2\beta_\delta}}.
\end{aligned}
\end{equation*}
Therefore, via Lemma~\ref{lmm:grad:F:n}(\ref{lmm:grad:F:n:i}),
\begin{equation*}
\inf_{n_\delta\leq k\leq n}\mathbb{E}[|F(\theta_k)-F_\star|\mathbbm1_{\mathbb{A}_\delta}]
\leq\mu_\delta^{-1}\inf_{n_\delta\leq k\leq n}\mathbb{E}\big[\|\nabla F(\theta_k)\|^2\big]^{1\wedge\frac1{2\beta_\delta}}
\leq C_\delta\Big(\sum_{k=n_\delta+1}^{n+1}\gamma_k\Big)^{-1\wedge\frac1{2\beta_\delta}}.
\end{equation*}
Thus
\begin{equation*}
\mathbb{E}\big[\inf_{0\leq k\leq n}|F(\theta_k)-F_\star|\mathbbm1_{\mathbb{A}_\delta}\big]
\leq\inf_{n_\delta\leq k\leq n}\mathbb{E}[|F(\theta_k)-F_\star|\mathbbm1_{\mathbb{A}_\delta}]
\leq C_\delta\Big(\sum_{k=1}^{n+1}\gamma_k\Big)^{-1\wedge\frac1{2\beta_\delta}},
\end{equation*}
where we used that $\sum_{k=1}^{n+1}\gamma_k\sim\sum_{k=n_\delta+1}^{n+1}\gamma_k$ as $n\to\infty$.

Finally, up to a modification of $C_\delta$, for all $n\geq0$,
\begin{equation}
\begin{aligned}
\mathbb{E}\big[&\inf_{0\leq k\leq n}|F(\theta_k)-F_\star|\mathbbm1_{\mathbb{A}_\delta}\big]\\
&=\mathbb{E}\big[\inf_{0\leq k\leq n}|F(\theta_k)-F_\star|\mathbbm1_{\mathbb{A}_\delta}\big]\mathbbm1_{n<n_\delta}+\mathbb{E}\big[\inf_{0\leq k\leq n}|F(\theta_k)-F_\star|\mathbbm1_{\mathbb{A}_\delta}\big]\mathbbm1_{n\geq n_\delta}\\
&\leq\Big(\sup_{k\geq0}{\mathbb{E}[|F(\theta_k)-F_\star|]}\Big(\sum_{k=1}^{n_\delta}\gamma_k\Big)^{1\wedge\frac1{2\beta_\delta}}+C_\delta\Big)\Big(\sum_{k=1}^{n+1}\gamma_k\Big)^{-1\wedge\frac1{2\beta_\delta}}
\leq C_\delta\Big(\sum_{k=1}^{n+1}\gamma_k\Big)^{-1\wedge\frac1{2\beta_\delta}},
\end{aligned}
\label{ineq:F:L1}
\end{equation}
with the convention $\sum_\varnothing=0$, where we used that, due to Lemmas~\ref{lmm:F(theta):converges}(\ref{lmm:F(theta):converges:ii},\ref{lmm:F(theta):converges:v}),
\begin{equation}\label{eq:sup:E:F-F*}
\sup_{k\geq0}\mathbb{E}[|F(\theta_k)-F_\star|]
\leq\sup_{k\geq0}\mathbb{E}[F(\theta_k)-\inf F]+|\inf F|+\mathbb{E}[|F_\star|]<\infty.
\end{equation}
\newline
\noindent{\emph{$\stackrel{\blacktriangleright}{}$ Step~3. High probability bound on $(F(\theta_n))_{n\geq0}$.}}\newline
Reusing the constant $C_\delta$ from the right hand side of \eqref{ineq:F:L1}, by Markov's inequality and \eqref{eq:A:delta},
\begin{equation*}
\begin{aligned}
\mathbb{P}\Big(\inf_{0\leq k\leq n}|F(\theta_k)-F_\star|
>\frac3\delta C_\delta\Big(\sum_{k=1}^{n+1}\gamma_k\Big)^{-1\wedge\frac1{2\beta_\delta}},&\tau_1<n_\delta,\mathcal{A}\subset\mathcal{L}_{\ell_\delta}\Big)\\
&\leq\frac{\delta\mathbb{E}[\inf_{0\leq k\leq n}|F(\theta_k)-F_\star|\mathbbm1_{\mathbb{A}_\delta}]}{3C_\delta\big(\sum_{k=1}^{n+1}\gamma_k\big)^{-1\wedge\frac1{2\beta_\delta}}}
\leq\frac\delta3,
\quad n\geq0.
\end{aligned}
\end{equation*}
Therefore, recalling \eqref{eq:n:delta},
\begin{equation*}
\begin{aligned}
\mathbb{P}\Big(\inf_{0\leq k\leq n}&|F(\theta_k)-F_\star|
\leq\frac3\delta C_\delta\Big(\sum_{k=1}^{n+1}\gamma_k\Big)^{-1\wedge\frac1{2\beta_\delta}}\Big)\\
&\geq\mathbb{P}\Big(\inf_{0\leq k\leq n}|F(\theta_k)-F_\star|\leq\frac3\delta C_\delta\Big(\sum_{k=1}^{n+1}\gamma_k\Big)^{-1\wedge\frac1{2\beta_\delta}},\tau_1<n_\delta,\mathcal{A}\subset\mathcal{L}_{\ell_\delta}\Big)
\geq1-\delta,
\quad n\geq0.
\end{aligned}
\end{equation*}
\newline
\noindent{\bf(\ref{thm:F-F*:i})\hyperref[thm:F-F*:i:b]{b}.}\
The proof is identical to the previous one, once setting $\mu_\delta$ to $1$.
\\

\noindent{\bf(\ref{thm:F-F*:ii})}\
Reusing notation from the previous proof, let $\omega\in\Lambda$ and $\theta_\star\in\mathcal{A}(\omega)$.
According to Proposition~\ref{prp:Lojasiewicz}(\ref{prp:Lojasiewicz:ii}), there exists $\beta\in(0,1]$ independent of $\omega$ and an open neighborhood $\mathcal{V}(\theta_\star)$ of $\theta_\star$ such that $|F(\theta)-F(\theta_\star)|^\beta\leq\|\nabla F(\theta)\|$, $\theta\in\mathcal{V}(\theta_\star)$.
Define $\mathcal{V}_\star(\omega)\coloneqq\cup_{\theta_\star\in\mathcal{A}(\omega)}\mathcal{V}(\theta_\star)$.
Given that $\mathcal{A}(\omega)\subsetneq\mathcal{V}_\star(\omega)$, one has $\mathcal{V}_\star(\omega)\neq\varnothing$.
Recalling \eqref{eq:stop:0}, consider
\begin{equation}\label{eq:stop}
\tau_2\coloneqq\tau_0\vee\inf{\{n\geq0\colon\forall k\geq n,\theta_k\in\mathcal{V}_\star\}},
\end{equation}
with the convention $\inf{\varnothing}=\infty$.
Since, according to Lemma~\ref{lmm:aux}, $(\theta_n)_{n\geq0}$ is $\Pas$ bounded, one has $\tau_2<\infty$ $\Pas$.

Let $\delta\in(0,1)$. By the reverse Fatou lemma,
\begin{equation*}
\limsup_{n\to\infty}\mathbb{P}(\tau_2\geq n)
\leq\mathbb{P}\big(\limsup_{n\to\infty}{\{\tau_2\geq n\}}\big)
=0.
\end{equation*}
Therefore, there exists $n_\delta\geq0$ such that
\begin{equation}
\label{eq:n:delta:bis}
\mathbb{P}(\tau_2<n_\delta)\geq1-\frac\delta2.
\end{equation}

Take
\begin{equation}\label{eq:B:delta}
\mathbb{B}_\delta\coloneqq\{\tau_2<n_\delta\}.
\end{equation}
Performing the substitutions $\mathbb{A}_\delta\gets\mathbb{B}_\delta$, $\beta_\delta\gets\beta$ and $\mu_\delta\gets1$ in Step~2 of the preceding proof and applying the same lines of reasoning therein line by line yields
\begin{equation}
\mathbb{E}\big[\inf_{0\leq k\leq n}|F(\theta_k)-F_\star|\mathbbm1_{\mathbb{B}_\delta}\big]
\leq C_\delta\Big(\sum_{k=1}^{n+1}\gamma_k\Big)^{-1\wedge\frac1{2\beta}},
\quad n\geq0.
\label{ineq:F:L1:bis}
\end{equation}
Fixing the $C_\delta$ above, via  Markov's inequality, \eqref{ineq:F:L1:bis}, \eqref{eq:B:delta} and \eqref{eq:n:delta:bis},
\begin{equation}\label{eq:high:probability}
\begin{aligned}
\mathbb{P}\Big(\inf_{0\leq k\leq n}&|F(\theta_k)-F_\star|
\leq\frac2\delta C_\delta\Big(\sum_{k=1}^{n+1}\gamma_k\Big)^{-1\wedge\frac1{2\beta}}\Big)\\
&\geq\mathbb{P}\Big(\inf_{0\leq k\leq n}|F(\theta_k)-F_\star|\leq\frac2\delta C_\delta\Big(\sum_{k=1}^{n+1}\gamma_k\Big)^{-1\wedge\frac1{2\beta}},\tau_2<n_\delta\Big)
\geq1-\delta,
\quad n\geq0.
\end{aligned}
\end{equation}
\newline
\noindent{\bf(\ref{thm:F-F*:iii})}\
It ensues from Proposition~\ref{prp:Lojasiewicz}(\ref{prp:Lojasiewicz:iii}) and Theorem~\ref{thm:cv}(\ref{thm:cv:i}) that $F_\star=\inf F$ $\Pas$.
The result follows from \eqref{eq:framework:3} and Lemma~\ref{lmm:grad:F:n}(\ref{lmm:grad:F:n:ii}).
\end{proof}

Our proofs for recovering the convergence speeds of Theorem~\ref{thm:F-F*:special:cases} are novel to literature and build upon two key ideas.
First, the quasisupermartingale property of the function values is decomposed into a drift and a martingale terms (\ref{eq:helper:2},\ref{eq:helper:rec}).
Their contributions are reflected in the exponents of the convergence rates displaying minima between multiple quantities.
Second, a new concept of final hitting time (\ref{eq:stop:1},\ref{eq:stop}) is leveraged, which, although not a stopping time, allows isolating the asymptotic regime in order to account for the {\L}ojasiewicz property $\mathcal{H}$\ref{asp:Lojasiewicz}.
The interplay between the martingale term and the final hitting time is accounted for in the convergence speed (\ref{eq:helper:5},\ref{eq:helper:terms:3}).
In the case of $\beta>\frac12$, the convergence rate exponent writes as a minimum of three quantities: $\frac1{2\beta-1}$, $\alpha\rho-1$ and $\frac{\rho-1}2$. The first two stem from the drift contribution and the last one from the martingale contribution. The double contribution from the drift is due to the correction of order disparity between the quantities $F(\theta_n)-F_\star=\mathrm{O}(\|\nabla F(\theta_n)\|^{1/\beta})$ (as per $\mathcal{H}$\ref{asp:Lojasiewicz}) and $\|\nabla F(\theta_n)\|^2$ resulting from the drift-martingale decomposition \eqref{eq:helper:end}.

\begin{proof}[Proof of Theorem~\ref{thm:F-F*:special:cases}]
\noindent Let $\delta\in(0,1)$.\\

\noindent{\bf(\ref{thm:F-F*:special:cases:i})\hyperref[thm:F-F*:special:cases:i:a]{a}.}\
We reposition ourselves in the framework of Theorem~\ref{thm:F-F*}(\ref{thm:F-F*:i})\hyperref[thm:F-F*:i:a]{a}'s proof.
Define
\begin{equation}\label{eq:A:delta:n}
\mathbb{A}^n_\delta
\coloneqq\mathbb{A}_\delta\cap\{F(\theta_n)>F_\star\}
=\{\tau_1<n_\delta\}\cap\{\mathcal{A}\subset\mathcal{L}_{\ell_\delta}\}\cap\{F(\theta_n)>F_\star\},
\quad n>n_\delta.
\end{equation}
Let $n_\delta<k\leq n$. Via \eqref{eq:F(theta:n)<F(theta:n-1)}, $\mathcal{H}$\ref{asp:f:abc}(\ref{asp:f:abc:ii}), \eqref{eq:Delta} and \eqref{eq:sup:grad:f:1+alpha},
\begin{equation}\label{eq:helper:init}
\mathbb{E}[(F(\theta_k)-F_\star)\mathbbm1_{\mathbb{A}^n_\delta}]
\leq\mathbb{E}[(F(\theta_{k-1})-F_\star)\mathbbm1_{\mathbb{A}^n_\delta}]
-\kappa\gamma_k\mathbb{E}\big[\|\nabla F(\theta_{k-1})\|^2\mathbbm1_{\mathbb{A}^n_\delta}\big]
+\gamma_k\mathbb{E}\big[\Delta_k\mathbbm1_{\mathbb{A}^n_\delta}\big]
+C\gamma_k^{1+\alpha}.
\end{equation}
\newline
\noindent{\emph{$\stackrel{\blacktriangleright}{}$ Case~1. $\beta_\delta\in\big(0,\frac12\big]$.}}\newline
\noindent{\emph{$\stackrel{\blacktriangleright}{}$ Step~1.1. $L^1(\mathbb{P})$ bound on $(F(\theta_n))_{n\geq0}$.}}\newline
\eqref{eq:helper:init} rewrites
\begin{equation}\label{eq:helper:0}
\begin{aligned}
\mathbb{E}[(F(\theta_k)-F_\star)\mathbbm1_{\mathbb{A}^n_\delta}]
&\leq(1-\kappa\mu_\delta\gamma_k)\mathbb{E}[(F(\theta_{k-1})-F_\star)\mathbbm1_{\mathbb{A}^n_\delta}]\\
&\quad+\kappa\gamma_k\mathbb{E}\big[\big(\mu_\delta(F(\theta_{k-1})-F_\star)-\|\nabla F(\theta_{k-1})\|^2\big)\mathbbm1_{\mathbb{A}^n_\delta}\big]
+\gamma_k\mathbb{E}\big[\Delta_k\mathbbm1_{\mathbb{A}^n_\delta}\big]
+C\gamma_k^{1+\alpha}.
\end{aligned}
\end{equation}
Owing to \eqref{eq:stop:0} and \eqref{eq:stop:1}, given that $k-1\geq n_\delta$, one has $\|\nabla F(\theta_{k-1})\|\leq1$ on $\mathbb{A}^n_\delta$. Thus, since $\beta_\delta\in\big(0,\frac12\big]$, by \eqref{eq:mu:delta} and \eqref{eq:Lojasiewicz:delta},
\begin{equation}\label{eq:helper:1}
\mu_\delta(F(\theta_{k-1})-F_\star)
\leq\|\nabla F(\theta_{k-1})\|^\frac1{\beta_\delta}
\leq\|\nabla F(\theta_{k-1})\|^2
\quad\text{on $\mathbb{A}^n_\delta$.}
\end{equation}
Hence, combining \eqref{eq:helper:0} and \eqref{eq:helper:1},
\begin{equation}\label{eq:helper:2}
\mathbb{E}[(F(\theta_k)-F_\star)\mathbbm1_{\mathbb{A}^n_\delta}]
\leq(1-\kappa\mu_\delta\gamma_k)\mathbb{E}[(F(\theta_{k-1})-F_\star)\mathbbm1_{\mathbb{A}^n_\delta}]
+\gamma_k\mathbb{E}\big[\Delta_k\mathbbm1_{\mathbb{A}^n_\delta}\big]
+C\gamma_k^{1+\alpha}.
\end{equation}

Note that, by \eqref{eq:A:delta:n}, $(F(\theta_n)-F_\star)\mathbbm1_{\mathbb{A}^n_\delta}=(F(\theta_n)-F_\star)^+\mathbbm1_{\mathbb{A}_\delta}$. Besides, since $\gamma_n\to0$ as $n\to\infty$, there exists $n_\delta'\geq0$ such that $1-\kappa\mu_\delta\gamma_n>0$, $n\geq n_\delta'$. Thus, via Lemma~\ref{lmm:rec}, unrolling the inequality \eqref{eq:helper:2} over $n_\delta\vee n_\delta'\eqqcolon p_\delta<k\leq n$ yields
\begin{equation*}
\begin{aligned}
&\mathbb{E}[(F(\theta_n)-F_\star)^+\mathbbm1_{\mathbb{A}_\delta}]
\leq\mathbb{E}[|F(\theta_{p_\delta})-F_\star|]\prod_{k=p_\delta+1}^n(1-\kappa\mu_\delta\gamma_k)\\
&\quad+\sum_{k=p_\delta+1}^n\gamma_k\mathbb{E}\big[\Delta_k\mathbbm1_{\mathbb{A}^n_\delta}\big]\prod_{j=k+1}^n(1-\kappa\mu_\delta\gamma_j)
+C\sum_{k=p_\delta+1}^n\gamma_k^{1+\alpha}\prod_{j=k+1}^n(1-\kappa\mu_\delta\gamma_j),
\end{aligned}
\end{equation*}
hence, using \eqref{eq:sup:E:F-F*} and the inequality $1+x\leq\mathrm{e}^x$, $x\in\mathbb{R}$,
\begin{equation}\label{eq:helper:3}
\begin{aligned}
&\mathbb{E}[(F(\theta_n)-F_\star)^+\mathbbm1_{\mathbb{A}_\delta}]
\leq\big(\sup_{k\geq0}\mathbb{E}[|F(\theta_k)-F_\star|]\big)\exp\Big(\kappa\mu_\delta\sum_{k=1}^{p_\delta}\gamma_k\Big)\exp\Big(-\kappa\mu_\delta\sum_{k=1}^n\gamma_k\Big)\\
&\quad+\sum_{k=p_\delta+1}^n\gamma_k\mathbb{E}\big[\Delta_k\mathbbm1_{\mathbb{A}^n_\delta}\big]\prod_{j=k+1}^n(1-\kappa\mu_\delta\gamma_j)
+C\sum_{k=1}^n\gamma_k^{1+\alpha}\exp\Big(-\kappa\mu_\delta\sum_{j=k+1}^n\gamma_j\Big),
\end{aligned}
\end{equation}
with the conventions $\prod_\varnothing=1$ and $\sum_\varnothing=0$.

Assuming that $\kappa\mu_\delta\gamma_\star>\alpha$ in the case $\mathcal{H}$\ref{asp:gamma:misc}(\ref{asp:gamma:misc:i})\hyperref[asp:gamma:misc:ib]{b}, Lemma~\ref{lmm:gamma}(\ref{lmm:gamma:1}) gives
\begin{equation}\label{eq:helper:4}
\exp\Big(-\kappa\mu_\delta\sum_{k=1}^n\gamma_k\Big)+\sum_{k=1}^n\gamma_k^{1+\alpha}\exp\Big(-\kappa\mu_\delta\sum_{j=k+1}^n\gamma_j\Big)\leq C\gamma_n^\alpha.
\end{equation}
Furthermore, using Jensen's inequality and that $(\Delta_n)_{n\geq1}$ are martingale increments,
\begin{equation*}
\begin{aligned}
\sum_{k=p_\delta+1}^n\gamma_k\mathbb{E}\big[&\Delta_k\mathbbm1_{\mathbb{A}^n_\delta}\big]\prod_{j=k+1}^n(1-\kappa\mu_\delta\gamma_j)\\
&\leq\mathbb{E}\Big[\Big|\sum_{k=p_\delta+1}^n\gamma_k\Delta_k\prod_{j=k+1}^n(1-\kappa\mu_\delta\gamma_j)\Big|^2\Big]^\frac12
=\Big(\sum_{k=p_\delta+1}^n\gamma_k^2\mathbb{E}\big[\Delta_k^2\big]\prod_{j=k+1}^n(1-\kappa\mu_\delta\gamma_j)^2\Big)^\frac12.
\end{aligned}
\end{equation*}
Thus, supposing $2\kappa\mu_\delta\gamma_\star>1$ in the case of $\mathcal{H}$\ref{asp:gamma:misc}(\ref{asp:gamma:misc:i})\hyperref[asp:gamma:misc:ib]{b}, by \eqref{eq:E[Delta^2]<} and Lemma~\ref{lmm:gamma}(\ref{lmm:gamma:1})\hyperref[lmm:gamma:1:i]{a},
\begin{equation}\label{eq:helper:5}
\sum_{k=p_\delta+1}^n\gamma_k\mathbb{E}\big[\Delta_k\mathbbm1_{\mathbb{A}^n_\delta}\big]\prod_{j=k+1}^n(1-\kappa\mu_\delta\gamma_j)
\leq C\Big(\sum_{k=1}^n\gamma_k^2\exp\Big(-2\kappa\mu_\delta\sum_{j=k+1}^n\gamma_j\Big)\Big)^\frac12
\leq C_\delta\gamma_n^\frac12.
\end{equation}

Gathering \eqref{eq:helper:3}--\eqref{eq:helper:5} therefore yields
\begin{equation}\label{eq:E[F+]<:0:1/2}
\mathbb{E}[(F(\theta_n)-F_\star)^+\mathbbm1_{\mathbb{A}_\delta}]
\leq C_\delta\gamma_n^{\alpha\wedge\frac12}.
\end{equation}
By \eqref{eq:E[F+]<:0:1/2} and \eqref{eq:sup:E:F-F*}, up to a modification of $C_\delta$, for $n\geq0$,
\begin{equation}\label{eq:helper:6}
\mathbb{E}[(F(\theta_n)-F_\star)^+\mathbbm1_{\mathbb{A}_\delta}]
\leq\big(\sup_{k\geq0}\mathbb{E}[|F(\theta_k)-F_\star|]\sup_{1\leq k\leq n_\delta}\gamma_k^{-\alpha\wedge\frac12}+C_\delta\big)\gamma_n^{\alpha\wedge\frac12}
\leq C_\delta\gamma_n^{\alpha\wedge\frac12}.
\end{equation}
\newline
\noindent{\emph{$\stackrel{\blacktriangleright}{}$ Step~1.2. High probability bound on $(F(\theta_n))_{n\geq0}$.}}\label{step:1.2}\newline
Let $n\geq0$.
Reusing the constant $C_\delta$ from the right hand side of \eqref{eq:helper:6}, by Markov's inequality and \eqref{eq:A:delta},
\begin{equation*}
\mathbb{P}\Big((F(\theta_n)-F_\star)^+>\frac3\delta C_\delta\gamma_n^{\alpha\wedge\frac12},\tau_1<n_\delta,\mathcal{A}\subset\mathcal{L}_{\ell_\delta}\Big)
\leq\frac{\delta\mathbb{E}[(F(\theta_n)-F_\star)^+\mathbbm1_{\mathbb{A}_\delta}]}{3C_\delta\gamma_n^{\alpha\wedge\frac12}}
\leq\frac\delta3.
\end{equation*}
Thus, via \eqref{eq:n:delta},
\begin{equation*}
\mathbb{P}\Big((F(\theta_n)-F_\star)^+\leq\frac3\delta C_\delta\gamma_n^{\alpha\wedge\frac12}\Big)
\geq\mathbb{P}\Big((F(\theta_n)-F_\star)^+\leq\frac3\delta C_\delta\gamma_n^{\alpha\wedge\frac12},\tau_1<n_\delta,\mathcal{A}\subset\mathcal{L}_{\ell_\delta}\Big)
\geq1-\delta.
\end{equation*}
\newline
\noindent{\emph{$\stackrel{\blacktriangleright}{}$ Case~2. $\beta_\delta\in\big(\frac12,1\big)$.}}\newline
\noindent{\emph{$\stackrel{\blacktriangleright}{}$ Step~2.1. $L^1(\mathbb{P})$ bound on $(F(\theta_n))_{n\geq0}$.}}\newline
Reusing \eqref{eq:Delta} and $\mathcal{H}$\ref{asp:f:abc}(\ref{asp:f:abc:ii}), we have
\begin{equation}\label{eq:helper:start}
-\gamma_k\mathbb{E}\big[g_k(\theta_{k-1})^\top\nabla F(\theta_{k-1})\mathbbm1_{\mathbb{A}^n_\delta}\big]
\leq\gamma_k\mathbb{E}\big[\Delta_k\mathbbm1_{\mathbb{A}^n_\delta}\big]
-\kappa\gamma_k\mathbb{E}\big[\|\nabla F(\theta_{k-1})\|^2\mathbbm1_{\mathbb{A}^n_\delta}\big].
\end{equation}
Define the deterministic event
\begin{equation*}
\mathfrak{S}\coloneqq\Big\{\mathbb{E}\big[\|\nabla F(\theta_{k-1})\|^2\mathbbm1_{\mathbb{A}^n_\delta}\big]^\frac12\geq\nu\Big(\sum_{j=1}^k\gamma_j\Big)^{-\frac{\beta_\delta}{2\beta_\delta-1}}\Big\},
\end{equation*}
where $\nu>0$ will be determined ulteriorly.
Recalling that $2\beta_\delta>1$, by \eqref{eq:Lojasiewicz:delta} and Jensen's inequality,
\begin{equation}\label{eq:helper:end}
\begin{aligned}
\mathbb{E}\big[\|\nabla F(\theta_{k-1})\|^2\mathbbm1_{\mathbb{A}^n_\delta}\big]
&\geq\mathbb{E}\big[\|\nabla F(\theta_{k-1})\|^2\mathbbm1_{\mathbb{A}^n_\delta}\big]^\frac{2\beta_\delta-1}{2\beta_\delta}
\times\mathbb{E}\big[\|\nabla F(\theta_{k-1})\|^2\mathbbm1_{\mathbb{A}^n_\delta}\big]^\frac1{2\beta_\delta}
\times\mathbbm1_{\mathfrak{S}}\\
&\geq\frac{\nu^\frac{2\beta_\delta-1}{\beta_\delta}}{\sum_{j=1}^k\gamma_j}\mathbb{E}\big[\|\nabla F(\theta_{k-1})\|^2\mathbbm1_{\mathbb{A}^n_\delta}\big]^\frac1{2\beta_\delta}\mathbbm1_{\mathfrak{S}}\\
&=\frac{\nu^\frac{2\beta_\delta-1}{\beta_\delta}}{\sum_{j=1}^k\gamma_j}\mathbb{E}\big[\|\nabla F(\theta_{k-1})\|^2\mathbbm1_{\mathbb{A}^n_\delta}\big]^\frac1{2\beta_\delta}
-\frac{\nu^\frac{2\beta_\delta-1}{\beta_\delta}}{\sum_{j=1}^k\gamma_j}\mathbb{E}\big[\|\nabla F(\theta_{k-1})\|^2\mathbbm1_{\mathbb{A}^n_\delta}\big]^\frac1{2\beta_\delta}\mathbbm1_{\overline{\mathfrak{S}}}\\
&\geq\frac{\nu^\frac{2\beta_\delta-1}{\beta_\delta}\mu_\delta}{\sum_{j=1}^k\gamma_j}\mathbb{E}\big[|F(\theta_{k-1})-F_\star|^{2\beta_\delta}\mathbbm1_{\mathbb{A}^n_\delta}\big]^\frac1{2\beta_\delta}
-\frac{\nu^2}{\big(\sum_{j=1}^k\gamma_j\big)^{\frac{2\beta_\delta}{2\beta_\delta-1}}}\\
&\geq\frac{\nu^\frac{2\beta_\delta-1}{\beta_\delta}\mu_\delta}{\sum_{j=1}^k\gamma_j}\mathbb{E}[(F(\theta_{k-1})-F_\star)\mathbbm1_{\mathbb{A}^n_\delta}]
-\frac{\nu^2}{\big(\sum_{j=1}^k\gamma_j\big)^{\frac{2\beta_\delta}{2\beta_\delta-1}}}.
\end{aligned}
\end{equation}
Define
\begin{equation}\label{eq:bar:mu:delta}
\bar\mu_{\nu,\delta}\coloneqq\kappa\nu^\frac{2\beta_\delta-1}{\beta_\delta}\mu_\delta.
\end{equation}
Therefore, combining \eqref{eq:helper:init}, \eqref{eq:helper:start} and \eqref{eq:helper:end},
\begin{equation}\label{eq:helper:rec}
\begin{aligned}
\mathbb{E}[(F(\theta_k)-F_\star)\mathbbm1_{\mathbb{A}^n_\delta}]
&\leq\Big(1-\frac{\bar\mu_{\nu,\delta}\gamma_k}{\sum_{j=1}^k\gamma_j}\Big)\mathbb{E}[(F(\theta_{k-1})-F_\star)\mathbbm1_{\mathbb{A}^n_\delta}]\\
&\quad+\frac{C_{\nu,\delta}\gamma_k}{\big(\sum_{j=1}^k\gamma_j\big)^{\frac{2\beta_\delta}{2\beta_\delta-1}}}+\gamma_k\mathbb{E}\big[\Delta_k\mathbbm1_{\mathbb{A}^n_\delta}\big]
+C\gamma_k^{1+\alpha}.
\end{aligned}
\end{equation}

Recall that $(F(\theta_n)-F_\star)\mathbbm1_{\mathbb{A}^n_\delta}=(F(\theta_n)-F_\star)^+\mathbbm1_{\mathbb{A}_\delta}$.
Moreover, given that $\frac{\gamma_n}{\sum_{k=1}^n\gamma_k}\to0$ as $n\to\infty$ by $\mathcal{H}$\ref{asp:gamma}, for some $n_\delta'\geq0$, one has $1-\frac{\bar\mu_{\nu,\delta}\gamma_n}{\sum_{k=1}^n\gamma_k}>0$, $n\geq n_\delta'$. Thus, by Lemma~\ref{lmm:rec}, iterating \eqref{eq:helper:rec} over $n_\delta\vee n_\delta'\eqqcolon p_\delta<k\leq n$ yields
\begin{equation*}
\begin{aligned}
&\mathbb{E}[(F(\theta_n)-F_\star)^+\mathbbm1_{\mathbb{A}_\delta}]
\leq\mathbb{E}[|F(\theta_{p_\delta})-F_\star|]\prod_{k=p_\delta+1}^n\Big(1-\frac{\bar\mu_{\nu,\delta}\gamma_k}{\sum_{j=1}^k\gamma_j}\Big)\\
&\quad+C_\delta\sum_{k=p_\delta+1}^n\frac{\gamma_k}{\big(\sum_{j=1}^k\gamma_j\big)^{\frac{2\beta_\delta}{2\beta_\delta-1}}}\prod_{j=k+1}^n\Big(1-\frac{\bar\mu_{\nu,\delta}\gamma_j}{\sum_{i=1}^j\gamma_i}\Big)
+\sum_{k=p_\delta+1}^n\gamma_k\mathbb{E}\big[\Delta_k\mathbbm1_{\mathbb{A}^n_\delta}\big]\prod_{j=k+1}^n\Big(1-\frac{\bar\mu_{\nu,\delta}\gamma_j}{\sum_{i=1}^j\gamma_i}\Big)\\
&\quad+C\sum_{k=p_\delta+1}^n\gamma_k^{1+\alpha}\prod_{j=k+1}^n\Big(1-\frac{\bar\mu_{\nu,\delta}\gamma_j}{\sum_{i=1}^j\gamma_i}\Big).
\end{aligned}
\end{equation*}
Hence, utilizing \eqref{eq:sup:E:F-F*} and the inequality $1+x\leq\mathrm{e}^x$, $x\in\mathbb{R}$,
\begin{equation}\label{eq:helper:unrec}
\begin{aligned}
\mathbb{E}[&(F(\theta_n)-F_\star)^+\mathbbm1_{\mathbb{A}_\delta}]
\leq\big(\sup_{k\geq0}\mathbb{E}[|F(\theta_k)-F_\star|]\big)\exp\Big(\bar\mu_{\nu,\delta}\sum_{k=1}^{p_\delta}\frac{\gamma_k}{\sum_{j=1}^k\gamma_j}\Big)\exp\Big(-\bar\mu_{\nu,\delta}\sum_{k=1}^n\frac{\gamma_k}{\sum_{j=1}^k\gamma_j}\Big)\\
&+C\sum_{k=1}^n\frac{\gamma_k}{\big(\sum_{j=1}^k\gamma_j\big)^{\frac{2\beta_\delta}{2\beta_\delta-1}}}\exp\Big(-\bar\mu_{\nu,\delta}\sum_{j=k+1}^n\frac{\gamma_j}{\sum_{i=1}^j\gamma_i}\Big)
+\sum_{k=p_\delta+1}^n\gamma_k\mathbb{E}\big[\Delta_k\mathbbm1_{\mathbb{A}^n_\delta}\big]\prod_{j=k+1}^n\Big(1-\frac{\bar\mu_{\nu,\delta}\gamma_j}{\sum_{i=1}^j\gamma_i}\Big)\\
&+C\sum_{k=1}^n\gamma_k^{1+\alpha}\exp\Big(-\bar\mu_{\nu,\delta}\sum_{j=k+1}^n\frac{\gamma_j}{\sum_{i=1}^j\gamma_i}\Big),
\end{aligned}
\end{equation}
with the conventions $\prod_\varnothing=1$ and $\sum_\varnothing=0$.

We set $\nu$ to a value $\nu_\delta$ that is large enough so that, recalling~\eqref{eq:bar:mu:delta},
\begin{equation}\label{eq:mu>beta}
\bar\mu_\delta\coloneqq\bar\mu_{\nu_\delta,\delta}=\kappa\nu_\delta^\frac{2\beta_\delta-1}{\beta_\delta}\mu_\delta>\frac1{2\beta_\delta-1}\vee(\alpha\rho-1)\vee\frac{\rho-1}2>0.
\end{equation}
Under $\mathcal{H}$\ref{asp:gamma:misc}(\ref{asp:gamma:misc:ii},\ref{asp:gamma:misc:iii}), using~\eqref{eq:mu>beta}, Lemma~\ref{lmm:gamma}(\ref{lmm:gamma:2}) gives
\begin{equation}\label{eq:helper:terms:1}
\exp\Big(-\bar\mu_\delta\sum_{k=1}^n\frac{\gamma_k}{\sum_{j=1}^k\gamma_j}\Big)
+\sum_{k=1}^n\frac{\gamma_k}{\big(\sum_{j=1}^k\gamma_j\big)^{\frac{2\beta_\delta}{2\beta_\delta-1}}}\exp\Big(-\bar\mu_\delta\sum_{j=k+1}^n\frac{\gamma_j}{\sum_{i=1}^j\gamma_i}\Big)
\leq\frac{C_\delta}{\big(\sum_{k=1}^n\gamma_k\big)^{\frac1{2\beta_\delta-1}}},
\end{equation}
and also
\begin{equation}\label{eq:helper:terms:2}
\begin{aligned}
\sum_{k=1}^n\gamma_k^{1+\alpha}\exp\Big(&-\bar\mu_\delta\sum_{j=k+1}^n\frac{\gamma_j}{\sum_{i=1}^j\gamma_i}\Big)\\
&\leq\sum_{k=1}^n\frac{\gamma_k}{\big(\sum_{j=1}^k\gamma_j\big)^{\alpha\rho}}\exp\Big(-\bar\mu_\delta\sum_{j=k+1}^n\frac{\gamma_j}{\sum_{i=1}^j\gamma_i}\Big)
\leq\frac{C_\delta}{\big(\sum_{k=1}^n\gamma_k\big)^{\alpha\rho-1}}.
\end{aligned}
\end{equation}
Furthermore, using Jensen's inequality, that $(\Delta_n)_{n\geq1}$ are martingale increments, \eqref{eq:E[Delta^2]<}, $\mathcal{H}$\ref{asp:gamma:misc}(\ref{asp:gamma:misc:ii},\ref{asp:gamma:misc:iii}) and Lemma~\ref{lmm:gamma}(\ref{lmm:gamma:2})\hyperref[lmm:gamma:2:i]{a},
\begin{equation}\label{eq:helper:terms:3}
\begin{aligned}
\sum_{k=p_\delta+1}^n\gamma_k\mathbb{E}\big[&\Delta_k\mathbbm1_{\mathbb{A}^n_\delta}\big]\prod_{j=k+1}^n\Big(1-\frac{\bar\mu_\delta\gamma_j}{\sum_{i=1}^j\gamma_i}\Big)
\leq\Big(\sum_{k=p_\delta+1}^n\gamma_k^2\mathbb{E}\big[\Delta_k^2\big]\prod_{j=k+1}^n\Big(1-\frac{\bar\mu_\delta\gamma_j}{\sum_{i=1}^j\gamma_i}\Big)^2\Big)^\frac12\\
&\leq C\bigg(\sum_{k=1}^n\frac{\gamma_k}{\big(\sum_{j=1}^k\gamma_j\big)^\rho}\exp\Big(-2\bar\mu_\delta\sum_{j=k+1}^n\frac{\gamma_j}{\sum_{i=1}^j\gamma_i}\Big)\bigg)^\frac12
\leq\frac{C_\delta}{\big(\sum_{k=1}^n\gamma_k\big)^\frac{\rho-1}2}.
\end{aligned}
\end{equation}
Consequently, combining \eqref{eq:helper:unrec}--\eqref{eq:helper:terms:3},
\begin{equation*}
\mathbb{E}[(F(\theta_n)-F_\star)^+\mathbbm1_{\mathbb{A}_\delta}]
\leq\frac{C_\delta}{\big(\sum_{k=1}^n\gamma_k\big)^{\frac1{2\beta_\delta-1}\wedge(\alpha\rho-1)\wedge\frac{\rho-1}2}}.
\end{equation*}
Let
\begin{equation}\label{eq:r:delta}
r_\delta\coloneqq\frac1{2\beta_\delta-1}\wedge(\alpha\rho-1)\wedge\frac{\rho-1}2>0,
\quad\text{if $\beta_\delta\in\Big(\frac12,1\Big)$.}
\end{equation}
Recalling \eqref{eq:sup:E:F-F*}, up to a redefinition of $C_\delta$, for all $n\geq0$,
\begin{equation}\label{eq:E[F+]<:1/2:1}
\mathbb{E}[(F(\theta_n)-F_\star)^+\mathbbm1_{\mathbb{A}_\delta}]
\leq\frac{\sup_{k\geq0}\mathbb{E}[|F(\theta_k)-F_\star|]\big(\sum_{k=1}^{p_\delta}\gamma_k\big)^{r_\delta}+C_\delta}{\big(\sum_{k=1}^n\gamma_k\big)^{r_\delta}}\leq\frac{C_\delta}{\big(\sum_{k=1}^n\gamma_k\big)^{r_\delta}}.
\end{equation}
\newline
\noindent{\emph{$\stackrel{\blacktriangleright}{}$ Step~2.2. High probability bound on $(F(\theta_n))_{n\geq0}$.}}\label{step:2.2}\newline
Reusing $C_\delta$ from the right hand side of \eqref{eq:E[F+]<:1/2:1}, by Markov's inequality and \eqref{eq:A:delta}, for all $n\geq0$,
\begin{equation*}
\mathbb{P}\Big((F(\theta_n)-F_\star)^+>\frac3\delta C_\delta\Big(\sum_{k=1}^n\gamma_k\Big)^{r_\delta},\tau_1<n_\delta,\mathcal{A}\subset\mathcal{L}_{\ell_\delta}\Big)
\leq\frac{\delta\mathbb{E}[(F(\theta_n)-F_\star)^+\mathbbm1_{\mathbb{A}_\delta}]}{3C_\delta\big(\sum_{k=1}^n\gamma_k\big)^{r_\delta}}
\leq\frac\delta3.
\end{equation*}
All in all, via \eqref{eq:n:delta},
\begin{equation*}
\begin{aligned}
\mathbb{P}\Big((F(\theta_n)-F_\star)^+&\leq\frac3\delta C_\delta\Big(\sum_{k=1}^n\gamma_k\Big)^{r_\delta}\Big)\\
&\geq\mathbb{P}\Big((F(\theta_n)-F_\star)^+\leq\frac3\delta C_\delta\Big(\sum_{k=1}^n\gamma_k\Big)^{r_\delta},\tau_1<n_\delta,\mathcal{A}\subset\mathcal{L}_{\ell_\delta}\Big)
\geq1-\delta.
\end{aligned}
\end{equation*}
\newline
\noindent{\bf(\ref{thm:F-F*:special:cases:i})\hyperref[thm:F-F*:special:cases:i:b]{b}.}\
The proof is similar to the prior one, once $\mu_\delta$ is set to $1$.
\\

\noindent{\bf(\ref{thm:F-F*:special:cases:ii})}\
We put ourselves in the framework of the proof of Theorem~\ref{thm:F-F*}(\ref{thm:F-F*:ii}). Define
\begin{equation}\label{eq:B:delta:n}
\mathbb{B}_\delta^n
\coloneqq\mathbb{B}_\delta\cap\{F(\theta_n)>F_\star\}
=\{\tau_2<n_\delta\}\cap\{F(\theta_n)>F_\star\},
\quad n>n_\delta.
\end{equation}
Swapping $\beta_\delta\gets\beta$, $\mu_\delta\gets1$ and $\mathbb{A}^n_\delta\gets\mathbb{B}^n_\delta$ and setting $\kappa\gamma_\star>\alpha\vee\frac12$ under $\mathcal{H}$\ref{asp:gamma:misc}(\ref{asp:gamma:misc:i})\hyperref[asp:gamma:misc:ib]{b}, the previous proof line by line shows that, for all $n\geq0$,
\begin{equation*}
\mathbb{E}[(F(\theta_n)-F_\star)^+\mathbbm1_{\mathbb{B}_\delta}]
\leq C_\delta
\begin{cases}
\gamma_n^{\alpha\wedge\frac12},
&\text{if $\beta\in(0,\frac12]$,}\\
\big(\sum_{k=1}^n\gamma_k\big)^{-r},
&\text{if $\beta\in(\frac12,1]$,}
\end{cases}
\end{equation*}
where
\begin{equation}\label{eq:r}
r\coloneqq\frac1{2\beta-1}\wedge(\alpha\rho-1)\wedge\frac{\rho-1}2>0,
\quad\text{if $\beta\in\Big(\frac12,1\Big]$.}
\end{equation}
The high probability bound is then retrieved via Markov's inequality similarly to \eqref{eq:high:probability}.
\\

\noindent{\bf(\ref{thm:F-F*:special:cases:iii})}\
According to Proposition~\ref{prp:Lojasiewicz}(\ref{prp:Lojasiewicz:iii}) and Theorem~\ref{thm:cv}(\ref{thm:cv:i}), one has $F_\star=\inf F$ $\Pas$.
Besides, there exists $(\beta,\zeta)\in\big[\frac\alpha{1+\alpha},1\big)\times\mathbb{R}^*_+$ such that \eqref{eq:framework:3} holds.

Let $n>0$ and set
\begin{equation}\label{eq:mu}
\mu=\zeta^{-\frac1\beta}.
\end{equation}
Taking the expectation in \eqref{eq:quasi:super:martingale} and using \eqref{eq:sup:grad:f:1+alpha},
\begin{equation}\label{eq:main}
\mathbb{E}[F(\theta_n)-\inf F]
\leq\mathbb{E}[F(\theta_{n-1})-\inf F]
-\kappa\gamma_n\mathbb{E}\big[\|\nabla F(\theta_{n-1})\|^2\big]+C\gamma_n^{1+\alpha}.
\end{equation}
Assuming that $\beta\in\big(\frac\alpha{1+\alpha},1\big)\cup\big\{\frac12\big\}$, we distinguish between three cases according to the values of $\alpha$ and $\beta$.
\\

\noindent{\emph{$\stackrel{\blacktriangleright}{}$ Case~1. $\beta\in\big(\frac\alpha{1+\alpha},\frac12\big)$, $\alpha\in(0,1)$.}}\newline
Let $\lambda>0$ and $\theta\in\mathbb{R}^m$ such that $F(\theta)>\inf F$. Then, via $\mathcal{H}$\ref{asp:f}, Lemma~\ref{lmm:Holder}(\ref{lmm:Holder:ii}) and \eqref{eq:framework:3},
\begin{equation*}
\begin{aligned}
\|\nabla F(\theta)\|^2
&\geq\Big(\frac\alpha{(1+\alpha)L^\frac1\alpha}\Big)^\lambda\frac{\|\nabla F(\theta)\|^{2+\lambda\frac{1+\alpha}\alpha}}{(F(\theta)-\inf F)^\lambda}\\
&\geq\zeta^{-2-\lambda\frac{1+\alpha}\alpha}\Big(\frac\alpha{(1+\alpha)L^\frac1\alpha}\Big)^\lambda(F(\theta)-\inf F)^{2\beta+\lambda(\frac{1+\alpha}\alpha\beta-1)}.
\end{aligned}
\end{equation*}
Setting
\begin{equation*}
\lambda=\frac{1-2\beta}{\frac{1+\alpha}\alpha\beta-1}>0
\quad\text{and}\quad
\widetilde\mu=\zeta^{-2-\lambda\frac{1+\alpha}\alpha}\Big(\frac\alpha{(1+\alpha)L^\frac1\alpha}\Big)^\lambda
\end{equation*}
shows that
\begin{equation}\label{eq:Lojasiewicz:beta<1/2=>beta=1/2}
\widetilde\mu(F(\theta)-\inf F)\leq\|\nabla F(\theta)\|^2.
\end{equation}
This relation remains true if $F(\theta)=\inf F$.
The remaining lines of reasoning coincide with the next case by swapping $\mu$ with $\widetilde\mu$.
\\

\noindent{\emph{$\stackrel{\blacktriangleright}{}$ Case~2. $\beta=\frac12$.}}\newline
By \eqref{eq:main}, \eqref{eq:framework:3} and \eqref{eq:mu},
\begin{equation*}
\mathbb{E}[F(\theta_n)-\inf F]
\leq(1-\kappa\mu\gamma_n)\mathbb{E}[F(\theta_{n-1})-\inf F]+C\gamma_n^{1+\alpha}.
\end{equation*}
Note that, since $\gamma_n\to0$ as $n\to\infty$, there exists $n_0$ such that $\kappa\mu\gamma_n<1$, $n\geq n_0$.
Thus, using Lemma~\ref{lmm:rec}, for $n\geq n_0$,
\begin{equation*}
\mathbb{E}[F(\theta_n)-\inf F]
\leq\mathbb{E}[F(\theta_{n_0})-\inf F]\prod_{k=n_0+1}^n(1-\kappa\mu\gamma_k)+C\sum_{k=n_0+1}^n\gamma_k^{1+\alpha}\prod_{j=k+1}^n(1-\kappa\mu\gamma_j),
\end{equation*}
with the convention $\prod_\varnothing=1$. But then, utilizing that $1+x\leq\mathrm{e}^x$, $x\in\mathbb{R}$, and that $\sup_{k\geq0}\mathbb{E}[F(\theta_k)-\inf F]<\infty$ by Lemma~\ref{lmm:F(theta):converges}(\ref{lmm:F(theta):converges:ii}),
\begin{equation*}
\begin{aligned}
\mathbb{E}[F(\theta_n)-\inf F]
\leq\big(\sup_{k\geq0}\mathbb{E}[F(\theta_k)-\inf F]\big)\exp\Big(&\kappa\mu\sum_{k=1}^{n_0}\gamma_k\Big)\exp\Big(-\kappa\mu\sum_{k=1}^n\gamma_k\Big)\\
&+C\sum_{k=1}^n\gamma_k^{1+\alpha}\exp\Big(-\kappa\mu\sum_{j=k+1}^n\gamma_j\Big),
\end{aligned}
\end{equation*}
with the convention $\sum_\varnothing=0$.
Assuming $\kappa\mu\gamma_\star>\alpha$ in the case $\mathcal{H}$\ref{asp:gamma:misc}(\ref{asp:gamma:misc:i})\hyperref[asp:gamma:misc:ib]{b}, invoking Lemma~\ref{lmm:gamma}(\ref{lmm:gamma:1}) yields
\begin{equation}
\mathbb{E}[F(\theta_n)-\inf F]\leq C\gamma_n^\alpha.
\end{equation}
\newline
\noindent{\emph{$\stackrel{\blacktriangleright}{}$ Case~3. $\beta\in\big(\frac12,1\big)$.}}\newline
Let
\begin{equation*}
\mathfrak{S}\coloneqq\Big\{\mathbb{E}\big[\|\nabla F(\theta_{n-1})\|^2\big]^\frac12\geq\nu\Big(\sum_{k=1}^n\gamma_k\Big)^{-\frac\beta{2\beta-1}}\Big\},
\end{equation*}
for some $\nu>0$ that will be selected later.
Observe that, by \eqref{eq:framework:3}, $\mathbb{E}[(F(\theta_{n-1})-\inf F)^{2\beta}]\leq \zeta^2\mathbb{E}[\|\nabla F(\theta_{n-1})\|^2]<\infty$.
Thus, via \eqref{eq:mu} and Jensen's inequality,
\begin{equation*}
\begin{aligned}
\mathbb{E}\big[\|\nabla F(\theta_{n-1})\|^2\big]
&\geq\mathbb{E}\big[\|\nabla F(\theta_{n-1})\|^2\big]^\frac{2\beta-1}{2\beta}
\times\mathbb{E}\big[\|\nabla F(\theta_{n-1})\|^2\big]^\frac1{2\beta}
\times\mathbbm1_{\mathfrak{S}}\\
&\geq\frac{\nu^\frac{2\beta-1}\beta}{\sum_{k=1}^n\gamma_k}\mathbb{E}\big[\|\nabla F(\theta_{n-1})\|^2\big]^\frac1{2\beta}\mathbbm1_{\mathfrak{S}}\\
&=\frac{\nu^\frac{2\beta-1}\beta}{\sum_{k=1}^n\gamma_k}\mathbb{E}\big[\|\nabla F(\theta_{n-1})\|^2\big]^\frac1{2\beta}
-\frac{\nu^\frac{2\beta-1}\beta}{\sum_{k=1}^n\gamma_k}\mathbb{E}\big[\|\nabla F(\theta_{n-1})\|^2\big]^\frac1{2\beta}\mathbbm1_{\overline{\mathfrak{S}}}\\
&\geq\frac{\nu^\frac{2\beta-1}\beta\mu}{\sum_{k=1}^n\gamma_k}\mathbb{E}[(F(\theta_{n-1})-\inf F)^{2\beta}]^\frac1{2\beta}
-\frac{\nu^2}{\big(\sum_{k=1}^n\gamma_k\big)^{\frac{2\beta}{2\beta-1}}}\\
&\geq\frac{\nu^\frac{2\beta-1}\beta\mu}{\sum_{k=1}^n\gamma_k}\mathbb{E}[F(\theta_{n-1})-\inf F]
-\frac{\nu^2}{\big(\sum_{k=1}^n\gamma_k\big)^{\frac{2\beta}{2\beta-1}}}.
\end{aligned}
\end{equation*}
Let
\begin{equation*}
\bar\mu_\nu\coloneqq\kappa\nu^\frac{2\beta-1}\beta\mu.
\end{equation*}
Hence, using \eqref{eq:main} and $\mathcal{H}$\ref{asp:gamma:misc}(\ref{asp:gamma:misc:iii}),
\begin{equation*}
\mathbb{E}[F(\theta_n)-\inf F]
\leq\Big(1-\frac{\bar\mu_\nu\gamma_n}{\sum_{k=1}^n\gamma_k}\Big)\mathbb{E}[F(\theta_{n-1})-\inf F]+\frac{C_\nu\gamma_n}{\big(\sum_{k=1}^n\gamma_k\big)^\frac{2\beta}{2\beta-1}}+\frac{C\gamma_n}{\big(\sum_{k=1}^n\gamma_k\big)^{\alpha\rho}}.
\end{equation*}
Because $\frac{\gamma_n}{\sum_{k=1}^n\gamma_k}\to0$ as $n\to\infty$, there exists $n_0$ such that $\frac{\bar\mu_\nu\gamma_n}{\sum_{k=1}^n\gamma_k}<1$, $n\geq n_0$.
But then, by Lemma~\ref{lmm:rec}, for $n\geq n_0$,
\begin{equation*}
\begin{aligned}
\mathbb{E}[&F(\theta_n)-\inf F]
\leq\mathbb{E}[F(\theta_{n_0})-\inf F]\prod_{k=n_0+1}^n\Big(1-\frac{\bar\mu_\nu\gamma_k}{\sum_{j=1}^k\gamma_j}\Big)\\
&+\sum_{k=n_0+1}^n\frac{C_\nu\gamma_k}{\big(\sum_{j=1}^k\gamma_j\big)^\frac{2\beta}{2\beta-1}}\prod_{j=k+1}^n\Big(1-\frac{\bar\mu_\nu\gamma_j}{\sum_{i=1}^j\gamma_i}\Big)
+\sum_{k=n_0+1}^n\frac{C\gamma_k}{\big(\sum_{j=1}^k\gamma_j\big)^{\alpha\rho}}\prod_{j=k+1}^n\Big(1-\frac{\bar\mu_\nu\gamma_j}{\sum_{i=1}^j\gamma_i}\Big),
\end{aligned}
\end{equation*}
with the convention $\prod_\varnothing=1$.
Then, using that $1+x\leq\mathrm{e}^x$, $x\in\mathbb{R}$ and that $\sup_{k\geq0}\mathbb{E}[F(\theta_k)-\inf F]<\infty$ by Lemma~\ref{lmm:F(theta):converges}(\ref{lmm:F(theta):converges:ii}),
\begin{equation*}
\begin{aligned}
\mathbb{E}[F(\theta_n)-\inf F]
\leq\big(&\sup_{k\geq0}\mathbb{E}[F(\theta_k)-\inf F]\big)\exp\Big(\bar\mu_\nu\sum_{k=1}^{n_0}\frac{\gamma_k}{\sum_{j=1}^k\gamma_j}\Big)\exp\Big(-\bar\mu_\nu\sum_{k=1}^n\frac{\gamma_k}{\sum_{j=1}^k\gamma_j}\Big)\\
&+\sum_{k=1}^n\frac{C_\nu\gamma_k}{\big(\sum_{j=1}^k\gamma_j\big)^{\frac{2\beta}{2\beta-1}}}\exp\Big(-\bar\mu_\nu\sum_{j=k+1}^n\frac{\gamma_j}{\sum_{i=1}^j\gamma_i}\Big)\\
&+\sum_{k=1}^n\frac{C\gamma_k}{\big(\sum_{j=1}^k\gamma_j\big)^{\alpha\rho}}\exp\Big(-\bar\mu_\nu\sum_{j=k+1}^n\frac{\gamma_j}{\sum_{i=1}^j\gamma_i}\Big),
\end{aligned}
\end{equation*}
with the convention $\sum_\varnothing=0$.
We select $\nu$ such that
\begin{equation*}
\bar\mu\coloneqq\bar\mu_\nu=\kappa\nu^\frac{2\beta-1}\beta\mu>\frac1{2\beta-1}\vee(\alpha\rho-1)>0.
\end{equation*}
Via Lemma~\ref{lmm:gamma}(\ref{lmm:gamma:2}), we deduce
\begin{equation}
\mathbb{E}[F(\theta_n)-\inf F]\leq\frac{C}{\big(\sum_{k=1}^n\gamma_k\big)^r},
\end{equation}
where $r=\frac1{2\beta-1}\wedge(\alpha\rho-1)>0$.
\\

The subsequent high probability bounds are obtained using Markov's inequality.
\end{proof}

\begin{proof}[Proof of Corollary~\ref{crl:cost}]
This is a direct consequence of Theorems~\ref{thm:F-F*} and~\ref{thm:F-F*:special:cases}.
\end{proof}

\section{Global Convergence Proofs}
\label{sec:global}

Henceforward, for a real-valued sequence $(v_n)_{n\geq0}$, we denote its discrete gradient
\begin{equation*}
\Delta v_n\coloneqq v_n-v_{n-1},\quad n\geq1.
\end{equation*}

Our approach for proving global convergence uses a deterministic vanishing nonincreasing Lyapunov sequence $(G_n)_{n\geq0}$ \eqref{eq:def:G} to help prove the $L^1(\mathbb{P})$-summability of $(\gamma_n\nabla F(\theta_{n-1}))_{n\geq1}$.
We also refer to the formalism developed in~\cite{CFR23} that relies on extending the K{\L} property to a Lyapunov function $V$ that is built so as $(V(\theta_n))_{n\geq0}$ is a converging supermartingale.

The next lemma is a short interlude providing a useful property.

\begin{lemma}\label{lmm:sum:Delta}
Recalling the definition \eqref{eq:Delta}, under $\mathcal{H}$\ref{asp:f}--\ref{asp:gamma}, the martingale $\big(\sum_{k=1}^n\gamma_k\Delta_k\big)_{n\geq1}$ converges in $L^2(\mathbb{P})$ and $\Pas$.
\end{lemma}

\begin{proof}
Considering \eqref{eq:E[Delta^2]<}, define the $L^2(\mathbb{P})$ $(\mathcal{F}_n)_{n\geq0}$-martingale
\begin{equation*}
M'_n\coloneqq\sum_{k=1}^n\gamma_k\Delta_k,\quad n\geq1.
\end{equation*}
Since $(\Delta_k)_{k\geq1}$ are martingale increments, using \eqref{eq:E[Delta^2]<} and that $\sum_{k=1}^\infty\gamma_k^2<\infty$ by $\mathcal{H}$\ref{asp:gamma},
\begin{equation*}
\sup_{n\geq1}\mathbb{E}[|M'_n|^2]
=\sum_{k=1}^\infty\gamma_k^2\mathbb{E}[\Delta_k^2]
\leq C\sum_{k=1}^\infty\gamma_k^2
<\infty.
\end{equation*}
$(M'_n)_{n\geq1}$ is therefore uniformly bounded in $L^2(\mathbb{P})$ and hence converges in $L^2(\mathbb{P})$ and $\Pas$.
\end{proof}

\begin{proof}[Proof of Theorem~\ref{thm:iterates}]
\noindent{\emph{$\stackrel{\blacktriangleright}{}$ Step~1. Final hitting time.}}\newline
Define
\begin{equation}\label{eq:stop:prime}
\tau'\coloneqq\inf\Big\{n\geq0\colon\forall k\geq n,|F(\theta_k)-F_\star|+\Big|\sum_{j=k+1}^\infty\gamma_j\Delta_j\Big|+\frac{L}{1+\alpha}\sum_{j=k+1}^\infty\gamma_j^{1+\alpha}\|g_j(\theta_{j-1})\|^{1+\alpha}
\leq1\Big\}.
\end{equation}
By Lemmas~\ref{lmm:F(theta):converges}(\ref{lmm:F(theta):converges:i},\ref{lmm:F(theta):converges:iii}) and~\ref{lmm:sum:Delta},
\begin{equation}\label{eq:lim:0}
\lim_{k\to\infty}|F(\theta_k)-F_\star|+\Big|\sum_{j=k+1}^\infty\gamma_j\Delta_j\Big|+\frac{L}{1+\alpha}\sum_{j=k+1}^\infty\gamma_j^{1+\alpha}\|g_j(\theta_{j-1})\|^{1+\alpha}=0
\quad\Pas,
\end{equation}
thus $\tau'<\infty$ $\Pas$.

We reconsider the framework of Theorem~\ref{thm:F-F*:special:cases}(\ref{thm:F-F*:special:cases:i})\hyperref[thm:F-F*:special:cases:i:b]{b} and, recalling \eqref{eq:levelset} and \eqref{eq:stop:1}, we define
\begin{equation}\label{eq:stop:tilde}
\widetilde\tau\coloneqq\tau_1\vee\tau'.
\end{equation}
Let $\delta\in(0,1)$. Since $\widetilde\tau<\infty$ $\Pas$ on $\{\mathcal{A}\subset\mathcal{L}_{\ell_\delta}\}$, there exists $n_\delta\geq0$ such that
\begin{equation}\label{eq:P:tau:tilde>1-delta}
\mathbb{P}(\widetilde\tau<n_\delta,\mathcal{A}\subset\mathcal{L}_{\ell_\delta})\geq1-\delta.
\end{equation}
\newline
\noindent{\emph{$\stackrel{\blacktriangleright}{}$ Step~2. Lyapunov sequence.}}\newline
Define
\begin{equation*}
\mathbb{C}_\delta\coloneqq\{\widetilde\tau<n_\delta\}\cap\{\mathcal{A}\subset\mathcal{L}_{\ell_\delta}\}.
\end{equation*}
Let $n>n_\delta$.
We hereby define a Lyapunov sequence $(G_n)_{n\geq n_\delta}$ associated to $(F(\theta_n))_{n\geq n_\delta}$ and provide some subsequent properties.

By \eqref{eq:F(theta:n)<F(theta:n-1)}, \eqref{eq:stop:tilde}, \eqref{eq:stop:1}, \eqref{eq:stop:0} and \eqref{eq:Delta},
\begin{equation}\label{eq:descent}
\begin{aligned}
(F(\theta_n)&-F_\star)\mathbbm1_{\mathbb{C}_\delta}\\
&\leq\Big(F(\theta_{n-1})-F_\star-\kappa\gamma_n\|\nabla F(\theta_{n-1})\|^2+\gamma_n\Delta_n+\frac{L}{1+\alpha}\gamma_n^{1+\alpha}\| g_n(\theta_{n-1})\|^{1+\alpha}\Big)\mathbbm1_{\mathbb{C}_\delta}.
\end{aligned}
\end{equation}
Let
\begin{equation}\label{eq:def:G}
G_n\coloneqq\kappa^{-1}\mathbb{E}\Big[\Big(F(\theta_n)-F_\star+\sum_{k=n+1}^\infty\gamma_k\Delta_k+\frac{L}{1+\alpha}\sum_{k=n+1}^\infty\gamma_k^{1+\alpha}\|g_k(\theta_{k-1})\|^{1+\alpha}\Big)\mathbbm1_{\mathbb{C}_\delta}\Big],
\quad n\geq n_\delta,
\end{equation}
which is well defined according to Lemmas~\ref{lmm:F(theta):converges}(\ref{lmm:F(theta):converges:i},\ref{lmm:F(theta):converges:iii}) and~\ref{lmm:sum:Delta}.
Then, by \eqref{eq:descent},
\begin{equation}\label{eq:grad<Delta}
\gamma_n\mathbb{E}\big[\|\nabla F(\theta_{n-1})\|^2\mathbbm1_{\mathbb{C}_\delta}\big]\leq-\Delta G_n.
\end{equation}
Hence $(G_n)_{n\geq n_\delta}$ is nonincreasing. Via \eqref{eq:stop:prime} and \eqref{eq:lim:0}, by dominated convergence, $G_n\to0$ as $n\to\infty$.
Therefore $G_n\geq0$, $n\geq n_\delta$.

Note that, since $\widetilde\tau\geq\tau_1$ $\Pas$, it follows from \eqref{eq:P:tau:tilde>1-delta} that $\mathbb{P}(\tau_1<n_\delta,\mathcal{A}\subset\mathcal{L}_{\ell_\delta})\geq1-\delta$, so that, via \eqref{eq:E[F+]<:0:1/2} and \eqref{eq:E[F+]<:1/2:1},
\begin{equation}\label{eq:G:term:1}
\mathbb{E}[(F(\theta_{n-1})-F_\star)^+\mathbbm1_{\mathbb{C}_\delta}]
\leq\mathbb{E}[(F(\theta_{n-1})-F_\star)^+\mathbbm1_{\mathbb{A}_\delta}]
\leq C_\delta u_n,
\end{equation}
where
\begin{equation}\label{eq:u}
u_n\coloneqq 
\begin{cases}
n^{-s(\alpha\wedge\frac12)},
&\text{if $\beta_\delta\in(0,\frac12]$ and $s\in(\frac1{1+\alpha},1]$, with $\kappa\gamma_\star>\frac12\vee\alpha$ if $s=1$,}\\
n^{-(1-s)r_\delta},
&\text{if $\beta_\delta\in(\frac12,1)$ and $s\in[\frac\rho{1+\rho},1)$,}\\
(\ln{n})^{-r_\delta},
&\text{if $\beta_\delta\in(\frac12,1)$ and $s=1$,}\\
\end{cases}
\qquad n\geq n_\delta,
\end{equation}
with $\rho>\frac1\alpha$ and $r_\delta$ defined in \eqref{eq:r:delta}.

By monotone convergence, Jensen's inequality, Doob's maximal inequality (or the Burkholder-Davis-Gundy inequality), \eqref{eq:E[Delta^2]<} and a series-integral comparison,
\begin{equation}\label{eq:G:term:2}
\begin{aligned}
\mathbb{E}\Big[\Big|\sum_{k=n}^\infty\gamma_k\Delta_k\Big|\Big]
&\leq\mathbb{E}\Big[\sup_{p\geq n}\Big|\sum_{k=n}^p\gamma_k\Delta_k\Big|\Big]\\
&=\lim_{N\to\infty}\mathbb{E}\Big[\sup_{n\leq p\leq N}\Big|\sum_{k=n}^p\gamma_k\Delta_k\Big|^2\Big]^\frac12
\leq\Big(\sum_{k=n}^\infty\gamma_k^2\mathbb{E}[\Delta_k^2]\Big)^\frac12
\leq\frac{C}{n^\frac{2s-1}2}.
\end{aligned}
\end{equation}
Moreover, by Fubini's theorem, \eqref{eq:sup:grad:f:1+alpha} and a series-integral comparison,
\begin{equation}\label{eq:G:term:3}
\begin{aligned}
\mathbb{E}\Big[\sum_{k=n}^\infty\gamma_k^{1+\alpha}\|g_k(\theta_{k-1})\|^{1+\alpha}\Big]
\leq\frac{C}{n^{(1+\alpha)s-1}}.
\end{aligned}
\end{equation}
Thus, gathering \eqref{eq:G:term:1}--\eqref{eq:G:term:3} and recalling \eqref{eq:def:G},
\begin{equation}\label{eq:G<}
G_{n-1}\leq C_\delta\Big(u_n+\frac1{n^\frac{2s-1}2}+\frac1{n^{(1+\alpha)s-1}}\Big).
\end{equation}
\newline
\noindent{\emph{$\stackrel{\blacktriangleright}{}$ Step~3. Summability of $(\gamma_n\|\nabla F(\theta_{n-1})\|)_{n\geq1}$.}}\newline
Let $s\in\big[\frac\rho{1+\rho},1]$, $\rho>\frac1\alpha$, and $\sigma\in(0,1)$.
Assume $\kappa\gamma_\star>\frac12\vee\alpha$ if $s=1$.

On the one hand, note that if $\mathbb{E}[\|\nabla F(\theta_{n-1})\|\mathbbm1_{\mathbb{C}_\delta}]>0$, necessarily, $G_{n-1}>0$. Otherwise, since $0\leq G_n\leq G_{n-1}$, $G_n=0$. But then, $\mathbb{E}[\|\nabla F(\theta_{n-1})\|\mathbbm1_{\mathbb{C}_\delta}]\leq\sqrt{\gamma_n^{-1}(-\Delta G_n)}=0$, which is absurd.
Thus, using Jensen's inequality, \eqref{eq:grad<Delta} and Lemma~\ref{lmm:concave}, if ${\mathbb{E}[\|\nabla F(\theta_{n-1})\|\mathbbm1_{\mathbb{C}_\delta}]>0}$,
\begin{equation}\label{eq:A<}
\begin{aligned}
\gamma_n\mathbb{E}\big[&\|\nabla F(\theta_{n-1})\|\mathbbm1_{\mathbb{C}_\delta}\big]\mathbbm1_{\mathbb{E}[\|\nabla F(\theta_{n-1})\|\mathbbm1_{\mathbb{C}_\delta}]>G_{n-1}^\sigma}\\
&\leq\gamma_nG_{n-1}^{-\sigma}\mathbb{E}\big[\|\nabla F(\theta_{n-1})\|^2\mathbbm1_{\mathbb{C}_\delta}\big]
\leq G_{n-1}^{-\sigma}(-\Delta G_n)
\leq\frac{-\Delta G_n^{1-\sigma}}{1-\sigma},
\end{aligned}
\end{equation}
recalling that $-\Delta G_n^{1-\sigma}=G_{n-1}^{1-\sigma}-G_n^{1-\sigma}$.
This holds also if $\mathbb{E}[\|\nabla F(\theta_{n-1})\|\mathbbm1_{\mathbb{C}_\delta}]=0$.

On the other hand, using \eqref{eq:G<} and \eqref{eq:u},
\begin{equation}\label{eq:B<}
\begin{aligned}
\gamma_n\mathbb{E}\big[\|\nabla F(\theta_{n-1})\|\mathbbm1_{\mathbb{C}_\delta}\big]&\mathbbm1_{\mathbb{E}[\|\nabla F(\theta_{n-1})\|\mathbbm1_{\mathbb{C}_\delta}]\leq G_{n-1}^\sigma}\\
&\leq\gamma_nG_{n-1}^\sigma
\leq C_\delta\Big(\frac{u_n^\sigma}{n^s}+\frac1{n^{s+\frac{2s-1}2\sigma}}+\frac1{n^{s+((1+\alpha)s-1)\sigma}}\Big).
\end{aligned}
\end{equation}

All in all, gathering the inequalities \eqref{eq:A<} and \eqref{eq:B<},
\begin{equation}\label{eq:gamma:grad:<}
\begin{aligned}
\gamma_n&\mathbb{E}\big[\|\nabla F(\theta_{n-1})\|\mathbbm1_{\mathbb{C}_\delta}\big]\\
&\leq\begin{cases}
\frac{C_\delta}{1-\sigma}\big(n^{-s-s(\alpha\wedge\frac12)\sigma}+n^{-s-\frac{2s-1}2\sigma}+n^{-s-((1+\alpha)s-1)\sigma}
-\Delta G_n^{1-\sigma}\big)
&\text{if $\beta_\delta\in(0,\frac12]$,}\\
\frac{C_\delta}{1-\sigma}\big(n^{-s-(1-s)r_\delta\sigma}+n^{-s-\frac{2s-1}2\sigma}+n^{-s-((1+\alpha)s-1)\sigma}
-\Delta G_n^{1-\sigma}\big)
&\text{if $\beta_\delta\in(\frac12,1)$ and $s<1$,}\\
\frac{C_\delta}{1-\sigma}\big(n^{-1}(\ln{n})^{-r_\delta\sigma}+n^{-1-\frac\sigma2}+n^{-1-\alpha\sigma}
-\Delta G_n^{1-\sigma}\big)
&\text{if $\beta_\delta\in(\frac12,1)$ and $s=1$.}
\end{cases}
\end{aligned}
\end{equation}
Setting $\rho>3\vee\frac2\alpha$, so that $s\in\big(\frac2{2+\alpha}\vee\frac34,1\big]$ and $r_\delta>1$ if $\beta_\delta\in\big(\frac12,1\big)$, then taking
\begin{equation*}
\sigma\in\begin{cases}
\big(\frac{1-s}{s(\alpha\wedge\frac12)}\vee\frac{2(1-s)}{2s-1}\vee\frac{1-s}{(1+\alpha)s-1},1\big)
&\text{if $\beta_\delta\in(0,\frac12]$,}\\
\big(\frac1{r_\delta}\vee\frac{2(1-s)}{2s-1}\vee\frac{1-s}{(1+\alpha)s-1},1\big)
&\text{if $\beta_\delta\in(\frac12,1)$ and $s<1$,}\\
\big(\frac1{r_\delta},1\big)
&\text{if $\beta_\delta\in(\frac12,1)$ and $s=1$,}
\end{cases}
\end{equation*}
and finally recalling that $G_n\to0$ as $n\to\infty$, we deduce that $\sum_{n=n_\delta}^\infty\gamma_n\mathbb{E}[\|\nabla F(\theta_{n-1})\|\mathbbm1_{\mathbb{C}_\delta}]<\infty$.
\\

\noindent{\emph{$\stackrel{\blacktriangleright}{}$ Step~4. Conclusion.}}\newline
By Fubini's theorem, $\sum_{n=1}^\infty\gamma_n\|\nabla F(\theta_{n-1})\|<\infty$ $\Pas$ on $\mathbb{C}_\delta$, thus on $\cup_{p\geq0}\mathbb{C}_{\delta/2^p}$, which is an event of probability $\mathbb{P}(\cup_{p\geq0}\mathbb{C}_{\delta/2^p})\geq\sup_{p\geq0}\big(1-\frac\delta{2^p}\big)=1$.
Therefore $\sum_{n=1}^\infty\gamma_n\|\nabla F(\theta_{n-1})\|<\infty$ $\Pas$ and $(\gamma_n\nabla F(\theta_{n-1}))_{n\geq1}$ is $\Pas$-absolutely summable.
Recalling $\mathcal{H}$\ref{asp:f:abc}(\ref{asp:f:abc:iii}), it ensues that $(\gamma_nb(\theta_{n-1}))_{n\geq1}$ is also $\Pas$-absolutely summable.
Finally, note that
\begin{equation}\label{eq:theta:n=sum}
\theta_n=\theta_0+\sum_{k=1}^n\gamma_k(b(\theta_{k-1})-g_k(\theta_{k-1}))-\sum_{k=1}^n\gamma_kb(\theta_{k-1}),
\quad n\geq1.
\end{equation}
By Lemma~\ref{lmm:F(theta):converges}(\ref{lmm:F(theta):converges:iv}) and the previous finding, $(\theta_n)_{n\geq0}$ converges $\Pas$ toward a $\Pas$ finite random variable $\theta_\star$.
Besides, by the continuity of $F$ and $\nabla F$ via $\mathcal{H}$\ref{asp:f}, $\nabla F(\theta_\star)=\lim_{n\to\infty}\nabla F(\theta_n)=0$ $\Pas$ via Lemma~\ref{lmm:gradF->0} and $F(\theta_\star)=\lim_{n\to\infty}F(\theta_n)=F_\star$ $\Pas$ via Lemma~\ref{lmm:F(theta):converges}(\ref{lmm:F(theta):converges:i}).
\end{proof}

The telescopic series decomposition \eqref{eq:theta:n=sum->8} suggests that the convergence speed of $(\theta_n)_{n\geq1}$ results from a dual contribution from the martingale residual $\big(\sum_{k=n+1}^\infty\gamma_k(g_k(\theta_{k-1})-b(\theta_k))\big)_{n\geq1}$ and the drift residual $\big(\sum_{k=n+1}^\infty\gamma_kb(\theta_{k-1})\big)_{n\geq1}$.
A classical interplay between the Jensen and Burkholder-Davis-Gundy inequalities shows that the martingale residual behaves like $\sqrt{\sum_{k=n+1}^\infty\gamma_k^2}\\=\mathrm{O}(n^{-\frac{2s-1}2})$. As for the drift residual, we could rely on the bound \eqref{eq:gamma:grad:<} that has been obtained constructively in Theorem~\ref{thm:iterates}'s proof. The ensuing bound requires however to strike a balance between terms with exponent $\sigma$ and terms with exponent $1-\sigma$, where $\sigma\in(0,1)$, in order to reach optimality.
To avoid such tradeoff, we seek alternative bounds on the trend residual in Corollary~\ref{crl:iterates:speed}'s proof.

\begin{proof}[Proof of Corollary~\ref{crl:iterates:speed}]
\noindent{\bf(\ref{crl:iterates:speed:i})\hyperref[crl:iterates:speed:i:a]{a}.}\
{\emph{$\stackrel{\blacktriangleright}{}$ Step~1. Preliminaries.}}\newline
We reuse notation from the preceding proof, bearing in mind that we are in the framework of Theorem~\ref{thm:F-F*:special:cases}(\ref{thm:F-F*:special:cases:i})\hyperref[thm:F-F*:special:cases:i:a]{a}.
Via the $\Pas$-absolute summability of $(\gamma_n\nabla F(\theta_{n-1}))_{n\geq1}$ and Lemma~\ref{lmm:F(theta):converges}(\ref{lmm:F(theta):converges:iv}), by telescopic summation,
\begin{equation}\label{eq:theta:n=sum->8}
\theta_n-\theta_\star
=\sum_{k=n+1}^\infty\gamma_k(b(\theta_{k-1})-g_k(\theta_{k-1}))
+\sum_{k=n+1}^\infty\gamma_kb(\theta_{k-1}).
\end{equation}
Hence
\begin{equation}\label{eq:theta:n-theta*<}
\mathbb{E}\big[\|\theta_n-\theta_\star\|\mathbbm1_{\mathbb{C}_\delta}\big]
\leq\mathbb{E}\Big[\Big\|\sum_{k=n+1}^\infty\gamma_k(b(\theta_{k-1})-g_k(\theta_{k-1}))\Big\|\Big]
+\sum_{k=n+1}^\infty\gamma_k\mathbb{E}\big[\|b(\theta_{k-1})\|\mathbbm1_{\mathbb{C}_\delta}\big].
\end{equation}
Note that, by monotone convergence, Jensen's inequality, Doob's maximal inequality (or the Burkholder-Davis-Gundy inequality) and \eqref{eq:E:sup:grad:f:2},
\begin{equation}\label{eq:term1<}
\begin{aligned}
\mathbb{E}\Big[&\Big\|\sum_{k=n+1}^\infty\gamma_k(b(\theta_{k-1})-g_k(\theta_{k-1}))\Big\|\Big]
\leq\mathbb{E}\Big[\sup_{p\geq n+1}\Big\|\sum_{k=n+1}^p\gamma_k(b(\theta_{k-1})-g_k(\theta_{k-1}))\Big\|\Big]\\
&=\lim_{N\to\infty}\mathbb{E}\Big[\sup_{n+1\leq p\leq N}\Big\|\sum_{k=n+1}^p\gamma_k(b(\theta_{k-1})-g_k(\theta_{k-1}))\Big\|^2\Big]^\frac12
\leq C\Big(\sum_{k=n}^\infty\gamma_k^2\Big)^\frac12
\leq\frac{C}{n^\frac{2s-1}2}.
\end{aligned}
\end{equation}
\newline
\noindent{\emph{$\stackrel{\blacktriangleright}{}$ Step~2. Alternative treatment for $\big(\sum_{k=n+1}^\infty\gamma_k\|\nabla F(\theta_{k-1})\|)_{n\geq1}$.}}\newline
Let $s\in\big[\frac\rho{1+\rho},1]$, $\rho>\frac1\alpha$, and $\sigma\in(0,1)$.
Assume
\begin{equation*}
\kappa\gamma_\star>\begin{cases}
(1\vee\zeta_\delta^{1/\beta_\delta})(\frac12\vee\alpha)
&\text{if $s=1$ and $\beta_\delta\in(0,\frac12]$,}\\
\frac12\vee\alpha
&\text{if $s=1$ and $\beta_\delta\in(\frac12,1)$,}
\end{cases}
\end{equation*}

Define
\begin{equation*}
v_n\coloneqq\frac1{n^\frac{2s-1}2}+\frac1{n^{(1+\alpha)s-1}},
\quad n\geq1.
\end{equation*}
Let $n>n_\delta$ and $\sigma\in(0,1)$. From \eqref{eq:G<}, one gets
\begin{equation*}
0<(G_{n-1}+u_{n-1}+v_{n-1})^\sigma
\leq C\sqrt{\mathbb{E}\big[\|\nabla F(\theta_{n-1})\|^2\mathbbm1_{\mathbb{C}_\delta}\big]+u_n^{2\sigma}+v_n^{2\sigma}}.
\end{equation*}
Thus, by Jensen's inequality,
\begin{equation}\label{eq:gamma:gradF<ratio}
\gamma_n\mathbb{E}\big[\|\nabla F(\theta_{n-1})\|\mathbbm1_{\mathbb{C}_\delta}\big]
\leq C\frac{\gamma_n(\mathbb{E}[\|\nabla F(\theta_{n-1})\|^2\mathbbm1_{\mathbb{C}_\delta}]+u_n^{2\sigma}+v_n^{2\sigma})}{(G_{n-1}+u_{n-1}+v_{n-1})^\sigma}.
\end{equation}
Observe that, by a first order Taylor-Lagrange expansion, for $\varrho>0$,
\begin{equation}\label{eq:equiv}
-\Delta\Big(\frac1{n^\varrho}\Big)\sim\frac\varrho{n^{1+\varrho}}
\quad\text{and}\quad
-\Delta\Big(\frac1{(\ln{n})^\varrho}\Big)\sim\frac\varrho{n(\ln{n})^{1+\varrho}}
\quad\text{as}\quad n\to\infty.
\end{equation}
Hence, taking $\rho>3\vee\frac2\alpha$ so that $r_\delta>1$ if $\beta_\delta\in\big(\frac12,1\big)$ and $s\in\big(\frac34\vee\frac2{2+\alpha},1\big]$, and setting
\begin{equation*}
\sigma\in\begin{cases}
\big[\frac1{2(2s-1)}\vee\frac{\alpha s}{2((1+\alpha)s-1)}\vee\frac{1-((1-\alpha)\vee\frac12)s}{2(\alpha\wedge\frac12)s},1\big)
&\text{if $\beta_\delta\in(0,\frac12]$,}\\
\big[\frac1{2(2s-1)}\vee\frac{\alpha s}{2((1+\alpha)s-1)}\vee\frac{1+r_\delta}{2r_\delta},1\big)
&\text{if $\beta_\delta\in(\frac12,1)$ and $s<1$,}\\
[\frac{1+r_\delta}{2r_\delta},1)
&\text{if $\beta_\delta\in(\frac12,1)$ and $s=1$,}
\end{cases}
\end{equation*}
one obtains, via \eqref{eq:gamma:gradF<ratio}, \eqref{eq:grad<Delta}, \eqref{eq:equiv} and Lemma~\ref{lmm:concave},
\begin{equation*}
\gamma_n\mathbb{E}\big[\|\nabla F(\theta_{n-1})\|\mathbbm1_{\mathbb{C}_\delta}\big]
\leq C_\delta\frac{-\Delta(G_n+u_n+v_n)}{(G_{n-1}+u_{n-1}+v_{n-1})^\sigma}
\leq-\frac{C_\delta}{1-\sigma}\Delta(G_n+u_n+v_n)^{1-\sigma}.
\end{equation*}
Thus, using that $G_n+u_n+v_n\to0$ as $n\to\infty$,
\begin{equation}\label{eq:term2<}
\sum_{k=n+1}^\infty\gamma_k\mathbb{E}\big[\|\nabla F(\theta_{k-1})\|\mathbbm1_{\mathbb{C}_\delta}\big]
\leq\frac{C_\delta}{1-\sigma}(u_n^{1-\sigma}+v_n^{1-\sigma}).
\end{equation}
\newline
\noindent{\emph{$\stackrel{\blacktriangleright}{}$ Step~3. Conclusion.}}\newline
All in all, recalling $\mathcal{H}$\ref{asp:f:abc}(\ref{asp:f:abc:iii}), gathering \eqref{eq:theta:n-theta*<}, \eqref{eq:term1<} and \eqref{eq:term2<} and applying similar reasoning to Steps~\hyperref[step:1.2]{1.2} and~\hyperref[step:2.2]{2.2} of the proof of Theorem~\ref{thm:F-F*:special:cases}(\ref{thm:F-F*:special:cases:i})\hyperref[thm:F-F*:special:cases:i:a]{a} yield the sought result.
\\

\noindent{\bf(\ref{crl:iterates:speed:i})\hyperref[crl:iterates:speed:i:b]{b}, (\ref{crl:iterates:speed:ii})}\
The proofs are identical to the anteceding one, by applying suitable substitutions.
\end{proof}

\begin{proof}[Proof of Corollary~\ref{crl:iter:cost}]
This is a straightforward consequence of Corollary~\ref{crl:iterates:speed}.
\end{proof}

\section*{Conclusion}
\label{sec:conclusion}

For $\alpha$-H\"older-differentiable loss landscapes satisfying a local {\L}ojasiewicz condition with exponent $\beta\in(0,1)$, our analysis shows that, for a learning schedule $(\gamma_n)_{n\geq1}$ such that $\sum_{n=1}^\infty\gamma_n=\infty$ and $\lim_{n\to\infty}\gamma_n=0$, \eqref{eq:theta:n} converges in high probability in $\mathrm{O}(\gamma_n^{\alpha\wedge\frac12})$ if $\beta\in\big(0,\frac12\big]$ and in $\mathrm{O}\big(\big(\sum_{k=1}^n\gamma_k\big)^{-\frac1{2\beta-1}\wedge(\alpha\rho-1)\wedge\frac{\rho-1}2}\big)$ if $\beta\in\big(\frac12,1\big)$, for some adjustable parameter $\rho>\frac1\alpha$ such that $\gamma_n=\mathrm{O}\big(\big(\sum_{k=1}^n\gamma_k\big)^{-\rho}\big)$.
These convergence rates are free from the requirement that the \eqref{eq:theta:n} trajectories be initialized or remain close to their target.
Furthermore, they help uncover optimal iteration amounts to employ in order to help cut down significantly the fine-tuning phase.

Future lines of research could involve extending our results to confined \eqref{eq:theta:n} trajectories around local minima or locally Lipschitz loss landscapes presenting singularities. An outlook on eliciting the dependency of the {\L}ojasiewicz exponent upon a neural network's architecture could also be envisaged.

\paragraph{Acknowledgements.}
The author would like to thank Guillaume Garrigos, St\'ephane Cr\'epey and Noufel Frikha for discussions on SGD and the {\L}ojasiewicz property.

\appendix

\section{Ancillary Results}
\label{apx:result}

The next result is given in~\cite[Lemma~1]{Lei+20} in a Hilbert space setting, but we provide a different take on its second part.
The first point is often referred to as the descent lemma.

\begin{lemma}\label{lmm:Holder}
Let $\varphi\colon\mathbb{R}^m\to\mathbb{R}$ be a differentiable function such that its derivative $\nabla\varphi$ is $(L,\alpha)$-H\"older continuous, with $L>0$ and $\alpha\in(0,1]$. Then,
\begin{enumerate}[\bf(i)]
\item\label{lmm:Holder:i}
one has
\begin{equation*}
\big|\varphi(\theta')-\varphi(\theta)-\nabla\varphi(\theta)^\top(\theta'-\theta)\big|\leq\frac{L}{1+\alpha}\|\theta'-\theta\|^{1+\alpha},
\quad\theta,\theta'\in\mathbb{R}^m.
\end{equation*}
\item\label{lmm:Holder:ii}
if $\varphi$ is additionally lower bounded,
\begin{equation*}
\|\nabla\varphi(\theta)\|^\frac{1+\alpha}\alpha\leq\frac{(1+\alpha)L^\frac1\alpha}\alpha\big(\varphi(\theta)-\inf\varphi\big),
\quad\theta\in\mathbb{R}^m.
\end{equation*}
\end{enumerate}
\end{lemma}

\begin{proof}
\noindent{\bf(\ref{lmm:Holder:i})}\
For $\theta,\theta'\in\mathbb{R}^m$, start with the first order Taylor-Lagrange expansion
\begin{equation*}
\varphi(\theta')=\varphi(\theta)+\nabla\varphi(\theta)^\top(\theta'-\theta)+\int_0^1(\theta'-\theta)^\top\big(\nabla\varphi(\theta+t(\theta'-\theta))-\nabla\varphi(\theta)\big)\mathrm{d}t,
\end{equation*}
then note that, since $\nabla\varphi$ is $(L,\alpha)$-H\"older,
\begin{equation*}
\bigg|\int_0^1(\theta'-\theta)^\top\big(\nabla\varphi(\theta+t(\theta'-\theta))-\nabla\varphi(\theta)\big)\mathrm{d}t\bigg|
\leq\frac{L}{1+\alpha}\|\theta'-\theta\|^{1+\alpha}.
\end{equation*}
\newline
\noindent{\bf(\ref{lmm:Holder:ii})}\
Let $\theta\in\mathbb{R}^m$. For $t\geq0$, by applying the previous result with $\theta'=\theta-t\nabla\varphi(\theta)$,
\begin{equation*}
0
\leq\varphi(\theta-t\nabla\varphi(\theta))-\inf\varphi
\leq\varphi(\theta)-\inf\varphi-t\|\nabla\varphi(\theta)\|^2+\frac{Lt^{1+\alpha}}{1+\alpha}\|\nabla\varphi(\theta)\|^{1+\alpha},
\end{equation*}
so that
\begin{equation*}
t\|\nabla\varphi(\theta)\|^2-\frac{Lt^{1+\alpha}}{1+\alpha}\|\nabla\varphi(\theta)\|^{1+\alpha}
\leq\varphi(\theta)-\inf\varphi.
\end{equation*}
Maximizing the left hand side with respect to $t$ leads to the optimal choice
\begin{equation*}
t_\star=\frac1{L^\frac1\alpha}\|\nabla\varphi(\theta)\|^\frac{1-\alpha}\alpha.
\end{equation*}
\end{proof}

The next technical result is invoked but not proven in~\cite[Theorem~2(c)]{Lei+20}.

\begin{lemma}\label{lmm:seq:jump}
Let $(u_n)_{n\geq0}$ be a real sequence and $\ell^-,\ell^+\in\mathbb{R}$ such that
\begin{equation*}
-\infty\leq\liminf_{n\to\infty}u_n<\ell^-<\ell^+<\limsup_{n\to\infty}u_n\leq\infty.
\end{equation*}
Then, there exist extractions $\chi,\psi\colon\mathbb{N}\to\mathbb{N}$ such that, for all $n\geq0$,
\begin{equation*}
u_{\chi(n)}<\ell^-,
\quad
u_{\psi(n)}>\ell^+
\quad\text{and}\quad
\ell^-\leq u_k\leq\ell^+,
\quad\chi(n)<k<\psi(n),
\end{equation*}
where, potentially, $\psi(n)=\chi(n)+1$.
\end{lemma}

\begin{proof}
\noindent
Since $\liminf_{n\to\infty}u_n<\ell^-$, then $u_n<\ell^-$ infinitely often.
Let $\mathcal{X}\coloneqq\{n\colon u_n<\ell^-\}$.
Thus $\#\mathcal{X}=\infty$, where $\#$ denotes cardinality.

Define $n'=\inf\{p>n\colon u_p>\ell^+\}$, $n\in\mathcal{X}$, with the convention $\inf{\varnothing}=\infty$.
Since $\limsup_{n\to\infty}u_n>\ell^+$, $u_n>\ell^+$ infinitely often so that $n'<\infty$, $n\in\mathcal{X}$.
Let $\mathcal{Y}=\{n'\colon n\in\mathcal{X}\}$. Note that $\#\mathcal{Y}=\infty$.

Next, let $n''\coloneqq\sup\{p<n'\colon u_n<\ell^-\}$, $n\in\mathcal{X}$, with the convention $\sup{\varnothing}={-\infty}$. Observe that, for $n\in\mathcal{X}$, it holds that $n''\in\mathcal{X}$, that $n''\geq n>-\infty$ and that $(n'')'=n'$. Let $\mathcal{Z}\coloneqq\{n''\colon n\in\mathcal{X}\}$.
Then $\#\mathcal{Z}=\infty$.
Moreover, for $n\in\mathcal{X}$, either $n'=n''+1$ or $\ell^-\leq u_k\leq\ell^+$, $n''<k<n'$.
Otherwise, there exists $n\in\mathcal{X}$ such that $n'>n''+1$ and there exists $n''<k<n'$ such that, either $u_k<\ell^-$, or $u_k>\ell^+$. The former possibility contradicts the definition of $n''$, and the latter contradicts that $(n'')'=n'$.

Finally, since $\mathcal{Z}$ and $\mathcal{Y}$ are ordered, it suffices to take $\chi$ and $\psi$  the ordered enumerations of $\mathcal{Z}$ and $\mathcal{Y}$ respectively.
\end{proof}

\begin{lemma}\label{lmm:beta+:zeta1}
Let $\varphi\colon\mathbb{R}^m\to\mathbb{R}$ be a continuously differentiable function and consider $\theta_\star\in\nabla\varphi^{-1}(0)$. Assume that there exist $\mathcal{V}$ an open neighborhood of $\theta_\star$, $\beta\in(0,1)$ and $\zeta>0$ such that
\begin{equation*}
|\varphi(\theta)-\varphi(\theta_\star)|^\beta\leq\zeta\|\nabla\varphi(\theta)\|,
\quad\theta\in\mathcal{V}.
\end{equation*}
Set $\beta'\in(\beta,1]$.
Then, there exists an open neighborhood $\mathcal{V}'$ of $\theta_\star$ such that
\begin{equation*}
|\varphi(\theta)-\varphi(\theta_\star)|^{\beta'}\leq\|\nabla\varphi(\theta)\|,
\quad\theta\in\mathcal{V}'.
\end{equation*}
\end{lemma}

\begin{proof}
\noindent
Note that
\begin{equation*}
\frac{|\varphi(\theta)-\varphi(\theta_\star)|^{\beta'}}{\|\nabla\varphi(\theta)\|}
\leq\zeta|\varphi(\theta)-\varphi(\theta_\star)|^{\beta'-\beta},
\quad\theta\in\mathcal{V}\setminus\nabla\varphi^{-1}(0).
\end{equation*}
By the continuity of $\varphi$, the right hand side above vanishes when $\mathcal{V}\ni\theta\to\theta_\star$.
There then exists an open neighborhood $\mathcal{V}'\subset\mathcal{V}$ of $\theta_\star$ such that $\zeta|\varphi(\theta)-\varphi(\theta_\star)|^{\beta'-\beta}\leq1$, $\theta\in\mathcal{V}'$.
Thus
\begin{equation}\label{eq:theta:in:Crit} 
|\varphi(\theta)-\varphi(\theta_\star)|^{\beta'}
\leq\|\nabla\varphi(\theta)\|,
\quad\theta\in\mathcal{V}'\setminus\nabla\varphi^{-1}(0).
\end{equation}
Since $\mathcal{V}'\subset\mathcal{V}$,
\begin{equation}\label{eq:theta:out:Crit}
|\varphi(\theta)-\varphi(\theta_\star)|^{\beta'}
\leq\zeta\,\|\nabla\varphi(\theta)\|\;|\varphi(\theta)-\varphi(\theta_\star)|^{\beta'-\beta}
=0,
\quad\theta\in\mathcal{V}'\cap\nabla\varphi^{-1}(0).
\end{equation}
Combining \eqref{eq:theta:in:Crit} and \eqref{eq:theta:out:Crit},
\begin{equation*}
|\varphi(\theta)-\varphi(\theta_\star)|^{\beta'}
\leq\|\nabla\varphi(\theta)\|,
\quad\theta\in\mathcal{V}'.
\end{equation*}
\end{proof}

The following result is a classic in stochastic approximation literature.

\begin{lemma}\label{lmm:rec}
Let $(u_n)_{n\geq0}$ and $(v_n)_{n\geq0}$ be real-valued sequences and $(\gamma_n)_{n\geq0}$ a nonnegative sequence such that
\begin{equation*}
u_n\leq(1-\gamma_n)u_{n-1}+v_n,
\quad n\geq0.
\end{equation*}
If $\gamma_n<1$ beginning from a certain rank $p\geq1$, then
\begin{equation*}
u_n\leq u_p\prod_{k=p+1}^n(1-\gamma_k)+\sum_{k=p+1}^nv_k\prod_{j=k+1}^n(1-\gamma_j),\quad n>p,
\end{equation*}
with the convention $\prod_\varnothing=1$.
\end{lemma}

\begin{proof}
\noindent
The assertion follows by recursion.
\end{proof}

The first part of the below lemma is a special case of~\cite[Lemmas~5.8 \&~5.9]{For15}.
Due to a minor mistake in the proof therein and for the sake of completeness, it is reproven below.
The second part of the lemma is novel.

\begin{lemma}\label{lmm:gamma}
Let $\mu>0$, $\alpha\geq0$ and $(\gamma_n)_{n\geq1}$ a positive sequence such that
\begin{equation}\label{eq:gamma}
\sum_{n=1}^\infty\gamma_n=\infty
\quad\text{and}\quad
\lim_{n\to\infty}\gamma_n=0.
\end{equation}
\begin{enumerate}[\bf(i)]
\item\label{lmm:gamma:1}
Assume that either of the following is satisfied:
\begin{enumerate}[\bf C1.]
\item\label{lmm:asp:gamma:misc:i}
$\ln{\big(\frac{\gamma_{n-1}}{\gamma_n}\big)}=\mathrm{o}(\gamma_n)$;
\item\label{lmm:asp:gamma:misc:ii}
there exists $\gamma_\star>\frac\alpha\mu$ such that $\ln{\big(\frac{\gamma_{n-1}}{\gamma_n}\big)}\sim\frac{\gamma_n}{\gamma_\star}$.
\end{enumerate}
Then, using the convention $\sum_\varnothing=0$,
    \begin{enumerate}[\bf a.]
    \item\label{lmm:gamma:1:i}
    $\limsup_{n\to\infty}\gamma_n^{-\alpha}\sum_{k=1}^n\gamma_k^{1+\alpha}\exp{\big(-\mu\sum_{j=k+1}^n\gamma_j\big)}<\infty$;
    \item\label{lmm:gamma:1:ii}
    $\limsup_{n\to\infty}\gamma_n^{-\alpha}\exp{\big(-\mu\sum_{k=1}^n\gamma_k\big)}=0$.
    \end{enumerate}
\item\label{lmm:gamma:2}
Let $0\leq\beta<\mu$.
Suppose that, in addition,
\begin{equation*}
\sum_{n=1}^\infty\frac{\gamma_n}{\sum_{k=1}^n\gamma_k}=\infty.
\end{equation*}
Then, with the convention that $\sum_\varnothing=0$,
    \begin{enumerate}[\bf a.]
    \item\label{lmm:gamma:2:i}
    $\limsup_{n\to\infty}\big(\sum_{k=1}^n\gamma_k\big)^\beta\sum_{k=1}^n\frac{\gamma_k}{(\sum_{j=1}^k\gamma_j)^{1+\beta}}\exp\Big(-\mu\sum_{j=k+1}^n\frac{\gamma_j}{\sum_{i=1}^j\gamma_i}\Big)<\infty$;
    \item\label{lmm:gamma:2:ii}
    $\limsup_{n\to\infty}\big(\sum_{k=1}^n\gamma_k\big)^\beta\exp\Big(-\mu\sum_{k=1}^n\frac{\gamma_k}{\sum_{j=1}^k\gamma_j}\Big)=0$.
    \end{enumerate}
\end{enumerate}
\end{lemma}

\begin{proof}
\noindent{\bf(\ref{lmm:gamma:1})\hyperref[lmm:gamma:1:i]{a}.}\
Define
\begin{equation*}
a_n=\gamma_n^{-\alpha}\sum_{k=1}^n\gamma_k^{1+\alpha}\exp{\Big(-\mu\sum_{j=k+1}^n\gamma_j\Big)},
\quad n\geq1.
\end{equation*}
Then, for $n\geq2$,
\begin{equation*}
a_n
=\Big(\frac{\gamma_{n-1}}{\gamma_n}\Big)^\alpha\exp(-\mu\gamma_n)a_{n-1}+\gamma_n
=(1-b_n)a_{n-1}+b_nb_n^{-1}\gamma_n,
\end{equation*}
where
\begin{gather}
b_n
=1-\exp\Big(-\mu\gamma_n+\alpha\ln{\frac{\gamma_{n-1}}{\gamma_n}}\Big)
\sim\begin{cases}
\mu\gamma_n,
&\text{case \hyperref[lmm:asp:gamma:misc:i]{C1},}\\
(\mu-\frac\alpha{\gamma_\star})\gamma_n,
&\text{case \hyperref[lmm:asp:gamma:misc:ii]{C2},}
\end{cases}
\label{eq:bn}
\\
b_n^{-1}\gamma_n\sim c_\star\coloneqq\begin{cases}
\frac1\mu,
&\text{case \hyperref[lmm:asp:gamma:misc:i]{C1},}\\
\frac1{\mu-\alpha/\gamma_\star},
&\text{case \hyperref[lmm:asp:gamma:misc:ii]{C2}.}
\end{cases}\notag
\end{gather}
Let $\varepsilon>0$ and $c=c_\star+\varepsilon$. Then, there exists $n_0\geq0$ such that $b_n^{-1}\gamma_n\leq c$, $n\geq n_0$.
Via the convexity of $x\in\mathbb{R}\mapsto x^+$ and  Jensen's inequality, for $n\geq n_0$,
\begin{equation*}
(a_n-c)^+
\leq(1-b_n)(a_{n-1}-c)^++b_n(b_n^{-1}\gamma_n-c)^+
=(1-b_n)(a_{n-1}-c)^+.
\end{equation*}
Considering \eqref{eq:bn}, one has $\sum_{n=1}^\infty b_n=\infty$, so that $\lim_{n\to\infty}(a_n-c)^+=0$.
Hence $\limsup_{n\to\infty}a_n\leq c=c_\star+\varepsilon$ for any $\varepsilon>0$, thus $\limsup_{n\to\infty}a_n\leq c_\star<\infty$.
\\

\noindent{\bf(\ref{lmm:gamma:1})\hyperref[lmm:gamma:1:ii]{b}.}\
Let
\begin{equation*}
a_n=\gamma_n^{-\alpha}\exp\Big(-\mu\sum_{k=1}^n\gamma_k\Big),
\quad n\geq1.
\end{equation*}
For $n\geq2$,
\begin{equation*}
a_n
=\Big(\frac{\gamma_{n-1}}{\gamma_n}\Big)^\alpha\exp(-\mu\gamma_n)a_{n-1}
=(1-b_n)a_{n-1},
\end{equation*}
with
\begin{equation*}
b_n
=1-\exp\Big(-\mu\gamma_n+\alpha\ln{\frac{\gamma_{n-1}}{\gamma_n}}\Big)
\sim\begin{cases}
\mu\gamma_n,
&\text{case \hyperref[lmm:asp:gamma:misc:i]{C1},}\\
(\mu-\frac\alpha{\gamma_\star})\gamma_n,
&\text{case \hyperref[lmm:asp:gamma:misc:ii]{C2}.}
\end{cases}
\end{equation*}
Thus $\sum_{n=1}^\infty b_n=\infty$ and $\lim_{n\to\infty}a_n=0$.
\\

\noindent{\bf(\ref{lmm:gamma:2})\hyperref[lmm:gamma:2:i]{a}.}\
Define
\begin{equation*}
a_n=\Big(\sum_{k=1}^n\gamma_k\Big)^\beta\sum_{k=1}^n\frac{\gamma_n}{\big(\sum_{j=1}^k\gamma_j\big)^{1+\beta}}\exp\Big(-\mu\sum_{j=k+1}^n\frac{\gamma_j}{\sum_{i=1}^j\gamma_i}\Big),
\quad n\geq1.
\end{equation*}
Let $n\geq2$. Then
\begin{equation*}
a_n
=\bigg(\frac{\sum_{k=1}^n\gamma_k}{\sum_{k=1}^{n-1}\gamma_k}\bigg)^\beta\exp\Big(-\frac{\mu\gamma_n}{\sum_{k=1}^n\gamma_k}\Big)a_{n-1}+\frac{\gamma_n}{\sum_{k=1}^n\gamma_k}
=(1-b_n)a_{n-1}+b_nb_n^{-1}\frac{\gamma_n}{\sum_{k=1}^n\gamma_k}.
\end{equation*}
Note that
\begin{gather}
b_n
=1-\exp\bigg(-\frac{\mu\gamma_n}{\sum_{k=1}^n\gamma_k}+\beta\ln{\frac{\sum_{k=1}^n\gamma_k}{\sum_{k=1}^{n-1}\gamma_k}}\bigg)
\sim\frac{(\mu-\beta)\gamma_n}{\sum_{k=1}^n\gamma_k},\label{eq:bn:equiv}\\
b_n^{-1}\frac{\gamma_n}{\sum_{k=1}^n\gamma_k}\sim\frac1{\mu-\beta}\eqqcolon c_\star.\notag
\end{gather}
Let $\varepsilon>0$ and $c=c_\star+\varepsilon$. Then, there is an $n_0\geq0$ such that $b_n^{-1}\frac{\gamma_n}{\sum_{k=1}^n\gamma_k}\leq c$, $n\geq n_0$.
Hence, for $n\geq n_0$,
\begin{equation*}
(a_n-c)^+
\leq(1-b_n)(a_{n-1}-c)^++b_n\Big(b_n^{-1}\frac{\gamma_n}{\sum_{k=1}^n\gamma_k}-c\Big)^+
=(1-b_n)(a_{n-1}-c)^+.
\end{equation*}
By \eqref{eq:bn:equiv}, $\sum_{n=1}^\infty b_n=\infty$, thus $\lim_{n\to\infty}(a_n-c)^+=0$.
Therefore $\limsup_{n\to\infty}a_n\leq c_\star+\varepsilon$ for all $\varepsilon>0$, whence $\limsup_{n\to\infty}a_n\leq c_\star<\infty$.
\\

\noindent{\bf(\ref{lmm:gamma:2})\hyperref[lmm:gamma:2:ii]{b}.}\
Let
\begin{equation*}
a_n=\Big(\sum_{k=1}^n\gamma_k\Big)^\beta\exp\Big(-\mu\sum_{k=1}^n\frac{\gamma_k}{\sum_{j=1}^k\gamma_j}\Big),
\quad n\geq1.
\end{equation*}
For $n\geq2$,
\begin{equation*}
a_n
=\bigg(\frac{\sum_{k=1}^n\gamma_k}{\sum_{k=1}^{n-1}\gamma_k}\bigg)^\beta\exp\Big(-\frac{\mu\gamma_n}{\sum_{k=1}^n\gamma_k}\Big)a_{n-1}
=(1-b_n)a_{n-1},
\end{equation*}
where $b_n$ satisfies \eqref{eq:bn:equiv}.
Thus $\sum_{n=1}^\infty b_n=\infty$ and $\lim_{n\to\infty}a_n=0$.
\end{proof}

The concavity inequality below is retrieved in~\cite[Equation (8)]{AB09} and~\cite[Equation (33)]{CF23} by applying the tangent line method to the convex functions $x\in\mathbb{R}_+\mapsto-x^{1-\sigma}$ and $x\in\mathbb{R}_+\mapsto x^\frac1{1-\sigma}$, $\sigma\in(0,1)$.
We rely here instead on a first order Taylor-Lagrange expansion of $x\in\mathbb{R}_+\mapsto x^{1-\sigma}$, and complement our result with an assertion for $\sigma=1$.

\begin{lemma}\label{lmm:concave}
Let $0\leq x\leq y$. Then, for $\sigma\in(0,1)$,
\begin{equation*}
\frac{y-x}{y^\sigma}\leq\frac{y^{1-\sigma}-x^{1-\sigma}}{1-\sigma}.
\end{equation*}
Moreover, if $x>0$,
\begin{equation*}
\frac{y-x}y\leq\ln{y}-\ln{x}.
\end{equation*}
\end{lemma}

\begin{proof}
\noindent
Suppose $x<y$. Then, for $\sigma\in(0,1)$,
\begin{equation*}
y^{1-\sigma}-x^{1-\sigma}
=(y-x)\int_0^1\frac{(1-\sigma)\mathrm{d}t}{(x+t(y-x))^\sigma}\geq(1-\sigma) y^{-\sigma}(y-x).
\end{equation*}
Assume that, in addition, $0<x$. Then
\begin{equation*}
\ln{y}-\ln{x}
=(y-x)\int_0^1\frac{\mathrm{d}t}{x+t(y-x)}\geq y^{-1}(y-x).
\end{equation*}
\end{proof}

\end{document}